\documentclass[11pt,a4paper,reqno]{amsart}

\usepackage{amsfonts,amsthm,amscd,epsfig,amsmath,amssymb,enumerate}

\author{P. Caputo}
\address{Dipartimento di Matematica, Universit\`a Roma Tre,
Largo S.\ Murialdo 1, 00146 Roma, Italy.}
\email{caputo@mat.uniroma3.it}

\author {A. Faggionato}
\address{Dipartimento di Matematica ``G. Castelnuovo", Universit\`a   ``La
  Sapienza'', P.le Aldo Moro  2, 00185  Roma, Italy.}
\email{faggiona@mat.uniroma1.it}

\author{T. Prescott}
\address{Department of Mathematics, University of North Dakota, 101 Cornell St Stop 8376, Grand Forks, ND 58202-8376, United States.}
\email{timothy.prescott@und.edu}

\numberwithin{equation}{section}

\DeclareMathSymbol{\leqslant}{\mathalpha}{AMSa}{"36} 
\DeclareMathSymbol{\geqslant}{\mathalpha}{AMSa}{"3E} 
\DeclareMathSymbol{\eset}{\mathalpha}{AMSb}{"3F}     
\renewcommand{\leq}{\;\leqslant\;}                   
\renewcommand{\geq}{\;\geqslant\;}                   
\newcommand{\dd}{\,\mathrm{d}}               
\newcommand{\sumtwo}[2]{\sum_{\substack{#1 \\ #2}}} 

\usepackage{graphicx}
\usepackage{url}

\newtheorem{theorem}{Theorem}

\newtheorem{corollary}[theorem]{Corollary}
\newtheorem{definition}[theorem]{Definition}

\newtheorem{lemma}[theorem]{Lemma}
\newtheorem{proposition}[theorem]{Proposition}

\newtheorem{theorem*}{Theorem}

\numberwithin{theorem}{section}
\newcommand{\ignore}[1]{}

\DeclareMathOperator{\Div}{div}

\newcommand{\Exp}[1]{{\exp\left\{ #1 \right\}}}

\newcommand{\coloneq}{\mathrel{\mathop:}=}

\newcommand{\Cvol}{C_{\mathrm{vol}}}

\newcommand{\la}{\label}

\newcommand{\be}{\begin{equation}}

\def\1{\ifmmode {1\hskip -3pt \rm{I}} \else {\hbox {$1\hskip -3pt \rm{I}$}}\fi}

\newcommand{\cA}{\ensuremath{\mathcal A}}
\newcommand{\cB}{\ensuremath{\mathcal B}}
\newcommand{\cC}{\ensuremath{\mathcal C}}
\newcommand{\cD}{\ensuremath{\mathcal D}}

\newcommand{\cF}{\ensuremath{\mathcal F}}
\newcommand{\cG}{\ensuremath{\mathcal G}}
\newcommand{\cH}{\ensuremath{\mathcal H}}

\newcommand{\cL}{\ensuremath{\mathcal L}}

\newcommand{\cN}{\ensuremath{\mathcal N}}

\newcommand{\cR}{\ensuremath{\mathcal R}}
\newcommand{\cS}{\ensuremath{\mathcal S}}
\newcommand{\cT}{\ensuremath{\mathcal T}}

\newcommand{\cW}{\ensuremath{\mathcal W}}

\newcommand{\bbB}{{\ensuremath{\mathbb B}} }

\newcommand{\bbE}{{\ensuremath{\mathbb E}} }

\newcommand{\bbG}{{\ensuremath{\mathbb G}} }

\newcommand{\bbI}{{\ensuremath{\mathbb I}} }

\newcommand{\bbL}{{\ensuremath{\mathbb L}} }

\newcommand{\bbN}{{\ensuremath{\mathbb N}} }

\newcommand{\bbP}{{\ensuremath{\mathbb P}} }
\newcommand{\bbQ}{{\ensuremath{\mathbb Q}} }
\newcommand{\bbR}{{\ensuremath{\mathbb R}} }

\newcommand{\bbZ}{{\ensuremath{\mathbb Z}} }

\newcommand{\si}{\sigma}

\newcommand{\wt}{\widetilde}

\let\a=\alpha \let\b=\beta   \let\d=\delta  \let\e=\varepsilon
 \let\g=\gamma     \let\k=\kappa  \let\l=\lambda
      \let\o=\omega    \let\p=\pi  
\let\r=\rho  \let\s=\sigma \let\t=\tau   
  \let\z=\zeta
\let\D=\Delta   \let\G=\Gamma  \let\L=\Lambda 
\let\O=\Omega      

\def\\{\hfill\break}

\def\thsp{\thinspace}

\def\tthsp{\kern .083333 em}

\def\?{\mskip -10mu}
\def\indbox#1{\hbox to \parindent{\hfil\ #1\hfil} }

\newfam\msafam
\newfam\msbfam
\newfam\eufmfam
\def\hexnumber#1{%
  \ifcase#1 0\or 1\or 2\or 3\or 4\or 5\or 6\or 7\or 8\or
  9\or A\or B\or C\or D\or E\or F\fi}
\font\tenmsa=msam10
\font\sevenmsa=msam7
\font\fivemsa=msam5
\textfont\msafam=\tenmsa
\scriptfont\msafam=\sevenmsa
\scriptscriptfont\msafam=\fivemsa
\edef\msafamhexnumber{\hexnumber\msafam}%
\mathchardef\restriction"1\msafamhexnumber16
\mathchardef\ssim"0218
\mathchardef\square"0\msafamhexnumber03
\mathchardef\eqd"3\msafamhexnumber2C
\def\QED{\ifhmode\unskip\nobreak\fi\quad
  \ifmmode\square\else$\square$\fi}
\font\tenmsb=msbm10
\font\sevenmsb=msbm7
\font\fivemsb=msbm5
\textfont\msbfam=\tenmsb
\scriptfont\msbfam=\sevenmsb
\scriptscriptfont\msbfam=\fivemsb
\def\Bbb#1{\fam\msbfam\relax#1}
\font\teneufm=eufm10
\font\seveneufm=eufm7
\font\fiveeufm=eufm5
\textfont\eufmfam=\teneufm
\scriptfont\eufmfam=\seveneufm
\scriptscriptfont\eufmfam=\fiveeufm

\def\({\left(}
\def\){\right)}
\let\integer=\bZ

\def\ZZ{{\integer^2}}

\let\neper=e
\let\ii=i

\def\nep#1{ \neper^{#1}}

\def\tc{\thsp | \thsp}
\let\<=\langle
\let\>=\rangle

\def\diam{\mathop{\rm diam}\nolimits}

\outer\def\nproclaim#1 [#2]#3. #4\par{\medbreak \noindent
   \talato(#2){\bf #1 \Thm[#2]#3.\enspace }%
   {\sl #4\par }\ifdim \lastskip <\medskipamount
   \removelastskip \penalty 55\medskip \fi}
\def\thmm[#1]{#1}
\def\teo[#1]{#1}
\def\sttilde#1{%
\dimen2=\fontdimen5\textfont0
\setbox0=\hbox{$\mathchar"7E$}
\setbox1=\hbox{$\scriptstyle #1$}
\dimen0=\wd0
\dimen1=\wd1
\advance\dimen1 by -\dimen0
\divide\dimen1 by 2
\vbox{\offinterlineskip%
   \moveright\dimen1 \box0 \kern - \dimen2\box1}
}
\begin{document}

\title[
Invariance principle for  random walk on point processes]
{Invariance principle for Mott variable range hopping and other
walks on point processes }

\begin{abstract}
We consider a random walk on a homogeneous  Poisson point process
with energy marks. The jump rates decay exponentially in the
$\a$-power of the jump length and depend on the energy marks via a
Boltzmann--like factor. The case   $\a = 1$   corresponds to the
phonon-induced Mott variable range hopping in disordered solids in
the regime of strong Anderson localization. We prove that for almost
every realization of the marked process, the diffusively rescaled
random walk, with an arbitrary start point, converges to a Brownian
motion whose diffusion matrix is positive definite  and independent
of the environment.    Finally, we extend
 the above result  to  other point processes including diluted lattices.
%

\medskip


\medskip

\end{abstract}

\maketitle
\thispagestyle{empty}
\section{Introduction and results}
Random walks on random point processes such as Mott variable range hopping
have been proposed in the physics literature as
effective models
for the analysis of the conductivity of disordered systems; see e.g.\ \cite{AHL,SE}.
They provide natural models of reversible random walks in random environments, which generalize in
several ways the well known random conductance lattice model.
Recently, several aspects of random walks on random point processes have been analyzed
with mathematical rigor: diffusivity \cite{FSS,CF2,FM}; isoperimetry and mixing times \cite{CF1}; and transience vs.\ recurrence \cite{CFG}.

\subsection{The model}
%
Let  $\xi$  denote the realization of a simple point
process on $\bbR^d$, $d\geq 1$,  and identify $\xi$ with the
countable collection of its points.
For example, one can take $\xi$ to be 
a homogeneous Poisson point process, or 
a Bernoulli process on $\bbZ^d$.
To each  point $x$ of $\xi$ we
associate an \textsl{energy mark} $E_x$, such that the family of energy
marks is independent from the point process and is given by i.i.d.\
random variables taking values in the interval $[-1,1]$.
We write $\bbP$ for the law of the resulting marked simple point process
$\o=\bigl( \xi, \{ E_x\} _{x \in \xi} \bigr)$,
which plays the role of the random environment.
Then, we consider
the discrete--time random walk $(X_n,\;n\geq 0)$ on $\xi$ jumping,
at each time step, from a point $x$ to a  point $y$  with probability
\begin{equation}\label{proba1}
 p(x,y)= \frac{r(x,y)  e^{-u(E_x,E_y)}  }{w(x) }\,,
 \quad w(x)=\sum_{z\in\xi} r(x,z)e^{- u(E_x,E_z)}  \,,
\end{equation}
where the functions $u$ and $r$ satisfy the following properties
for some constants $c,\a>0$:
\begin{enumerate}[(i)]
\item
$u:[-1,1]^2\to \bbR_+$ is a bounded nonnegative symmetric function:
\begin{equation}
0\leq u(E_x,E_y)=u(E_y,E_x) \leq c\,,\la{paletti1}
\end{equation}

\item
 $r$ is symmetric and translation invariant, i.e.\ $r(x,y)=r(y,x)=r(y-x)$, and
\begin{equation}\label{paletti2}
 c^{-1}\, \exp{\(- c\, |x|^\a \)}
 \leq r(x) = r(-x)
 \leq c \, \exp{\(-c^{-1}\, |x|^\a\)}\,,   \qquad x\in\bbR^d\,,
\end{equation}
\end{enumerate}
Here and below $|\cdot|$ denotes euclidean distance.
For this model to be well defined it suffices to assume
that $w(x) < \infty$ for
all $x \in \xi$ and almost all realizations of the environment (see
Lemma \ref{arancia}).
Below, we write $X_t := X_{\lfloor t\rfloor}$, $t\geq 0$,
and consider the associated distribution on the
space $\cD=D([0,\infty),\bbR^d)$ of right-continuous paths with left
limits, equipped with the Skorohod topology.

Similarly, consider the continuous--time version
 of the above random walk, with state space $\xi$ and infinitesimal generator
\begin{equation}\label{generatore}
 \bbL f(x) = \sum _{y \in \xi } r(x,y) \,\nep{- u(E_x,E_y)} \bigl(
f(y)-f(x) \bigr)\,, \qquad x\in \xi\,,
\end{equation}
for bounded functions $f:\xi\to\bbR$.
With some abuse of notation, the resulting random process on $\cD$
is again denoted by
$(X_t \,:\, t \geq 0)$. To avoid confusion we shall refer to
the two processes as the DTRW (discrete-time random walk)
and the CTRW (continuous-time random walk).
In words, the CTRW behaves as follows:
having arrived at a point
$x\in \xi$, it waits an exponential time with parameter $w(x)$,
after which it jumps to a point $y\in \xi$ with probability $p(x,y)$.
In Lemma \ref{arancia} we give
some sufficient conditions ensuring that the CTRW is well
defined, i.e.\ no explosion takes place.

An important special case of the model introduced above is
\emph{Mott variable range hopping}, obtained by choosing
\begin{equation}\la{mottvrh}
r(v)=e^{-|v|}\,,
\quad u(E_x,E_y)= \b\,\(|E_x|+|E_y|+|E_y-E_x|\)
\end{equation}
where $\b$ is a positive constant proportional to the inverse temperature.
Here the underlying point process is  often taken as the homogeneous
Poisson process or the diluted lattice $\bbZ^d$, the common law
$\nu$ of the energy marks is assumed to be of the form $\nu (dE)=c
|E|^\g dE$  on $[-1,1]$ for some  constants $c>0$ and $\g\geq 0$,
and the relevant issue is the asymptotic behavior as $\b\to\infty$. 
Mott variable range hopping   is
 a mean field dynamics 
describing low temperature phonon--assisted electron transport in disordered solids, in
which the Fermi level (which is $0$ above) lies in a region of
strong Anderson localization. The points of $\xi$  correspond to the
impurities of the disordered solid and the electron Hamiltonian has
exponentially localized quantum eigenstates with localization
centers $x\in \xi$ and corresponding energy $E_x$. The rate of
transitions between the localized eigenstates can be calculated from
first principles by means of the Fermi golden rule \cite{MA,SE}. Due
to localization, one can approximate the above quantum system by an
exclusion process, where the hard--core interaction comes from the
Pauli blocking induced by the Fermi statistics. If, however, the
blocking is treated in a mean field approximation, one obtains a
family of independent random walks with rates described by
\eqref{mottvrh} in the limit $\b \to \infty$ \cite{MA,AHL}.  Mott's law represents a 
fundamental principle describing the decay of  the DC conductivity at
low temperature  \cite{Mo,SE}.  
In view of
Einstein's relation \cite{S}, this law can be restated in terms of
the   diffusivity of  Mott variable range hopping.

\subsection{Invariance principle}
When we need to emphasize the dependence on the environment $\o$
and the starting point $x_0$,
we write 
$X_t(x_0,\o)$ for the two processes defined above and $P_{x_0,\o}$ for
the associated laws on $\cD$.
Asymptotic diffusive behavior of both DTRW and CTRW is studied via the
rescaled process
\be\la{resca}
X^{(\e)}(t):= \e \,X_{t/\e^2}\,,
\end{equation}
and the associated laws $P^{(\e)}_{x_0,\o}$ on $\cD$.

\begin{definition}\la{invprin}
We say that the \emph{strong invariance principle} (SIP) holds if there
exists a positive definite $d\times d$ matrix $D$ such that $\bbP$
almost surely, for every $x_0\in\xi$, $P^{(\e)}_{x_0,\o}$ converges
weakly to $d$-dimensional Browninan motion with diffusion matrix $D$.
We say that the \emph{weak invariance principle} (WIP) holds if the above
convergence takes place in $\bbP$-probability.
\end{definition}
The terms \emph{quenched} and \emph{annealed} are sometimes used to
replace \emph{strong} and \emph{weak}, respectively, in the above definition.
Diffusive behavior of the CTRW
has been rigorously investigated in
\cite{FSS}. Under suitable assumptions on the law of the point process
the authors prove the WIP.
Moreover, \cite{FSS} proves lower bounds on the
diffusion coefficient $D$ in agreement with Mott's law for the
special case \eqref{mottvrh}, as $\b\to\infty$.
  %
The corresponding upper bound is proven
in \cite{FM}. In \cite{CF2}, the authors
consider the case $d=1$, where they obtain
the SIP, and analyze the large $\b$ behavior of various
generalizations of the model \eqref{mottvrh} at the edge of subdiffusivity.

\smallskip
The aim of the present work is to
prove the strong
invariance principle 
in dimension $d \geq 2$. 
To state our main result we need to introduce some more notation.
 We write $\xi(A)$ for the number of points of
$\xi$ in a bounded Borel set $A\subset\bbR^d$.
Let $\bbE$ denote the expectation associated to the law $\bbP$ of the
environment $\o$.
%
Set $\r_k=\bbE\(\xi([0,1)^d)^k\)$, so that $\r_1$ is the density and
$\r_2$ stands for the second moment of the point process. If $\xi$
is a stationary simple point process with finite density, then we
can consider the associated Palm distribution. If $\xi$ is a
homogeneous Poisson point process (from now on, PPP), then its Palm
distribution is simply the law of the point process obtained from
$\xi$ by adding a point at the origin. In general, if $\bbP$ is the
law of $\o=(\xi,\{E_x\})$, then we let $\bbP_{0}$ denote the
associated Palm distribution (see Section \ref{ambiente} for the
definition) and we write $\bbE_0$ for the expectation with respect
to $\bbP_0$. As explained in Lemma \ref{arancia} in the appendix, if
$\r_2<\infty$, then $\bbP$--a.s.\ the law $P_{x_0,\o}$ on $\cD$ is
well defined for both DTRW and CTRW, for all $x_0\in\xi$. Moreover,
under the same assumption, the law $P_{0,\o}$ on $\cD$ with starting
point $0$ is well defined $\bbP_{0}$--a.s.

Our main result applies to several examples of point processes.
These include homogeneous PPP, as well as Bernoulli fields on a
lattice, referred to as the \emph{diluted lattice} case below.
In Section \ref{ambiente.1} we describe
conditions on the point process, under which all our arguments apply.
Below we restrict to $d\geq 2$ since
the one dimensional case is already treated in \cite{CF1}.
\begin{theorem}\label{atreiu}
Let  $d\geq 2$, $\a>0$, and fix an arbitrary law $\nu$ on $[-1,1]$.
 Let $\xi$ be the realization of a homogeneous PPP, or a diluted lattice, or
else any
stationary simple point process with $\r_2<\infty$,
and satisfying the
assumptions listed in  Section \ref{ambiente.1}.
Then, the following holds for both the DTRW and the CTRW:
$\bbP_{0}$
almost surely, as $\e\to0$, $P^{(\e)}_{0,\o}$ converges
weakly to $d$-dimensional Brownian motion with positive
diffusion matrix $D_{DTRW}$ and $D_{CTRW}$ respectively. Moreover,
the diffusion coefficients are related by
\be\la{relaz}
D_{CTRW}=  \bbE_0 w(0)\, D_{DTRW} \,.
\end{equation}
\end{theorem}

\bigskip
The desired result is then a consequence of Theorem \ref{atreiu}
together with standard properties of the Palm distribution:
\begin{corollary}\la{strong}
Under the assumptions of Theorem \ref{atreiu}, the strong invariance principle holds for both DTRW and CTRW, with the same diffusion matrices appearing in Theorem \ref{atreiu}.
\end{corollary}
As a consequence of the above result, for the model \eqref{mottvrh} the quenched diffusion
matrix $D_{CTRW}$ satisfies stretched exponential estimates as $\b\to\infty$, in agreement with Mott's law.
This follows from the bounds of
 \cite{FSS} and \cite{FM} on the annealed diffusion matrix and the fact that the quenched and annealed
 diffusion matrices must coincide.

\subsection{Background and discussion}
To illustrate the kind of difficulties
encountered in the proof of Theorem \ref{atreiu},
let us briefly recall the standard approach
(see \cite{Kozlov,KV,DFGW} and references therein) for the invariance
principle in the case of reversible random walks in random
environment. The main idea
is to consider the environment as seen from the moving particle, and
to use this new Markov process to define a displacement field
$\chi(x)=\chi(\o,x)$ that compensates the local drift felt by the
random walk $X_t$ in such a way that the process
$M_t:=X_t+\chi(X_t)$ defines a martingale. The displacement field
$\chi$ is usually referred to as the \emph{corrector}. A strong
invariance principle for the martingale $M_t$ can be obtained in a
rather standard way, so that what remains is to show that the
corrector's contribution is negligible. In particular, one needs
that for every $t>0$: \be\la{neglig} \lim_{\e\to 0} \,\e\,\chi
(X_{t/\e^2},\o) = 0\,,\quad\;
\text{in\;$P_{x_0,\o}$-\,probability}\,.
\end{equation}
Roughly speaking, the $L^2$-theory developed in \cite{KV,DFGW} allows to obtain
the statement \eqref{neglig}
in probability with respect to the random environment. This
approach can then be used to prove the WIP, as
detailed in \cite{FSS}. Moreover, this approach provides
an expression for the limiting diffusion matrix in terms of a
variational principle.
However, to have the strong invariance principle,
\eqref{neglig} must hold almost surely with respect
to the environment. This turns out to be related to a highly
non trivial ergodic property of the field $\chi$.

The same difficulty appears in
analogous investigations for the random conductance model in $\bbZ^d$.
In this model,
one has i.i.d.\ non-negative weights $r(x,y)$ on the nearest neighbor
edges $\{x,y\}$ of $\bbZ^d$, so that the random walk with generator
\eqref{generatore} becomes
a reversible nearest neighbor lattice walk. When the weights $r(x,y)$
are uniformly positive and bounded, the SIP
for this model has been proved in \cite{SS}; see also \cite{Kozlov,Bo,BoDe}. In the case of
super-critical Bernoulli
weights,
\cite{SS} proved the SIP for $d\geq 4$. Later, \cite{MP,BB} proved the
SIP for all $d\geq 2$. These results were recently extended in
\cite{M,BP} to the general case of bounded but possibly vanishing
weights, under the only assumption that positive weights percolate.
More recent
developments include the case of unbounded weights \cite{BD}.
All these works prove the SIP using the approach outlined above, although
the techniques used may differ.
Following \cite{SS,BB,MP},
an important ingredient for the proof of estimate \eqref{neglig} is
represented
by suitable heat kernel and isoperimetric
estimates. However, it is known that such
estimates cannot hold if the system lacks ellipticity, i.e.\ if
arbitrarily small weights are allowed; see \cite{FoMa,BBHK}.
An important idea of \cite{M,BP}
to overcome this problem was then to consider the random walk embedded
in an elliptic cluster and to control the corrector for this
restricted process.

Our random walk on random point process has several features in common
with the random conductance model. The lack of ellipticity corresponds
to the existence in the point process of 
regions of isolated points, where the walk can be trapped. For
instance, it was shown in \cite{CF1} that the existence of these
traps is responsible for the loss of diffusive isoperimetric and
Poincar\'e inequalities, as soon as the power $\a$ in
\eqref{paletti2} is larger than the dimension $d$.


On the other hand, there are some important differences with respect to previous work on the
random conductance model: the long-range nature of the
jumps, the existence of overcrowded regions (i.e.\ regions with atypically high density of points)
and the intrinsically non-deterministic nature of the state
space, that is the lack of a natural lattice structure for the point process.
As we will see, these
are the source of new technical difficulties.

As in \cite{M,BP}, we are forced to work with a
suitable cluster of good points. 
In our setting this good set has to be defined in such a way that:
(i) good points $x$ must have uniformly bounded
weights $w(x)$, (ii) given two  good points $x,y $  there must exist
a path from $x$ to $y$ visiting only good points with uniformly
bounded jump lengths. The requirement (ii) alone could be
achieved by a simple local construction as in \cite{M,BP}. On the
other hand, due to long jumps, non negligible contributions to the
weights $w(x)$ may come from arbitrarily far overcrowded regions.
Therefore, requirement (i)  forces a  nonlocal construction  of
the family of good points, making harder any  quantitative control
on its geometry. For the homogeneous PPP,
 this problem is solved by showing that a suitable discretized $\{0,1\}$--random
 field with
infinite--range
 spatial correlations
 stochastically dominates a supercritical Bernoulli field on $\bbZ^d$; see Section \ref{ambiente}.

In addition,  the ability of the walk
to take long jumps has led us to the development of an extended
version of the analysis of ``holes'' in the cluster of good points
needed in \cite{M,BP}. In particular, a suitable \emph{enlargement
of holes} is required in Section \ref{enlarge} 
to gain some control on the number of jumps and the distance traveled
by the walk between successive visits
to the cluster of good points.

A convenient way to deal with the lack of a lattice structure, and
to obtain statements valid for every starting point
$x_0\in\xi$, is to work with the Palm distribution  of the point process, as
in the statement of Theorem \ref{atreiu}. 
On the other hand, the method
developed in \cite{BB} to establish sublinearity of the corrector
is intrinsically based on a lattice structure
and, at a first analysis,  the  Palm distribution and the lattice
strategy of \cite{BB}  seem   to collide. To overcome this
conceptual obstacle, in several steps of the proof,  we have
introduced  intermediate ``bridge'' distributions (cf.\ Sections
\ref{mediano1}, \ref{mediano}, and \ref{mandarino}).
%
These
distributions are probability measures on the space of the
environments, having both a  lattice structure and an absolutely
continuous Radon--Nykodim derivative with respect to the Palm
distribution (or other related distributions that appear along the proof).
An alternative option would be to follow \cite{MP,M} rather than \cite{BB,BP} to establish (\ref{neglig}).
This approach is more naturally adapted
to the continuum setting.  However, it is more demanding in terms of heat kernel and tightness estimates and
more extra work would be needed to establish the bounds used there.

It is worthy of note that similar problems are encountered when
analyzing random walks on Vorono\u\i\ tessellations or random walks
on the infinite cluster of the supercritical continuous percolation
for Poisson processes (this last case is treated in a work in
progress of S.\ Buckley \cite{Bu}). 
Some of the methods developed here are likely to find applications in the analysis of these other models.

\subsection{Outline of the
paper} As we mentioned, the proof of Theorem \ref{atreiu} is
entirely based on a suitable control of the corrector field. Since
the energy marks play a very minor role in such an estimate, for the
sake of simplicity we set $u(E_x,E_y)=0$, throughout most of the
paper, and we identify the environment $\o$ with the point process
$\xi$. The extension to non trivial energy marks will be discussed
only in Section
\ref{estensione}. 
Another simplification which causes very little loss of generality
is obtained by setting $r(x)=\nep{-|x|^\a}$, for some $\a>0$,
throughout the rest of the paper.

In Section \ref{ambiente} we take a close look at the random
environment, state our main assumptions and define the cluster of
good points. In particular, in Section \ref{PPP} we verify that the
homogeneous PPP satisfies the main assumptions. The corrector field
is introduced in Section \ref{ascolipiceno}. The main sublinearity
estimate for the corrector is stated in Section \ref{theproof}, cf.\
Theorem \ref{piccolo}. There, this estimate is shown to imply
Theorem \ref{atreiu} and Corollary \ref{strong}. The rest of the
paper is then devoted to the proof of the sublinearity estimate.
Section \ref{restricted} introduces the restricted random walk,
i.e.\ the random walk $X_n$ embedded in the cluster of good points.
In particular, we state two crucial estimates: the heat kernel bound
and the expected distance bound. This section also contains the
analysis of ``holes'' in the cluster of good points. The heat kernel
bound is proved in Section \ref{refrigerio}, while the expected
distance bound is proved in Section   \ref{distante}. Section
\ref{sublinear} is entirely devoted to the proof of the sublinearity
estimate. Finally, Section  \ref{estensione} deals with the slight
modifications needed in the presence of energy marks. The appendix
collects several technical results used in the main text.

\section{The random environment
}\label{ambiente} Since we have set $u(\cdot,\cdot)=0$, we may
disregard   the energy marks, and the random environment coincides
with the state space $\xi$ of the random walk, i.e.\ the point
process.

\subsection{Stationary, ergodic simple point process and Palm measure}
 We denote by  $\cN$   the family of
  locally finite subsets $\xi$  of
 $\bbR^d$ endowed of the $\s$--algebra generated by the sets
 $\{ \xi(A_1)=n_1,\dots, \xi (A_k)=n_k\}$,
$A_1, \dots, A_k$ being  disjoint bounded Borel subsets of $\bbR^d$,
$n_1,\dots, n_k$ varying  in $\bbN=\{0,1,\dots\}$ and $\xi(A):= |\xi\cap A|$.
Elements $\xi\in \cN$ are usually identified with the counting
measure on $\xi$. Moreover, given $\xi \in \cN$ and $x \in \bbR^d$,
we denote by $\t_x \xi$ the translated set $\xi -x$. A simple point
process is a measurable map from a probability space to the
measurable space $\cN$.

\smallskip

Fix a  simple point process on $\bbR^d$ with law $\bbP$, ergodic and
stationary w.r.t.\ space translations, having finite density $\rho=\r_1=
\bbE( \xi ([0,1)^d) )$.  Due to stationarity $\rho$ can also be
expressed as $\bbE( \xi (A))/\ell(A)$ for any bounded Borel subset
$A \subset \bbR^d$ having positive Lebesgue measure $\ell(A)$. We
denote by $\bbP_0$ the Palm distribution associated to $\bbP$.
Considering the measurable subset $\cN_0= \{ \xi \in \cN\,:\, 0 \in
\xi\}$, $\bbP_0$ is a probability law on $\cN_0$
    coinciding, roughly speaking, with ``$ \bbP
(\cdot |0 \in \xi)$" (cf.\ Theorem 12.3.V in \cite{DV}).  A key
relation between $\bbP$
 and $\bbP_0$ is given by the Campbell identity \cite{DV}:
for any nonnegative measurable function $f$ on $ \bbR^d
 \times \cN_0$ 
 \begin{equation}\label{campbell}
 \int_{\bbR^d} dx \int_{\cN_0} \bbP_0(d \xi) f(x,\xi)= \frac{1}{\rho}
\int_{\cN} \bbP (d \xi) \int _{\bbR^d} \xi(dx) f(x, \t_x \xi) \,.
\end{equation}

\subsection{Black and white boxes}
For
any $K>0$ we write $B_K=[0,K)^d$ for the cube of side $K$ in
$\bbR^d$. Boxes $B(z):= B_K+Kz$, $z \in \bbZ^d$, are generically
called $K$--boxes. We also use the notation $B_z=B(z)$, for the $K$-box at $z\in\bbZ^d$.
 A $K$--box $B(z)$ is called \emph{occupied} if
$\xi \cap B(z) \not = \emptyset$. We encode this information in the
field $\s= ( \s_z\,:\, z \in \bbZ^d)$ defined on $\cN$ by
\begin{equation}\la{si_z}
\s_z (\xi) =
\begin{cases}
1 & \text{ if $B(z)$ is occupied}\,,\\
0 & \text{ otherwise}\,.
\end{cases}
\end{equation}

Let us now introduce another parameter $T_0>0$. A $K$--box $B(z)$ is
called \emph{overcrowded} if the number of points of $\xi$ in $B(z)$,
$n_z:= \xi ( B(z) )$, satisfies $n_z \geq T_0$. We define
\begin{equation}\la{rzs}
R_z (\xi)= \begin{cases} (\log n_z) ^{2/\a} & \text{ if $B(z)$ is
overcrowded}\,,\\
0 & \text{ otherwise}\,.
\end{cases}
\end{equation}
Next, we define $\cG= \cup_{z \in \bbZ^d} Q(z,R_z) $, where
$Q(z,r)=\{z' \in \bbZ^d\,:\, |z-z'|_\infty < r\}$. Note that
$Q(z,0)=\emptyset$.
Of course, $\cG$
contains all points $z$ such that $B(z)$ is overcrowded. The
interest in the set $\cG$ comes from the following simple estimate.

\begin{lemma}\la{leT}
There exists a positive constant $T= T(\a,K,T_0)$ such that $w(x)
\leq T$, for all $x \in \xi \cap B(z)$ with $z\in \bbZ^d \setminus
\cG$.
\end{lemma}
\begin{proof}
Note that if $x \in B(z)$ and $y \in B(v)$ then $|x-y|_\infty \geq
K|z-v|_\infty-2K$. Therefore we can find positive constants $c_1,c_2$ (depending
on $\a,K,T_0$) such that, for any $x \in
B(z)\cap \xi $ we have
\begin{equation}\label{monte}
w(x) \leq c_1\sum _{v \in \bbZ^d}
n_v \,e^{- c_2 |z-v|_\infty^\a}\,.
\end{equation}
Since $z\in\bbZ^d \setminus \cG$, it must be that all points $v \in
\bbZ^d$ satisfy $|z-v|_\infty \geq (\log n_v)^{\a/2}$. Therefore
$n_v\leq \exp\{ |z-v|_\infty^{\a/2}  \}$
  and using this in \eqref{monte} one has
$w(x)\leq c_3$ for some new constant $c_3=c_3(\a,K,T_0)$.
\end{proof}

We  call a point $z \in \bbZ^d$ \emph{black} if $z$ belongs to $\cG$
or if the box $B(z)$ is unoccupied. If $z$ is not black, we call it
\emph{white}. From Lemma \ref{leT}, if $z$ is white then $w(x)\leq T$, $x
\in\xi \cap B(z) $, for some constant $T$.
Finally, we introduce the field $\vartheta =
(\vartheta_z\,:\, z \in \bbZ^d) $ defined on $\cN$ as
\begin{equation}\la{blawhi}
\vartheta_z (\xi)=
\begin{cases}
0 &\text{ if $z$ is black}\,,\\
1 & \text{ if $z$ is white}\,.
\end{cases}
\end{equation}
The random
fields $ \s(\xi)$ and $\vartheta(\xi)$,  where $\xi$ is sampled with
law $\bbP$, are often denoted simply $\si,\vartheta$.
We shall write $ \s^K$, $\vartheta ^{K,T_0}$, when the
dependence on the parameters $K,T_0$ needs to be emphasized.
Clearly, these random fields are stationary w.r.t.\
$\bbZ^d$--translations due to the stationarity of $\bbP$.

\subsection{Main assumptions on the point process} \la{ambiente.1}
Given a stationary, ergodic point process with finite
density $\rho$ and law $\bbP$, we shall make the following assumptions:
\begin{enumerate}[({H}1)]
\item For each $p \in (0,1)$ there exist $K, T_0>0$ such that
the random field of white points $\vartheta ^{K,T_0}$ stochastically dominates the
independent Bernoulli process $Z(p)$ on $\bbZ^d$ with parameter $p$.

\item For each $K>0$ and for each
  vector $e \in \bbZ^d$ with $|e |_1 =1$, consider the product  probability space $\Theta:=\cN \times
  ( [0,K)^d \cup \{ \partial \} )^{\bbZ} $ whose elements $\bigl(\xi, (a_i : i \in \bbZ) \bigr)$ are sampled as follows:
  choose $\xi$ with law  $\bbP$, and then  choose independently  for  each index $i$  a point $b_i \in \xi \cap B( i e)$ with uniform probability
  and set $a_i:= b_i - i K e\in   [0,K)^d$. If $\xi \cap B(ie)=\emptyset$, set
  $a_i = \partial$.
We assume that  the resulting law $P^{(K, e)} $ on
 $\cN \times
  ( [0,K)^d \cup \{ \partial \} )^{\bbZ} $
  is ergodic w.r.t.\ the transformation
 \begin{equation}\t :\bigl(\xi, (a_i : i \in \bbZ) \bigr) \to \bigl( \t_{Ke } \xi , ( a_{i+1} : i \in \bbZ ) \bigr) \,.
 \end{equation}
\end{enumerate}

\subsubsection{Remarks}\label{equoesolidale}
Since $\vartheta_z=1$ implies $\si_z=1$, it is clear that
assumption (H1) implies the following statement, which we shall refer to as Property (A):
\begin{enumerate}[({A})]
\item For all $p\in(0,1)$, there exists $K>0$ such that
the random field $ \si^{K}$  stochastically dominates the
independent Bernoulli process $Z(p)$ on $\bbZ^d$ with parameter $p$.
\end{enumerate}

Also, it is not hard to check that if $p(K)$ is the largest $p$ such that $\si^K$ stochastically dominates $Z(p)$, then
$p(mK)\geq (1 - (1-p(K))^{m^d})$, and therefore $p(mK)\to 1$ as $m\to\infty$.

\smallskip

Observe that  the law $P^{(K,e)} $ is invariant w.r.t.\
the transformation $\t$ (due to the stationarity
of $\bbP$).  Assumption (H2) means that, for each $K>0$
and for each
  vector $e \in \bbZ^d$ with $|e |_1=1$, any measurable
  subset $\cA \subset \Theta$ such that $ \t \cA= \cA$ must have
  $P^{(K,e)}$--probability $0$ or $1$.
  We point out that assumption (H2) alone
   implies that   $\bbP$ is ergodic.

When working with the energy marks, assumption (H2)
will be slightly modified as discussed in Section \ref{estensione}.

\subsection{The homogeneous PPP 
satisfies (H1)--(H2)}\label{PPP} The homogeneous PPP with density
$\rho$ is plainly an ergodic, stationary simple point process with
finite moments of any order. In order to prove assumption (H2), we
fix $\cA\subset \Theta $ such that $\t\cA= \cA$ and set
$P=P^{(K,e)}$. If $\cA$ depends only on  $\xi $ restricted to
$[-\ell K,\ell K ]$ and on
 $\{ a_i : |i| \leq \ell -1 \}$ for some integer $\ell$, then
$\cA$ and $ \t^{m \ell} \cA$ are independent for $m$ large and
therefore
$
 P(\cA) = P( \cA \cap \t^{m \ell} \cA)
 = P(\cA) P( \t^{m \ell} \cA) =  P( \cA)^2$. This implies that
 $P(\cA)\in \{0,1\}$. The general case can be treated by a standard
 approximation argument.
The rest of this section is concerned with the proof of assumption
(H1), which we reformulate as follows.
\begin{theorem}\la{poisa3}
For every $p\in(0,1)$ and $\r>0$, there exist constants $K,T_0$,
depending on $p$ and $\r$, such that, for the homogeneous PPP with
density  $\r$, the random field $\vartheta = \vartheta^{K,T_0}$
defined in \eqref{blawhi} stochastically dominates the Bernoulli
field $Z(p)$ with parameter $p$.
\end{theorem}

\subsubsection{Preliminary estimates}
Before we start the proof of Theorem \ref{poisa3} we shall establish a few
preliminary facts.
\begin{lemma}\la{poisest}
A Poisson variable $N$ with mean $\l$ satisfies
\[
 \bbP(N>t)
 \leq \exp\{ - t(\log t -\log \l )+t -\l\} \leq \exp\{-t\} \,, \qquad
 \forall t \geq e^2 \l.
\]
\end{lemma}
\begin{proof} Take $s = \log (t/\l )$ in the
following expression, valid for all $s\geq0$:
\[
 \bbP( N>t)=\bbP( e^{sN} > e^{st})
 \leq e^{-s t}\bbE(e^{s N})= \exp \{ -s t +\l e^s-\l\}\,.\qedhere
\]
\end{proof}
Next, recall the definition \eqref{rzs} of the set $\cG$. The random
variables $n_z$ are i.i.d.\ Poisson variables with mean $\r K^d$, and using Lemma
\ref{poisest}, the variables $R_z$ satisfy
\begin{equation}\label{tonsilla}
\bbP( R_z > r) = \bbP \big( n_z \geq \exp \big(r^{\frac{\a}{2}}\big)\,,\, n_z
\geq T_0 \big) \leq \exp \big( -\exp \big( r^{\frac{\a}{2}} \big) \big)
\end{equation}
whenever $\exp \big(r^{\frac{\a}{2}} \big) \geq e^2 \r K^d$.

Set $\g_m=\exp\big( -\exp ( m^{\frac{\a}{4}} ) \big)$, $m\in\bbN$, and
consider the Bernoulli random field $Z(\g_m)$ on $\bbZ^d$ with
parameter $\g_m$. Next, let $\{Z(\g_m), m\in\bbN\}$ denote an
independent sequence of the random fields $Z(\g_m)$ on some
probability space with law $P$  and set
\[ \wt R_z:= \sup \{ m \geq m_0\,:\, Z(\g_m)_z=1\}\,, \]
with the convention that the supremum of the emptyset is $0$. Here
and below $m_0$ is a constant related to $T_0$ by \be\la{m0T0}
T_0=\exp \(m_0^{\a/2} \) \,.
\end{equation}
Note that the random variables $\wt R_z$, $z \in \bbZ^d$, are
independent. Moreover,  $\wt R_z $ is finite $P$--a.s.\ since
\[
 E \Big( \sum_{m=m_0}^\infty Z (\g_m)_z \Big)
 = \sum_{m=m_0}^\infty \g_m < \infty\,.
\]
\begin{lemma} \la{rober}
For all $\r,K>0$, there exists a constant $T_0$ such that, for all $z\in\bbZ^d$:
\begin{equation}\label{roberto}
 P ( \wt R_z <t) \leq \bbP( R_z <t)\,,
\qquad \forall t >0\,.
\end{equation}
 \end{lemma}
\begin{proof}
For every $t>0$,
\[
 P(\wt R_z < t)=
 \prod _{k\geq  \lceil t\rceil \lor m_0 }^\infty(1-\g_k)
 \leq \exp\Big\{-\sum_{k\geq  \lceil t\rceil \lor m_0 }\g_k\Big\}\,.
\]
If $t\leq m_0$, then
$
 P (\wt R_z<t)=
 P( \wt R_z=0)=\prod _{k= m_0 }^\infty (1- \g_k) $ which can be bounded by  $e
 ^{-\g_{m_0}}$.
Now, take $m_0,T_0$ as in \eqref{m0T0} and assume $T_0 \geq e^2 \r
K^d$. In particular, $R_z<m_0$ is equivalent to $R_z=0$, cf.\
\eqref{rzs}, and \eqref{tonsilla} holds for all $r\geq m_0$.
%
Therefore, for all
$t \leq m_0$ one has
$$
\bbP( R_z <t) = \bbP( R_z =0) = 1-\bbP(R_z\geq m_0) \geq 1 - \exp
\bigl( -\exp \bigl( m_0^{\a/2} \bigr) \bigr). $$ Using that $\nep{-2
x}\leq 1-x$ for $x\in[0,\frac12]$, and that $
 \exp \bigl( -\exp \bigl( m_0^{\a/2} \bigr) \bigr)
 \leq  \g_{m_0}/2
$ for  $m_0$  sufficiently large, we conclude that
$
 \bbP( R_z <t) \geq 1 - \g_{m_0}/2
 \geq e^{ - \g_{m_0} }
$ for all $t \leq m_0$. This concludes the proof of \eqref{roberto}
for $0<t\leq m_0$.

Suppose now that $m_0 \leq m-1 < t \leq m $. Then $ P(\wt R_z<t)=
\prod _{k=m}^\infty (1- \g_k) \leq   e ^{- \g_{m}} $. On the other
hand, reasoning as above we see that  $\bbP(R_z\geq m-1)\leq
\frac12\g_{m}$ and therefore
$$
\bbP( R_z <t) \geq  \bbP( R_z \leq m-1) =1-\bbP(R_z\geq m-1)
 \geq 1- \frac12\g_{m}
 \geq e^{ - \g_{m} }\,.
$$
This ends the proof of the lemma.
\end{proof}

Lemma \ref{rober} implies that the random field $R=(R_z:z \in
\bbZ^d)$ is stochastically dominated by the random field $\wt R=
(\wt R_z : z \in \bbZ^d)$. 
Taking a coupling between $R$ and $\wt R$ on an enlarged probability
space such that $ R_z \leq \wt R_z$ for all $z \in \bbZ^d$ a.s.\ we
get that \be\la{gbound} \cG=\cup _{a \in \bbZ^d} Q(a,R_a) \subset
\wt \cG :=\cup_{a \in \bbZ^d} Q(a, \wt R_a)\,.
\end{equation}
The random set $\wt \cG$ can be described by the random field
$Y=(Y_z\,:\, z \in \bbZ^d)$ (in the sense that $ z \in \wt \cG $ if
and only if $ Y_z =1$)
where
\[
 Y_z:=\max \{ Y^{(m)}_z\,:\, m \geq m_0\}\,,
\]
and, for each $m$,
\[
 Y^{(m)}_z:=
\begin{cases}
 1 & \text{ if } \exists a \in \bbZ^d:\, Z(\g_m)_a=1,   z \in Q(a,m)\,,\\
 0 & \text{ otherwise}\,.
\end{cases}
\]
\begin{lemma}\label{kdep}
For every $K,\r>0$, there exist a constant $T_0$ such that
for each $m \geq m_0$, the random field $Y^{(m)}$
is stochastically dominated by the Bernoulli
random field $Z(q_m)$ with parameter $q_m:=2^{-m+1}$.
\end{lemma}
\begin{proof}
We apply a result of \cite{LSS} on stochastic domination, in the form
which appears in Theorem (7.65) in \cite{G}.
Namely, set $\ell= 2 (m +1)+1$, so that $Y^{(m)}$
is a $\ell$--dependent field taking values in $\{0,1\}$. Note that
\begin{equation}\label{mangio}
P( Y^{(m)}_0 =1) \leq \ell^d \g_m \,. \end{equation} Suppose that
there exist two parameters $u,v>0$ such that
\begin{align*}
& (1-u) (1-v) ^{\ell^d} \geq  
\ell^d \g_m \,,
\\
 & (1-u ) u^{\ell^d}
\geq
\ell^d \g_m\,.
\end{align*}
Then, by Eqs.\ (7.114)--(7.115) in \cite{G}, we know that $Y^{(m)}$
is stochastically dominated by an independent Bernoulli random field
with parameter $1-uv$. Therefore, we have to prove that $u$ and $v$
can be taken in such a way that $1-uv\leq 2^{-m+1}$. This can be
achieved by the choice $ u = 1-\g_m\,\ell^d \,2^{m\ell^d}$ and $  v
= 1-2^{-m}.$  Indeed, $(1-u) (1-v) ^{\ell^d} =\ell^d \g_m $ and
$(1-u ) u^{\ell^d} \geq \g_m\,\ell^d$ if $m$ is large
enough (using that $1-x \leq e^{-x}$). 
Moreover, by definition $1-uv\leq \g_m\,\ell^d \,2^{m\ell^d} +
2^{-m} \leq 2^{-m+1}$ if $m$ is large enough. This concludes the
proof.
\end{proof}

Since all random fields $Y^{(m)}$ are independent, thanks to the
above lemma we can build, on a suitable probability space, the
random fields $(Y^{(m)},Z(q_m) )$, $m \geq m_0$, such that they are
all independent and
\[
 Y^{(m)}_z \leq Z(q_m) _z  \qquad \forall z \in \bbZ^d\,,\qquad\text{a.s.}\,
\]
In particular we have that
$Y= \max \{ Y^{(m)}\,:\, m \geq m_0\}$ is
stochastically dominated by the random field $Z:= \max \{
Z(q_m)\,:\, m \geq m_0\}$.
Note that $Z$ is a Bernoulli random field, with parameter
\be\la{qbound}
q= \bbP( Z_0=1)
\leq \sum _{m=m_0}^\infty q_m
= \sum _{m=m_0}^\infty  2^{-m+1} = 2^{-m_0+2}\,.
\end{equation}

\subsubsection{Proof of Theorem \ref{poisa3}}
The results discussed above can be summarized as follows.
\begin{proposition}\la{Gdom}
For every $K,\r>0$, and $\e>0$, there exists $T_0$ such that,
for the homogeneous PPP with intensity $\r$, the
random set $\cG=\cG(K,T_0)$ is stochastically
dominated by the Bernoulli field $Z(\e)$ with parameter $\e$.
\end{proposition}
\begin{proof}
From Eq.\ \eqref{gbound}, Lemma \ref{kdep} and Eq.\ \eqref{qbound}, it suffices
to take $T_0$ (and therefore, by \eqref{m0T0}, $m_0$) so large that $q\leq \e$.
\end{proof}
We can now conclude the proof of Theorem \ref{poisa3}.
Let us fix $p\in (0,1)$ and $\e=1-\sqrt p$.
Then fix $K=K(\r)$
such that $\bbP (\xi(B(0))=0)\leq\e/2$. Also, let $t_0>0$ be so large that
$\bbP (\xi(B(0))=0\tc \xi(B(0))<t )\leq\e$ for all $t\geq t_0$.

Next, choose $T_0=T_0(K,\r)$ so large that $\cG$ is stochastically
dominated by $Z(\e)$ as in Proposition \ref{Gdom}. Let $\o_z=1$ if
$z\in\bbZ^d\setminus \cG$ and $\o_z=0$ otherwise. For fixed $z$, let
$A_z\subset \{0,1\}^{\bbZ^d}$ be an arbitrary measurable set such
that $A_z\subset\{\o_z=1\}$. If $T_0\geq t_0$ then, by independence
of the Poisson field, one has
\[
\bbP\bigl(\xi(B(z))>0\tc \o\in A_z\bigr)
= \bbP\bigl(\xi(B(z))>0\tc \xi(B(z))<T_0\bigl) \geq 1-\e\,.
\]
Since this bound is uniform over all possible values of $\o_{z'}$,
$z'\neq z$, it follows that the set of white boxes $\vartheta$,
i.e.\ $\vartheta_z=\o_z\si_z$, $z\in\bbZ^d$,
stochastically dominates the
Bernoulli field with parameter $(1-\e)^2=p$. This ends the proof.

\section{The corrector field $\chi$}\label{ascolipiceno}
Let $\mu$ be the measure on $\cN _0\times \bbR^d$ such that  the
scalar product in $L^2(\cN_0\times \bbR^d, \mu)$ is given by
\begin{equation}\la{scalare}
(u,v)_\mu
= \bbE_0\Bigl[\sum_{x\in\xi} r(x)  \,u(\xi,x) \,v (\xi,x)\Bigr]\,.
\end{equation}
Since $\bbP$ has a finite second moment, by Lemma \ref{techn_lemma} in
the appendix
\[
 (1,1)_\mu=
 \bbE_0\Bigl[\sum_{x\in\xi} r(x)
 \Bigr]
 = \bbE_0[w(0)]<\infty\,.
\]

\subsection{Potential vs.\  solenoidal forms}
We call $ u \in L^2 (\mu)$ a \emph{square integrable form}. In what
follows we shall study this space in some detail. In general, we
will call a \emph{form} any measurable function $u\colon\cN_0 \times
\bbR^d \to \bbR$.

\medskip

Given $\psi\colon\cN_0 \to \bbR$ we define the
\emph{gradient form} $\nabla \psi$ as
\begin{equation}
\nabla \psi (\xi, x)\coloneq
\psi (\t_x \xi ) -\psi (\xi)
\,.
\end{equation}
Hence, $\nabla \psi \in L^2 (\mu)$ whenever $\psi$ is bounded
(written  $\psi \in B(\cN_0) $).

The space $\cH_\nabla \subset L^2 (\mu)$  of \emph{potential forms}
is defined as the closure of the subspace given by the gradient
forms $\nabla \psi$ with $\psi \in B(\cN_0)$.
Its orthogonal complement $\cH_\nabla^\perp$ is the space  
of \emph{solenoidal forms}.

A form $u\colon\cN_0 \times \bbR ^d\to \bbR $  is called
\emph{curl--free} if for any $\xi \in \cN_0$, $n\geq 1$ and any
family of $n$ points $x_0, x_1 , \dotsc , x_n\in\xi$ with $x_0=x_n$,
we have
\begin{equation}\label{pappa}
 \sum _{j=0}^{n-1} u (\t_{x_j}\xi, x_{j+1}-x_j )=0\,.
\end{equation}
 A square integrable form $u\in L^2(\mu)$ is
called curl--free if
this 
holds for $\bbP_0$--a.e.\ $\xi$.

\begin{lemma}\label{idraulico}
Each potential form $u\in \cH_\nabla$ is curl--free.
\end{lemma}
\begin{proof}
This is trivial to check for $u=\nabla \psi$, $\psi\in B(\cN_0)$. In
the general case, let $\psi_n$ be a sequence in $B(\cN_0)$ such that
$\nabla \psi _n$ converges to $ u$ in $L^2 (\mu)$. By taking a
subsequence we can assume that the convergence holds also $\mu$--a.s.
Since each $\nabla \psi_n$ satisfies  \eqref{pappa} $\bbP_0$--a.s.\
by taking the limit in \eqref{pappa} we conclude that the same
identity holds for $u$.
\end{proof}

A form $u$ is called \emph{shift--covariant} 
if
\begin{equation}\label{alzatipresto}
 u(\xi,x) = u(\xi,y) + u(\t_y\xi,x-y)\,,\qquad \forall\,x,y\in\xi\,.
\end{equation}
If $u$ is a square integrable  form, we call it shift--covariant if
the above property holds for $\bbP_0$--a.a.\ $\xi$.
\begin{lemma}\label{implicazioni}
Each 
curl--free form is shift--covariant.
\end{lemma}
\begin{proof}
Let $u : \cN_0 \times \bbR^d \to \bbR$ be a curl--free form.
Taking  in \eqref{pappa} $n=3$, $x_0=x_3=0$ (recall that $\xi \in
\cN_0$), $x_1=y$ and $x_2=x$, we get that
\begin{equation}\label{ariafresca}
 u(\xi, y)+ u ( \t_y \xi, x-y)+ u(\t_x
\xi, -x)=0\,. \end{equation}
 On the other hand, taking in
\eqref{pappa} $n=2$, $x_0=x_2=0$ and $x_1=x$, we obtain
\begin{equation}\label{ariaviziata}
u(\xi, x)+ u(\t_x \xi, -x)=0
\end{equation}
for any $\xi$. From \eqref{ariafresca} and \eqref{ariaviziata} one
obtains \eqref{alzatipresto}.
\end{proof}

\smallskip

Given $u\in L^2 (\mu)$ we define the \emph{divergence}
as
$\Div u \,(\xi) =
\sum_{x\in\xi} r(x) u(\xi, x)$.
Since $\bbE_0\lvert\Div u\rvert
\leq (u,u)_\mu^{1/2}(1,1)_\mu^{1/2}$, we have that $\Div u\in
L^1(\mu)$.
By these definitions, 
we have a key relation between the gradient and divergence.

\begin{lemma}\label{carattere}
 For each $\psi\in B(\cN_0)$ and each  curl--free $u\in L^2 (\mu)$ 
 \begin{equation}\label{integro}
  (u,\nabla\psi)_\mu
  = -2\bbE_0\left[ \psi \,\Div u \right]\,.
 \end{equation}
 In particular, a form $u \in L^2 (\mu)$ is solenoidal (that is,
 $u \in \cH_\nabla^\perp$) if
 and only  if $\Div u =0$, $\bbP_0$--a.s.
\end{lemma}
\begin{proof}
We only need to prove \eqref{integro}, since the last statement is
then obvious. Due to \eqref{ariaviziata} (which holds for all $x \in
\xi$, $\bbP_0$--a.s.)  we can rewrite the l.h.s.\ of \eqref{integro}
as
\begin{equation}\label{fucile}
(u, \nabla \psi)_\mu= -\bbE_0 \Big[ \sum _{x \in\xi} r(x)
u(\t_x\xi,-x) \psi (\t_x\xi )\Big] - \bbE_0 \Big[\sum _{x \in \xi}
r(x) u(\xi,x) \psi (\xi )\Big]\,.
\end{equation}
We define the function $f$ on $ \cN_0\times \bbR^d
$ as $ f(\xi,x )= r(x) u ( \xi, x) \psi (\xi) $. Then  it holds $ f(
\t_x \xi,-x )=r(x) u(\t_x \xi, -x) \psi (\t_x \xi)$. In addition,
$$ \bbE_0\Big[ \sum_{x \in \xi} f(\xi, x)\Big] \leq \| \psi\|_\infty
\bbE_0  [ w(0)]^{1/2} (u,u)_\mu^{1/2} <\infty\,.
$$
  This allows us to apply
Lemma \ref{techn_lemma} (i) in the Appendix to the function $f$  and
to conclude that
\begin{equation}
\bbE_0 \Big[\sum _{x \in \xi} r(x) u(\xi,x) \psi (\xi )\Big]= \bbE_0
\Big[ \sum _{x \in\xi} r(x) u(\t_x\xi,-x) \psi (\t_x\xi )\Big]
\end{equation}
(by Lemma \ref{techn_lemma} (i) we know that  the integrand in the
r.h.s.\  belongs to $L^1(\bbP_0)$). The above identity allows then
to rewrite the r.h.s.\ of \eqref{fucile} as
\[
 - 2\bbE_0 \Big[\sum _{x \in \xi} r(x) u(\xi,x) \psi (\xi )\Big]=
 -2 \bbE_0 [ \psi \Div u ]\,.\qedhere
\]

\end{proof}

\begin{lemma}\label{rino}
   Let $u \in \cH_\nabla^\perp$.
 Then for $\bbP_0$--a.a.\ $\xi$
\begin{equation}\label{maradona}
\sum _{y \in \xi} r(y) |u(\xi,y)|<\infty\,, \qquad  \sum _{y \in
\xi} r(y) u(\xi,y)=0\,.
\end{equation}
In particular, for $\bbP$--a.a.\ $\xi$ and for all $x \in \xi$
\begin{equation}\label{benedetto}
\sum _{y \in \xi} r(y-x) |u(\t_x\xi,y-x)|<\infty\,, \qquad  \sum _{y
\in \xi} r(y-x) u(\t_x\xi,y-x)=0\,.
\end{equation}
\end{lemma}
\begin{proof} The second statement \eqref{benedetto} follows from
the first one \eqref{maradona} by Lemma \ref{francoforte} in the appendix.
 In order to prove \eqref{maradona}, we first observe that
$$ \int \bbP_0(d\xi) \sum _{y \in \xi}r(y)|u(y)|\leq
 \bbE_0 \bigl[ w(0) \bigr]^{\frac 12} \Big\{  \int
\bbP_0(d \xi) \sum _{y \in \xi} r(y) |u(\xi,y)|^2\Big\}^{\frac 12}
=C (u,u)^{\frac 12}_\mu\,,
$$
and  the last member is finite. This implies the upper  bound in
\eqref{maradona}. The identity in \eqref{maradona} is equivalent to
$\Div u=0$ $\bbP_0$--a.s.\,, which follows from the previous lemma.
\end{proof}

\subsection{Corrector field}
 We can now define the corrector field $\chi$ following the construction of
 \cite{MP}.
Consider the form
$u_i\colon\cN_0 \times \bbR^d\to \bbR^d$, $i=1,\dots,d$,
defined by $u_i(\xi,x)=x_i$ (the $i$--th coordinate of
$x\in\bbR^d$). Note that, since $\bbP$ has finite second moment,
Lemma \ref{techn_lemma} assures us that $u_i\in L^2(\mu)$.
 Let
$\pi\colon L^2 (\mu) \to \cH_\nabla$ be the orthogonal
projection on potential forms and
define
\begin{equation*}
 \chi_i\coloneq\pi (-u_i)\,,\quad i=1,\dotsc,d\,.
\end{equation*}
Setting
$\Phi_i:=x_i+\chi_i\in \cH_{\nabla}^\perp\,,$ from Lemma \ref{rino}
we see that $\Phi_i$ is {\em harmonic}, i.e.\ for $\bbP_0$-a.a.\
$\xi$, $\Phi_i\in L^1(P_{0,\xi})$ and $E_{0,\xi} \Phi_i = 0$, for
all $i=1,\dots,d$. The vector form $\chi=(\chi_1,\dotsc,\chi_d)$ is
the so called \emph{corrector field}.

\ignore{ Given $\xi \in \cN$,  a function $u: \xi \to \bbR,
 \bbR^d$ is called  \textsl{harmonic} w.r.t.\  the random walk $X_n$ if
for all $x \in \xi$
 $u(X_1)\in L^1( P_{x,\xi})$ and moreover
 $E_{x,\xi} \bigl(u(X_1) \bigr)= u(x) $. In other words, it must be $
\sum_{y \in\xi} r(y-x) |u(y)|<\infty$ and $\sum_{y \in\xi} r(y-x)
\bigl( u(y)-u(x)\bigr)$ for all $x \in \xi$.

\begin{proposition}\label{armonia}
 For $\bbP$--a.a.\ $\xi$, 
 the map
 \begin{equation*}
  \xi \ni x \mapsto \varphi(\xi,x): =
  x+ \chi (\xi,x) \in \bbR ^d\,,
 \end{equation*}
 is harmonic  w.r.t.\ the random walk  $X_n$.
\end{proposition}
\begin{proof}
By definition, $u_i+\chi_i \in \cH_\nabla^\perp$. On the other hand,
it is straightforward to check that $u_i$ is curl--free, so that
$u_i+\chi_i$ must be curl--free.
Hence it is enough to apply Lemma \ref{rino} with $\Phi= u_i+\chi_i$
for each $i$.
\end{proof}
}

Up to now $\chi$ has been defined as element of $L^2( \mu)^d $,
hence as a pointwise function it is defined  modulo a set of zero
$\mu$--measure. It is convenient to work with a special
representative of $\chi$, which is  everywhere defined
  on $\cN_0 \times \bbR^d$ and has good properties:
\begin{lemma}\label{festicciola}
There exists a representative $\bar \chi: \cN_0 \times \bbR^d\to
\bbR^d$ of the corrector $\chi \in L^2 (\mu)^d $ such that
\begin{equation}\label{coop1}
\bar \chi (\xi,x)= \bar  \chi (\xi, y)+\bar   \chi (\t_y \xi,
x-y)\,, \qquad \forall \xi \in \cN_0, \; \forall x,y \in \xi\,.
\end{equation}
In particular, $\bar  \chi (\xi ,0)=0$ for all $ \xi \in \cN_0$.
\end{lemma}
\begin{proof}
The conclusion of the Lemma follows from \eqref{coop1} by taking
$x=y=0$. Let us therefore concentrate on \eqref{coop1}.  Due to
Lemma \ref{idraulico}, $\chi_i$ is a curl--free square integrable
form. We fix a representative $\hat \chi _i$   of $\chi_i$ as
pointwise function on $\cN_0 \times \bbR^d$ and  call $\cB_i \subset
\cN_0$ the set of $\xi$ satisfying \eqref{pappa} w.r.t.\ the form
$\hat \chi_i$,  for any family of $n $ points $x_0, x_1, \dots, x_n$
in $\xi$. By definition, it must be $\bbP_0 (\cB_i )=1$.

We claim that if $\xi \not\in \cB_i$ then $\t_x \xi \not\in \cB_i$ for all
$x\in\xi$.
Suppose for the sake of contradiction that  $ \t_x\xi \in \cB_i$ and fix a
family of $n $ points $x_0, x_1, \dots, x_n$ in $\xi$. Then the
points $y_0,y_1, \dots ,y_n$ defined as $y_k=x_k-x$ lie in $\t_x
\xi$. Because $\t_x \xi \in \cB_i$, we conclude that
\[
 0=
 \sum _{j=0}^{n-1} \hat \chi_i ( \t_{y_j}( \t_x \xi), y_{j+1}-y_y)=
 \sum_{j=0} ^{n-1} \hat \chi _i ( \t_{x_j} \xi, x_{j+1}-x_j) \,,
\]
thus implying that $\xi \in \cB_i$, which is a contradiction. This concludes
the proof of our claim.

 At this point we define $\cB= \cap _{i=1}^d
\cB_i$ and
\[
 \bar \chi_i  (\xi,x):=
 \begin{cases}
  \hat \chi_i  (\xi, x) & \text{ if } \xi \in \cB\,,\\
  0 & \text{ otherwise}\,.
 \end{cases}
\]
Let us check \eqref{coop1}. If $\xi \in \cB$ and $x\in \xi$, then
also $\t_y \xi$ must belong to $\cB$ (if it was not in $\cB_i$ for
some $i$, since $-y \in \t_y \xi$ we would conclude that $\xi=
\t_{-y}( \t_y \xi) $ does not belong to $\cB_i\supset \cB$). In
particular, the identity \eqref{coop1} can be rewritten as
\[
 \hat \chi _i (\xi, x)
 = \hat \chi_i (\xi,y)+ \hat \chi_i( \t_y\xi, x-y)\,,
 \qquad \forall i= 1, \dots, d\,,
\]
which is trivially true by definition of $\cB$ and $\cB_i$. Take now
$\xi \not \in \cB$ and $x,y \in \xi$. By  definition of $\bar
\chi$ we get $\bar  \chi (\xi, x)= \bar \chi (\xi, y)=0$. Since
for some $i$ it must be $\xi \not \in \cB_i$, we know that  also
$\t_y \xi $ does not belong to $\cB_i\supset \cB$. As consequence,
it must be  $\bar \chi (\t_y \xi , x-y)=0$ and the identity in
\eqref{coop1} reduces to $0=0+0$.
\end{proof}

From now on, when working with the corrector field $\chi$ we will
always refer to the pointwise function $\tilde \chi : \cN_0 \times
\bbR^d \to \bbR^d$ of the above lemma.

\subsection{Sublinearity and the proof of Theorem \ref{atreiu}  and
  Corollary \ref{strong}}\la{theproof}
The core of the proof of Theorem \ref{atreiu} lies in the
following result:
\begin{theorem}\label{piccolo}
Under the assumptions of Theorem \ref{atreiu}: for $\bbP_0$--a.a.\ $\xi$
\begin{equation}
\lim _{n\to \infty} \,\frac{1}{n}\, \max_{\substack{x \in \xi\,:\\
|x|_\infty \leq n } }  \bigl| \chi (\xi, x) \bigr|=0 \,.
\end{equation}
\end{theorem}

The proof of Theorem \ref{piccolo} is completed in Section \ref{sublinear}.
Here, we show how Theorem \ref{atreiu} follows from Theorem
\ref{piccolo}. The argument is standard; see \cite{SS,BB,BP,BD}
for very similar arguments. We only sketch the main steps.

\subsubsection{Proof of Theorem \ref{atreiu}}
Let us start with the discrete parameter case.
Set $\Phi(\xi,x):=x+\chi(\xi,x)$, so that $M_n = \Phi(\xi,X_n)$,
$n\in\bbN$, is a martingale (see Lemma \ref{rino}).
Also, define $M^v_n:=v\cdot M_n$, for
$v\in\bbR^d$, and, for every $K\geq 0$:
\[
F_K(\xi)=E_{0,\xi}\(|M^v_1|^2\,;\;|M^v_1|\geq K\)
\]
Let $\wt\bbP_0$ denote the probability on $\cN_0$ with density
$\frac{w(0)}{\bbE_0 w(0)}$ with respect to
$\bbP_0$, and let $\wt \bbE_0$ denote the associated expectation.
Recall that the Markov chain on environments,
$n\mapsto \t_{X_n}\xi$, is ergodic with reversible invariant
distribution given by $\wt \bbP_0$; see \cite[Proposition 2]{FSS}.
Therefore, for every
$K\geq 0$:
\be\la{ergo1}
\frac1n\sum_{k=0}^n F_K\circ\t_{X_k}\xi
\to
\wt \bbE_0 F_K\,,
\end{equation}
$P_{0,\xi}$--a.s., for $\bbP_0$--a.a.\ $\xi$, as $n\to\infty$.
As in \cite[Section 6.1]{BB} and \cite[Section 5]{BP}, using
the above convergence together with the monotonicity of $K\mapsto F_K$,
allows us
to verify the assumptions of the Lindeberg-Feller Martingale
Functional CLT (as in, e.g., \cite[Theorem 7.7.3]{Durrett}).
It follows that, for every $v\in\bbR^d$,
$\bbP_0$--a.s.\ $t\mapsto \e\,M^v_{\lfloor t/\e^2\rfloor}$ converges
weakly, as $\e\to 0$,
to one-dimensional Brownian motion with diffusion coefficient
\begin{equation}\la{dtrw}
\left\langle v,D_{DTRW}v
\right\rangle = \wt\bbE_0 E_{0,\xi}\(|M^v_1|^2\)\,.
\end{equation}
In particular, with the notation \eqref{scalare}, it holds $
[D_{DTRW}]_{i,j} = (\Phi_i,\Phi_j)_\mu /(1,1)_\mu $, where $\Phi_i =
x_i + \chi_i$, $i,j=1,\dots,d$. That $D_{DTRW}$ is positive definite
follows from the fact that if \eqref{dtrw} is zero for some
coordinate axis $v=e_j$, $j=1,\dots,d$, then $x_j = -
\chi_j(\xi,x)$, $\bbP_0$--a.s.\ for every $x\in\xi$, and this is not
compatible with Theorem \ref{piccolo}.

To conclude the proof we 
argue as in \cite[Section 6.2]{BB}. Namely, from Theorem
\ref{piccolo} we have that $\bbP_0$--a.s.\ for $\d\in(0,1/2)$, there
exists $\k(\xi)<\infty$ such that for all $x\in\xi$ it holds $
|\chi(\xi,x)|\leq \k(\xi) + \d\,|x|$.  Writing
$x=\Phi(\xi,x)-\chi(\xi,x)$, one has
\[
 \max_{k\leq n}|\chi(\xi,X_k)|\leq 2\k(\xi) + 2\d\,\max_{k\leq n}\,|M_k|\,,
\]
which implies, by the arbitrariness of $\d$, that
$\max_{k\leq n}|\chi(\xi,X_{n})|=o(\sqrt n)$ in
$P_{0,\xi}$ probability, for $\bbP_0$--a.a.\ $\xi$. This ends the proof
of Theorem \ref{atreiu} for the DTRW.

To treat the CTRW, observe that it consists of a time change of the
DTRW. Indeed,  the CTRW waits at site $x$ an exponential time of
parameter $w(x)$ and then jumps to $y\in\xi$ with probability
$p(x,y)$ (it could be $y=x$). With this notion of ``jump'',   if
$n^*(t)$ denotes the number of jumps of the CTRW up to time $t$,
then the CTRW at time $t$ coincides with the DTRW at $n^*(t)$.
Therefore, arguing as in \cite[Theorem 4.5]{DFGW}, it is sufficient
to show that the limit
\begin{equation}\la{ergos}
\lim_{t\to\infty} n^*(t)/t = \bbE_0 w(0)\,,
\end{equation}
holds $P_{0,\xi}$--a.s., for $\bbP_0$--a.a.\ $\xi$. Let $\si_i$,
$i=0,1,2,\dots$ denote an independent family of i.i.d.\ exponentials
of parameter $1$ and write $T_i:=\si_i/w(X_i)$ for the waiting time
after the $i$-th jump of the discrete-time chain. Then, setting
$R_0=0$, and $R_n = T_{n-1} + \cdots + T_0$, we have that $n^*(t)=n$
if and only if $R_{n-1}\leq t < R_n$. Observing that
$w(x)=w(x,\xi)=w(0,\t_x\xi)$ for every $x\in\xi$, and invoking the
ergodicity of the environment process $n\mapsto \t_{X_n}\xi$ as in
\eqref{ergo1} above, we see that $\bbP_0$--a.s.\ and
$P_{0,\xi}$--a.s.
\[
\lim_{n\to\infty} R_n/n = \wt \bbE_0 \bigl(1/w(0)\bigr) = 1/ \bbE_0
w(0)\,.
\]
This implies \eqref{ergos}. Moreover, this also shows that $
[D_{CTRW}]_{i,j} = (\Phi_i,\Phi_j)_\mu\,, $ and the relation
\eqref{relaz} must hold. \qed

\subsubsection{Proof of Corollary \ref{strong}}
This corollary is a direct
consequence of Theorem \ref{atreiu} and Lemma \ref{francoforte} in the appendix, taking in Lemma \ref{francoforte} $\cA_0$
as
the  set of configurations $\xi \in \cN_0$ such that both the DTRW and the CTRW starting at the origin converge  under diffusive rescaling  to the Brownian motions described in Thereom \ref{atreiu}.

\section{Restricted random walk}\label{restricted}
Recall the definition of occupied boxes and white boxes, as in Section
\ref{ambiente}.
As consequence of (H1), see also property (A), once $p$ is taken large enough,
$\bbP$--a.s.\ the random sets $\{x \in \bbZ^d\,:\, \s_x=1\}$ and
$\{x \in \bbZ^d\,:\, \vartheta _x=1\}$ have a unique infinite
connected component, the infinite cluster; see e.g.\ \cite{G}.
Here points $x,y$ are thought
of as connected if there exist points $x_0,x_1, \dots, x_n$ in the
above random sets such that $x_0=x, x_n=y$ and
$|x_i-x_{i+1}|_1=1$ 
for all $i=0,\dots, n-1$. We call
$C_\infty$
the
infinite clusters in $\{x \in
\bbZ^d\,:\, \s_x=1\}$, and $C_\infty ^\ast$ the infinite clusters in
$\{x \in \bbZ^d\,:\, \vartheta  _x=1\}$.
By taking $p$ large,  from the domination assumptions (H1) we also know that there exists $c<\infty$ such that $\bbP$--a.s.\ the holes of $C_\infty$ and $C^*_\infty$ intersecting the box $[-n,n]^d$ have diameter at most $c\log n$ \cite[Prop. 2.3]{BP} (in particular, all holes have finite cardinality).
 Finally, we define
\[
 \cC_\infty= \cup _{x \in C_\infty} B(x)\,, \qquad
 \cC_\infty ^\ast = \cup _{x \in C^\ast_\infty} B(x)\,.
\]
The dependence on the parameters $K,T_0$ is understood. The set
$\cC_\infty ^\ast$ is often referred to as the cluster of white
boxes, while $\cC_\infty$ is called the cluster of occupied
points. Clearly, $\cC_\infty \supset \cC_\infty ^\ast$. The points $\xi\cap\cC^*_\infty$ are sometimes
referred to as the good points.

Given a starting point in  $ \xi \cap \cC^\ast _\infty$, the random
walk $Y_n$ is the discrete--time  random walk made by the
consecutive visits of $X_n$ to the set $\xi \cap \cC^\ast_\infty$:
setting
\begin{equation}\label{tao}
T_1\coloneq\min \{n\geq 1\,\colon X_n \in \xi \cap
\cC^\ast_\infty\}\,,
\end{equation}
 the transition probability from
$x$ to $y$ of $Y$ is given by
\begin{equation}\label{eq_omegaxy}
\o_{x,y}(\xi) \coloneq P_{x,\xi}( X_{T_1} =y)\,.
\end{equation}
Thus $Y_n=X_{T_n}$, where $T_n$ is the time of the $n$-th visit to
$\xi \cap \cC^\ast_\infty$.
A poissonization of $Y_n$ yields the
continuous--time random walk $\wt Y_t$. Equivalently, $\wt Y_t$ is the
continuous--time Markov chain on $ \xi \cap \cC^\ast_\infty$ whose
infinitesimal generator is given by
\begin{equation}\label{restrictedGenerator}
 \bbL_* f (x) = \sum_{y\in\xi\cap\cC^\ast_\infty}
 \o_{x,y}(\xi) \bigl ( f(y) - f(x) \bigr)\,.
\end{equation}
In order to simplify the notation, we simply write $Y_t $ instead of
$\wt Y_t$  when no confusion can be generated. It is simple to check
that $( w(x) : x \in \xi \cap \cC^*_\infty)$ is a reversible measure
both for $Y_n$ and for $Y_t$.

Following \cite{BP},
a crucial step towards the proof of Theorem \ref{piccolo}
consists
in establishing the following bounds on the distance and heat kernel of the
restricted walk.

\begin{proposition}\label{mirta}
For a deterministic sequence $b_n=o(n^2)$ and for $\bbP$--a.a.\ $\xi$:
\begin{align}
 &
 \limsup_{ n \to \infty} \max_{\substack{ x \in \xi \cap \cC^\ast_\infty : \\
 |x|_\infty \leq n  }} \sup _{t \geq n}  t^{d/2} P_{x,\xi}
 (Y_t=x)<\infty
 \,.\label{mirtillo2}\\
& \limsup_{n \to \infty } \max_{\substack{ x \in \xi \cap \cC^\ast_\infty  : \\
 |x|_\infty \leq n  }} \sup _{t \geq b_n} \frac{ E_{x,\xi} |Y_t-x|}{
 \sqrt{t} } <\infty
\,,\label{mirtillo1}
 \end{align}
\end{proposition}
Before going to the proof of Proposition \ref{mirta}, which is given
in Section \ref{refrigerio} and
Section \ref{distante},  we start by developing some tools that will
be repeatedly used in the sequel.

\subsection{Enlargement of holes}
\la{enlarge}
 Connected components in the complement of $ C^*_\infty$ and in the complement of $\cC^*_\infty$
are called \emph{discrete holes} and \emph{holes}, respectively.
 A generic discrete hole  $C$  is thus a
finite set, while a subset $C'\subset \bbR^d$ is a hole if and only
if it  can be written as $ C'= B(z_1)\cup\cdots\cup B(z_m) $, where
$\{z_1, \dots, z_m\}$ is a discrete hole.
%
For the moment we restrict our analysis to discrete holes.

Given $z \in \bbZ^d$, we
 use $G(z)$ to denote the unique discrete hole $C$
such that $z\in C$.
For a vertex $z\in C^*_\infty$, we set $G(z)=\emptyset$. Let
$d_1(z,z')=|z-z'|_1$, $z,z'\in\bbZ^d$, denote the $\ell_1$ distance,
i.e.\ the graph distance in $\bbZ^d$. Also, we use
$d_2(z,z')=|z-z'|$ for the usual $\ell_2$ distance. Given an
arbitrary $D\subset\bbZ^d$ and $i=1,2$, we write $d_i(z,D)$ for the
point-to-set $\ell_i$ distance and $\diam_i(D)=\sup_{z,z'\in D}
d_i(z,z')$ for the $\ell_i$ diameter of $D$.
We write $|D|$ for the cardinality of $D$.

The \emph{enlargement} of $C$ is given by the set \be\la{enlar} \wt C
= \{ z\in\bbZ^d: d_2(z,C)\leq \diam_2(C)\}\,.
\end{equation}
 Note that an enlarged discrete hole $\wt C$ will in
general
contain some vertices $z\in C_\infty^*$, 
and a vertex $z\in C^*_\infty$ can be covered by
several enlarged holes. When $z\notin C^*_\infty$, we use the
notation $\wt G(z):=\wt{G(z)}$ for the enlargement of $G(z)$. When
$z\in C^*_\infty$, we set $\wt G(z)=\emptyset$.


Two 
vertices $z,z'\in\bbZ^d$ are said to be \emph{related} if
they both belong to some enlarged discrete hole $\wt C$.
This notion induces in an
obvious way an equivalence relation between 
vertices:
two 
vertices $z,z'$ are \emph{equivalent} (written $z \sim z'$)  if and only if there exist 
vertices $z_0,\dots,z_n$ such that $z_0=z,z_n=z'$, and $z_i,z_{i+1}$
are related for all $i=0,\dots,n-1$.
Consider now the graph obtained from $\bbZ^d$ identifying
all equivalent vertices. Call $\bar d(z,z')$ the associated graph
distance (each vertex is at distance $0$ from any member of its
equivalence class).
Note that according to this definition,
two distinct vertices $z,z'\in C^*_\infty$ may well have distance
$0$ (if there exists a nearest neighbor path $\g$ connecting $z,z'$
such that $\g$ is fully contained in the union of all enlarged
holes).

Clearly, $\bar d(z,z')\leq d_1(z,z')$ for any pair of vertices. Our
assumptions allow us to compare the two distances in the opposite
direction as well.

\begin{proposition}\la{dist0}
For all $a>0$, there exist $K,T_0$ 
such that for all $z\in\bbZ^d$:
 \begin{equation}\la{probdbar}
  \bbP\Big (\bar d(0,z)
\leq \frac12\,d_1(0,z)\Big)\,\leq\, \nep{-a\,d_1(0,z)}\,.
 \end{equation}
\end{proposition}
\begin{proof}
Due to assumption (H1) we can find $K,T_0$ such that  the field of
white points $\vartheta$ dominates a supercritical Bernoulli field
$Z(p)$ with parameter $p$. Therefore,   the probability appearing in
\eqref{probdbar} is smaller than $ \bbP_p\bigl(\bar d(0,z) \leq
\frac12\,d_1(0,z)\bigr)$, where $\bbP_p$ is the law of the Bernoulli
random field $Z(p)$ (and $\bar d(0,z)$ is accordingly defined as a
function of $Z(p)$ and its unique infinite cluster instead of
$\vartheta$ and $C_*^\infty$, respectively).

Let us first observe that if $\bar d(0,z)=0$ then there exist
discrete  holes $C_1,\dots,C_m$ such that the union of the enlarged
discrete holes $\wt
C_1,\dots,\wt C_m$ 
contains a nearest neighbor path from $0$ to $z$. In particular,
there must exists a nearest neighbor path $(z_0=0,\dots,z_n=z)$ in
$\bbZ^d$ of length $n\geq d_1(0,z)$ such that
\[
n\leq \sum_{i=0}^n\diam_1(\tilde G(z_i))1_{\{z_i\notin
G(z_j)\,,\;\forall
  j<i\}} \,.
\]
More generally, by pasting together different paths as in the
example above, one obtains that the event $\bar
d(0,z)\leq \frac12\,d_1(0,z)$ is contained in the event: there exist
$n\geq d_1(0,z)$ and a nearest neighbor path
$\g_n=(z_0=0,\dots,z_n=z)$ in $\bbZ^d$ such that
\[
 \frac{n}2\,
 \leq X(\g_n)
 :=\sum_{i=0}^n\diam_1(\tilde G(z_i))1_{\{z_i\notin G(z_j)\,,\;\forall j<i\}\,.}
\]
Therefore, using the exponential Chebyshev estimate and a union
bound, for any $\l>0$:
\[
 \bbP_p\Big(\bar d(0,z)\leq \frac12\,d_1(0,z)\Big)
 \leq \sum_{n\geq d_1(0,z)} \sum_{\g_n}
 \nep{-\l\,n/2}\,\bbE_p[\nep{\l  X(\g_n)}]\,,
\]
where $\bbE_p$ denotes expectation w.r.t.\ $\bbP_p$. We claim that
for every $\l>0$ there exists $p\in(0,1)$
such that \be\la{clams}
 \bbE_p[\nep{\l  X(\g_n)}]\leq 2^n\,,
\end{equation}
for all nearest neighbor paths $\g_n$ of length $n$. Once
\eqref{clams} is proved, estimating by $(2d)^n$ the total number of
paths $\g_n$ connecting $0$ and $z$, one obtains that for all $a>0$,
there exist suitable constants $\l>0$ and $p\in(0,1)$
such that
\[
 \bbP_p\Big (\bar d(0,z)\leq \frac12\,d_1(0,z)\Big)
 \leq \sum_{n\geq d_1(0,z)} (4d)^n
 \nep{-\l\,n/2} \leq \nep{-a\,d_1(0,z)}\,,
\]
and the proposition follows.

We turn to the proof of \eqref{clams}. Let the path
$\g_n=(z_0=0,\dots,z_n=z)$ be fixed. Let $\cF_i$ denote the
$\si$--algebra generated by the random variables
$G(z_0),\dots,G(z_i)$. To prove \eqref{clams}, it is sufficient to
establish the uniform estimate: For any $\l>0$ there exists
$p\in(0,1)$ such that
\begin{equation}\la{lapla}
 \bbE_p \Big[
 \exp\bigl\{\l\,\diam_1(\wt G(z_i))1_{\{z_i\notin G(z_j)\,,\forall j<i\}}\bigr\}
 \,\bigl|\, \cF_{i-1}\Big]
 \leq 2\,.
\end{equation}
Note that the definition \eqref{enlar} implies that $\diam_1\wt
G(0)\leq c\diam_1 G(0)$, for some finite constant $c=c(d)$.
Therefore, it suffices to show that for any $\l>0$ there exists
$p\in(0,1)$ such that \eqref{lapla} holds with $\wt G(z_i)$ replaced
by $G(z_i)$. At this point the conclusion follows from a standard
Peierls argument, as in \cite[Lemma 3.1; proof of Eq.\ (3.12)]{BP}.
\end{proof}


\bigskip

We extend the definition of $\bar d$ to all points in
the process $\xi$ using the corresponding $K$-boxes. Namely,
for any $x\in\xi$, let $z(x)$ denote
the unique point of $\bbZ^d$ such that $x\in B(z(x))$. Then, we set
\be\la{dbardef} \bar d(x,y) := \bar d(z(x),z(y))\,,\quad    \;
x,y\in\xi\,.
\end{equation}
The next estimate is a useful corollary of Proposition
\ref{dist0}.

\begin{corollary}
\la{coroscaz} Take $K, T_0$ satisfying \eqref{probdbar} for some
$a>0$. Then, $\bbP$--a.s.\ there exists $\kappa =
\kappa(\xi,K)<\infty$ such that for all $x,y\in\xi$: \be\la{scaz1}
|x-y|\leq \kappa\(1+\log(1+|x|)+ \bar d(x,y)\)\,.
\end{equation}
\end{corollary}
\begin{proof}
From the definition \eqref{dbardef}, and the fact that
\[ K|z(x)-z(y)|-c_1(K)\leq |x-y|\leq K|z(x)-z(y)|+c_1(K)\,, \]
for some
constant $c_1(K)$, we see that it suffices to prove for $\bbZ^d$  an
estimate similar to \eqref{scaz1}: $\bbP$--a.s.\ there exists $\kappa
= \kappa(\xi)<\infty$ such that for all $z,z'\in\bbZ^d$:
\be\la{scaz} |z-z'|\leq \kappa\(1+\log(1+|z|)+ \bar d(z,z')\)\,.
\end{equation}
Combining the estimate of Proposition \ref{dist0} and the Borel
Cantelli argument shows that $\bbP$--a.s.\ there exists
$n_0=n_0(\xi)<\infty$ such that whenever $n\geq n_0$, $|z|\leq n$,
then \be\la{n0n} |z-z'|\leq \frac12\,\bar d(z,z')\,,\quad \text{for
all}\; |z-z'|\geq \log n\,.
\end{equation}
Let us verify that this implies \eqref{scaz}.
We may suppose first that $|z-z'|\geq \log (1+|z|)$.
If $|z|\geq n_0$ then we may take $n=\lceil |z|\rceil$ and
the claim follows from \eqref{n0n}.
Thus, assume that $|z|\leq n_0$.
Clearly, we may further assume that $|z'|>n_0$,
since otherwise $|z-z'|\leq 2n_0$.
Therefore, if $|z-z'|\geq\log (1+|z'|)$ the claim follows from \eqref{n0n} by
taking $n=\lceil |z'|\rceil$ and exchanging the roles of $z$ and $z'$.
In conclusion, the only case remaining is when
$|z|\leq n_0, |z'| > n_0$ and $|z-z'|< \log (1+|z'|)$.
\ignore{
But if $t_0$
is such that $\log(1+t)\leq t/2$ for all $t\geq t_0$, then we must
have either $|z-z'|\leq t_0$ (and in this case \eqref{scaz} is
fulfilled) or
\[
 |z'|\leq |z|+ |z-z'| \leq  |z|+\log (1+|z-z'|) \leq  |z|+ |z-z'|/2
 \leq n_0 + |z'|/2\,,
\]
which implies $|z'|\leq 2n_0$. Therefore,
$|z-z'|\leq |z|+|z'|\leq 3n_0
+ t_0$. This concludes the proof of \eqref{scaz}.
}
But if $t_0$
is such that $\log(1+t)\leq t/2$ for all $t\geq t_0$, then we must
have either $|z'|\leq t_0$ (and in this case $|z-z'|\leq n_0+t_0$) or
\[
 |z'|\leq |z|+ |z-z'| \leq  |z|+\log (1+|z'|) \leq  |z|+ |z'|/2\,,
\]
which implies $|z'|\leq 2n_0$. Therefore,
$|z-z'|\leq |z|+|z'|\leq 3n_0$. This concludes the proof of \eqref{scaz}.
\end{proof}

\subsection{Some uniform estimates}
The enlarged discrete holes allow us to obtain some useful estimates
that we collect here. It is first convenient to extend our notation.
Consider the map $\Psi (A)$ defined on finite subsets $A\subset
\bbZ^d$ as $\Psi (A)= \cup _{z \in A} B(z) $. Then, given a hole $C=
\Psi (C')$ ($C'\subset \bbZ^d$ being a discrete hole), its
enlargement $\wt C$ is defined as $\Psi ( \wt C') $. Given $x \not
\in \cC_\infty ^*$ we write now  $G(x)$ for the unique hole $C$
containing $x$. If $x \in \cC_\infty ^*$, we set $G(x)=\emptyset$.

Given a hole $C$ we call $[C]$ - the \emph{class} of $C$ - the union
of all holes $C'$ such that $\bar d(C,C')=0$, i.e.\ such that there
exists a chain of holes $C=C_0,\dots,C_m=C'$ such that $\wt C_i\cap
\wt C_{i+1}\neq\emptyset$. We stress that $[C]$ is not the family of
points $x \in \xi$ such that $ \bar d (x, C)=0$, in particular
$ [C]\cap \cC_\infty^*=\emptyset$.

Finally, we define
\[
\G_n=\sum_{i=1}^{n} \bbI_{\{X_i\notin
\cC_\infty\cup[G(X_{i-1})]\}}\,.
\]
$\G_n$ represents the number of jumps into a new class of holes up
to time $n$ (to be distinguished from the number of different
classes of holes visited up to time $n$). As an example, suppose
$X_0\in\cC^*_\infty, X_1\notin\cC^*_\infty$ and
$X_2\in\cC^*_\infty$. In this case, $\G_2 = \G_1 = 1$.

Let \be\la{tt11} T_1=\inf \{n\geq 1:\; X_n\in\cC^*_\infty\}\,,
\end{equation}
and call $\G=\G_{T_1}$ the number of jumps into new classes
before the return to $\cC_\infty^*$.
\begin{lemma}\la{lecN}
There exists $\d>0$ such that uniformly in $\xi\in\cN$, $x\in\xi$:
\be\la{phip} P_{x,\xi}(\G\geq k)\leq (1-\d)^k\,,\quad \;
k\in\bbN\,.
\end{equation}
\end{lemma}
\proof
Define the times at which the walk jumps out of a class of
holes: \be\la{tauii} \t_i=\inf\{n>\t_{i-1}: \; X_n\notin
[G(X_{\t_{i-1}})]\}\,,\quad \; \t_0=0\,.
\end{equation}
Note that jumps within the cluster are included. From the strong
Markov property applied to the stopping times $\t_i$, the claim
\eqref{phip} follows from the estimate: for some $\d>0$, for all
$i$, uniformly in $\xi$ it holds $
P_{X_{\t_i},\xi}\left(X_{\t_{i+1}}\in\cC^*_\infty \right)\geq \d$.
Thus, it suffices to show that for some $\d>0$, uniformly in $\xi$
and $x\in\xi$: \be\la{topre} P_{x,\xi}\left(X_1\in\cC^*_\infty\tc
X_1\notin [G(X_0)]\right)\geq \d\,.
\end{equation}
If $x=X_0\in\cC^*_\infty$ the bound \eqref{topre} is easy: here
$[G(X_0)]=\emptyset$ and all we have to show is that
$P_{x,\xi}\left(X_1\in\cC^*_\infty\right)\geq \d$; this follows
from the fact that $w(x)\leq T(K,T_0)$ (see Lemma \ref{leT}) and
$r(x,y)\geq \d_1=\d_1(K)$ for some $y\in\cC^*_\infty$, so that, with
e.g.\ $\d=\d_1/T$.
\[
P_{x,\xi}\left(X_1\in\cC^*_\infty\right)\geq
p(x,y)=\frac{r(x,y)}{w(x)}\geq \d\,.
\]

If $x=X_0\notin\cC^*_\infty$, then $x\in C$ for some hole
$C=G(X_0)$. Then
\[
P_{x,\xi} \left(X_1\in\cC^*_\infty\tc X_1\notin [G(X_0)]\right)
=  \frac{P_{x,\xi} ( X_1 \in \cC_\infty^* )}{P_{x,\xi}
(X_1\in\cC^*_\infty)+P_{x,\xi} (X_1\notin\cC^*_\infty \cup [C])}
\]
It is sufficient to prove that uniformly, for some $\d>0$:
\be\la{topre1} P_{x,\xi} (X_1\notin\cC^*_\infty \cup [C]) \leq
\d^{-1}\,P_{x,\xi} (X_1\in\cC^*_\infty)\,.
\end{equation}
To prove \eqref{topre1} we write \be\la{ppto} P_{x,\xi}
(X_1\notin\cC^*_\infty \cup [C])= \sum_{C':\;C'\notin
[C]}\sum_{z'\in\bbZ^d:B(z')\subset C'} P_{x,\xi}(X_1\in B(z'))\,,
\end{equation}
where the sum is over holes $C'$ in a different class than $C$. Let
$y$ denote the closest point $y\in\xi\cap\cC^*_\infty$ to $x$. Note
that $P_{x,\xi} (X_1\in\cC^*_\infty) \geq p(x,y) = r(y-x)/w(x)$.

If $z,\zeta\in\bbZ^d$ denote the vertices such that $x\in B(z)$ and
$y\in B(\zeta)$, then $|x-y|\geq K|z-\zeta| - c_1(K)$.
By construction, if $x'\in B(z')$ with $z'\in C'$ for some $C'\notin [C]$,
then $|x-x'|\geq K|z-z'| - c_1(K)$ and  $|z-z'|\geq 2|z-\zeta|-2$.
To justify the last inequality, let $C= \Psi ( \underline{C})$,
$\underline{C}$ being  a discrete hole in $\bbZ^d$. Since $C'\notin
[C]$ and  therefore $ z' \not \in \widetilde{\underline{C}}$, by
\eqref{enlar} we conclude that
\begin{equation}\label{santiago}
|z-z'|\geq |z-\zeta|+d_2 (z',\underline{C})>
2\text{diam}_2(\underline{C})\geq
2|z-\z|-1 \,,
\end{equation}
 thus concluding the proof of our claim.
 In addition, we observe
that
 \be\la{bnz} n_{z'}\leq \nep{|z-z'|^{\a/2}}\,,
\end{equation}
since otherwise
$z\in Q(z',R_{z'})$, cf.\ \eqref{rzs}, and $z,z'$ would
not belong to distinct classes of holes.

Define the function $\varphi:(0,\infty)\times\bbN \to(0,\infty)$ as
$ \varphi(a,m)=\sum_{v\in\bbZ^d:\;|v|\geq m} \nep{-a\,|v|^\a} $. It
is not hard to check that, for all fixed $a>0,\a>0$, as
$m\to\infty$:
\begin{equation}\la{boub}
\varphi(a,m) = O(\nep{-a\,m^\a}m^{d-\a})\,.
\end{equation}

To bound \eqref{ppto} we write, for all $\e>0$:
\begin{multline}
w (x)\,P_{x,\xi} (X_1\notin\cC^* _\infty \cup [C]) \leq
c(K)\sum_{z'\in\bbZ^d:\;|z-z'|\geq  2|z-\zeta|-2}
n_{z'}\nep{-K^\a|z-z'|^\a}  \\
\leq c'(\e,K) \sum_{v\in\bbZ^d:\;|v|\geq  2|z-\zeta|-2}
\nep{-(1-\e)\,K^\a\,|v|^\a}  = c'(\e,K)\varphi((1-\e)K^\a,
2|z-\zeta|-2)\,, \la{ppto1}
\end{multline}
where we have used \eqref{bnz}. On the other hand
\begin{equation}\la{ppto2}
w(x)\,P_{x,\xi} (X_1\in\cC^* _\infty)\geq r (y-x)\geq
c''(K)\,\nep{-K^\a\,|z-\zeta|^\a}\,.
\end{equation}
From \eqref{boub}, for all $\a,K>0$, taking e.g.\ $\e>0$ such that
$(1-\e)2^\a>1$, the ratio
\[
\frac{\varphi((1-\e)K^\a,2|z-\z|-2)}{\nep{-K^\a\,|z-\zeta|^\a}}
\]
is uniformly bounded in $z,\zeta\in\bbZ^d$. Using \eqref{ppto1} and
\eqref{ppto2}, this proves \eqref{topre}. \qed

\begin{lemma}\la{lecda}
For $c,\g\geq 1$, set $u_{\g,c}(t)=t^\g\exp[c(\log (t+1))^\a]$.
Uniformly in $\xi$ and $x\in\xi$:
\be\la{phipo} E_{x,\xi} \Big[\sup_{1\leq j\leq T_1} u_{\g,c}(\bar
d(x,X_j))\Big] <\infty\,.
\end{equation}
\end{lemma}
\proof
With the definition \eqref{tauii}
we have that $\G = \max\{n\geq 0:\; \t_n < T_1\}$ and
\be\la{lecda1} \sup_{1\leq j\leq T_1} \bar d(x,X_j)\leq
\sum_{i=0}^\G \bar d(X_{\t_i},X_{\t_{i+1}})\,,
\end{equation}
where $\G$ is defined as in Lemma \ref{lecN}.

Let us first suppose that $\a>1$. In this case it is sufficient to
show the claim with $u_{\g,c}(t)$ replaced by
$\exp[c(\log(t+1))^\a]$, which is a convex function. Since, for
every $N\in\bbN$:
\[
\log\Bigl [1+\sum_{i=0}^N \bar d(X_{\t_i},X_{\t_{i+1}})\Bigr] \leq
\log (N +1)+ \log\Big[\frac1{N+1}\sum_{i=0}^N (1+ \bar
d(X_{\t_i},X_{\t_{i+1}}))\Big]\,,
\]
simple estimates yield
\begin{multline}
\exp\Big\{2c\Big(\log \Big[1+\sum_{i=0}^N \bar
d(X_{\t_i},X_{\t_{i+1}})\Big]\Big)^\a \Big\}
\\
\leq \exp\Big\{ c_1(\log (N+1))^\a
\Big\}\,\frac1{N+1}\sum_{i=0}^N\exp\Big\{ c_1(\log[1+ \bar
d(X_{\t_i},X_{\t_{i+1}})] )^\a \Big\}\,, \la{lecda2}
\end{multline}
for some constant $c_1=c_1(\a,c)$. Suppose that, for some constant
$c_2$, uniformly in $\xi$ and $x\in\xi$ and $i\in\bbN$
\be\la{suplec} E_{x,\xi} \exp\left\{ c_1(\log[1+ \bar
d(X_{\t_i},X_{\t_{i+1}})] )^\a \right\}\leq c_2\,.
\end{equation}
Then, taking expectation in \eqref{lecda2}, using Schwarz'
inequality:
\begin{align*}
E_{x,\xi} & \Big[\exp\Big\{c\Big(\log \Big[1+\sum_{i=0}^\G \bar
d(X_{\t_i},X_{\t_{i+1}})\Big]\Big)^\a \Big\}\Big]
\\
&\leq \sum_{N=0}^\infty P_{x,\xi}(\G=N)^\frac12 \,E_{x,\xi}\Big[
\exp\Big\{2c\Big(\log \Big[1+\sum_{i=0}^N \bar
d(X_{\t_i},X_{\t_{i+1}})\Big]\Big)^\a \Big\} \Big]^\frac12
\\
& \leq \sqrt{c_2} \sum_{N=0}^\infty P_{x,\xi}(\G=N)^\frac12
\exp\Big\{ c_1(\log (N+1))^\a \Big\}\,.
\end{align*}
The last sum  above is uniformly finite by Lemma \ref{lecN}. It
remains to show the validity of \eqref{suplec}. To this end, it
suffices to show  that, for some constant $c_2$, uniformly in $\xi$
and $x\in\xi$: \be\la{suplec2} E_{x,\xi} \exp\left\{ c_1(\log[1+
\bar d(x,X_{\t_1})] )^\a \right\}\leq c_2\,.
\end{equation}
By summing over all possible ways of jumping out of the starting
class of holes $[G(x)]$ one obtains that \eqref{suplec2} follows
from \be\la{suplec3} E_{x,\xi} \left[\exp\left\{ c_1(\log[1+ \bar
d(x,X_1)] )^\a \right\}\,\tc\,X_1\notin [G(x)]\right] \leq c_2\,.
\end{equation}
Let $x,y$ and $z,\zeta$ be as  in the proof of Lemma \ref{lecN},
i.e.\ $x \in B(z)$, $y \in B(\z)$ and $y$ is the closest point $y \in
\xi \cap \cC_\infty^*$ to $x$. As in \eqref{ppto2} we have
\be\la{asinp}P_{x,\xi} (X_1\notin [G(x)])\geq P_{x,\xi}
(X_1\in\cC_\infty) \geq c_3\,\frac{\nep{-K^\a\,|z-\zeta|^\a}}{w
(x)}\,,
\end{equation}
for some constant $c_3$. Let $z'\in\bbZ^d$ be such that $X_1$ is in
the $K$-box $B(z')$. Note that $\bar d(x,X_1)=0$ for all $z'$ such
that $|z-z'|\leq 2|z-\zeta|-2$. For other values of $X_1$ we simply
bound $\bar d(x,X_1)\leq c_4|z-z'|$, for some $c_4=c_4(d,K,\a)$.
Reasoning as in \eqref{ppto1}, to bound \eqref{suplec3} we write,
for all $\e>0$:
\begin{align}
&w_\a(x)\,E_{x,\xi}\left[ \exp\left\{ c_1(\log[1+ \bar d(x,X_1)]
)^\a \right\}\,;\,X_1\notin [G(x)]
\right] \nonumber\\
&\qquad\leq 1+ c(K) \sumtwo{z'\in\bbZ^d:}{|z-z'|\geq 2|z-\zeta| -2}
n_{z'}\nep{-K^\a|z-z'|^\a} \,\nep{c_5(\log [1+|z-z'|])^\a} \nonumber
\\ &\qquad
\leq 1+c'(\e,K) \sumtwo{v\in\bbZ^d:}{|v|\geq 2|z-\zeta|-2}
\nep{-(1-\e)\,K^\a\,|v|^\a} \nonumber
\\
 &\qquad= 1+c'(\e,K)\varphi((1-\e)K^\a,2|z-\zeta|-2)\,,
\la{ppto15}
\end{align}
where we have used \eqref{bnz}. From \eqref{ppto15} and
\eqref{asinp} we can conclude, as in the proof of Lemma \ref{lecN},
that the left hand side of \eqref{suplec3} is uniformly bounded.
This ends the proof of the case $\a>1$.

To prove the claim for $\a\leq 1$, observe that it is sufficient to
prove the estimate \eqref{phipo} with $u_{\g,c}(t)$ replaced by
$t^\g$. Here $\g\geq 1$ and $t^\g$ is convex. Thus,
$$
E_{x,\xi}\Big[\Big(\sum_{i=0}^\G\bar
d(X_{\t_i},X_{\t_{i+1}})\Big)^\g\Big] \leq \sum_{N=0}^\infty
P_{x,\xi}(\G=N)^\frac12 \, N^\g \Big[\frac1N \sum_{i=0}^N
E_{x,\xi}\Big( \bar d(X_{\t_i},X_{\t_{i+1}})^{2\g}\Big)
\Big]^\frac12\,.
$$
A uniform estimate of the expectation in the right hand side above
can be obtained exactly as in the proof of \eqref{suplec}. This
ends the proof.
\qed

\medskip

We turn 
to a simple corollary of our previous
results.
\begin{lemma}\la{lecda20}
For every $p\geq 1$, there exists $c>0$
such that uniformly in $\xi$ and $x\in\xi$, and for
all $n$:
\be\la{phiposa} E_{x,\xi} \Big[\sup_{1\leq j\leq T_n} \bar
d(x,X_j)^p\Big] \leq c\, n\,.
\end{equation}
\end{lemma}
\begin{proof}
Setting $\D_i:=\sup_{1\leq j\leq T_{i}-T_{i-1}}
\bar d(X_{T_{i-1}},X_{T_{i-1}+j})$, we have
$$
\sup_{1\leq j\leq T_n}
\bar d(x,X_j)^p\leq c\sum_{i=1}^n \D_i^p\,,
$$
for some constant $c=c(p)$.
The strong Markov property at time $T_{i-1}$ together with the
uniform estimate of Lemma \ref{lecda} imply that for some $c'$,
for all $x$ and $i$ one has $E_{x,\xi} \left[\D_i^p\right]\leq
c'$.
\end{proof}

\subsection{Some almost sure estimates}
We describe some more consequences of the estimates derived above.
\begin{lemma}\label{moltobuono}
$\bbP$--a.s.\,,  for every $p\geq 1$, there exists $\k=\k(p,\xi)<\infty$
such that for all $x\in\xi$, and for
all $n\in\bbN$:
\begin{equation}\label{fifona}
 E_{x,\xi} \Big[ \sup_{1\leq j \leq T_n} |X_j-x  |^ p \Big]\leq \k \bigl[1+\log (1+|x|)\bigr]^p
 n\,.
\end{equation}
\end{lemma}
\begin{proof}
Using Corollary \ref{coroscaz}, we can write
\[
 |X_j-x|^p
 \leq  \k\,\bigl[ 1 + \log ( 1 + |x|) \bigr]^p +\k\, \bar d (X_j , x)^p\,.
\]
The conclusion then  follows from Lemma \ref{lecda20}.
\end{proof}

\begin{lemma}\la{t1lemma}  Take $K, T_0$ as in Proposition
  \ref{dist0}.
Then,
$\bbP$-almost surely, for all $x\in\xi$ 
$P_{x,\xi}(T_1<\infty)=1$.
\end{lemma}
\begin{proof}
For every $m,n$, write
\[
P_{x,\xi}(T_1>m)=P_{x,\xi}(T_1>m\,,\;S_{n}>m) +
P_{x,\xi}(T_1>m\,,\;S_{n}\leq m)\,,
\]
where $S_n$ denotes the first time $k\geq 1$ such that
$|X_k|_\infty\geq n$.

  Given $y \in \xi $, let
$z(y) \in \bbZ^d$ be the unique point such that $y \in B(z)$. Fix
$v\in\xi \setminus \cC^*_\infty$ such that $|z(v)|\leq n$. Then
 $z(v)$ is in a discrete hole (i.e.\ $z \not \in C_\infty^*$).
We know that $\vartheta$ dominates a supercritical Bernoulli field
with large parameter $p$, and for the latter it is well known that
a.s.\ holes intersecting $[-n,n]^d$ have diameter at most $O(\log n
)$ (see \cite{BP}[Prop.\ 2.3]). By the stochastic domination, the
same property still holds for the the discrete holes, which are the
holes in $\vartheta$. We claim that
$w(v) \leq c' \nep{c(\log n)^{\frac{\a}2}}$.
To prove our claim, we write
\[
w(v)=\sumtwo{y\in\xi:}{|z(y)-z(v)|\leq n} r(y-v)
 + \sumtwo{y\in \xi:}{|z(y)-z(v)|>n} r(y-x)\,.
\]
For every $z\in \bbZ^d$ such that $|z-z(v)|\leq n$, if $n_z \geq
T_0$ (i.e.\ if $B(z)$ is overcrowded) then it must be  $n_z\leq
\nep{c(\log n)^{\frac{\a}2}}$ because a hole around $v$ can have
diameter at most $O(\log n)$. Taking $n$ large enough that $ T_0
\leq \nep{c(\log n)^{\frac{\a}2}}$ we have
\[
\sumtwo{y\in \xi :}{|z(y)-z(v)|\leq n} r(y-v)
\leq \sumtwo{z\in\bbZ^d\,:}{|z-z(v) |\leq n }
\nep{c(\log n)^{\frac{\a}2}}\,\nep{-c_1(K)|z-z(x)|^\a}
\leq c_2 \nep{c(\log n)^{\frac{\a}2}}\,.
\]
On the other hand, if $|z(y)-z(v)|>n$, then $z(y)$ cannot belong to
the same hole of $z(v)$, and therefore $n_z\leq
\nep{|z(y)-z(v)|^{\frac{\a}2}}$ whenever $n_z \geq T_0$.
It follows that
\[
\sumtwo{y\in \xi :}{|z(y)-z(v)|> n} r (y-v)
\leq \sumtwo{z\in \bbZ^d:}{|z-z(v)|> n} T_0 \nep{|z-z(v)|^{\frac{\a}2}}\,
\nep{-c_1(K)|z-z(v)|^\a}
\leq c_3\,.
\]
The above estimates trivially imply our claim.

Due to this claim and using again the fact that the hole containing
$v$ has diameter at most $O(\log n)$  we can estimate
\[
 P_{v,\xi} ( X_1 \in \cC_\infty^*)
 \geq \frac{e^{- c_4 (\log n)^\a }}{w(v)}
 \geq  e^{- c_5 (\log n )^\a }\,.
\]
 From this observation we infer that
\begin{equation}
P_{x,\xi}(T_1>m\,,\; S_{n}>m) \leq \(1-\nep{-c_5(\log n)^\a}\)^m
\leq \exp{\left[-m\,\nep{-c_5(\log n)^\a}\right]}\,. \la{ob3}
\end{equation}
On the other hand, for any $\g>0$, we can estimate by Markov's
inequality
\begin{equation}\label{bernini}
P_{x,\xi}(T_1>m\,,\;S_{n}\leq m)\leq P_{x,\xi}\Big(\sup_{1\leq j\leq
T_1}
  |X_j| > n\Big)
 \leq n^{-\g } E_{x,\xi}\Big(\sup_{1\leq j\leq T_1}
  |X_j| ^\g\Big) \,.
\end{equation}
By  Lemma \ref{moltobuono} we conclude that \be\la{ob4}
P_{x,\xi}(T_1>m\,,\;S_{n}\leq m) \leq c_x(\xi)\,n^{-\g}\,,
\end{equation}
for some $\bbP$--a.s.\ finite constant $c_x(\xi)$.
Taking $m,n=n(m)$ such that $n(m)\to\infty$ and
$m\exp[-(\log n(m))^\a]\to \infty$, as $m\to\infty$, \eqref{ob3} and
\eqref{ob4} imply the conclusion.
\end{proof}

\smallskip

\subsection{Harmonicity with respect to the restricted random walk}\label{amoreodio}
%
\begin{proposition}\label{pizza}
Let $\Phi (\xi,x)= x + \chi (\xi,x)$. Then for  $\bbP$--a.a.\ $\xi$ and for any $x \in \xi \cap \cC^*_\infty$:
\begin{equation}\label{calabrese}
\Phi( \t_x \xi, X_{T_1}-x)\in L^1 ( P_{x,\xi})\,, \qquad  E_{x,\xi}
\bigl( \Phi( \t_x \xi, X_{T_1}-x)\bigr)=0\,.
\end{equation}
\end{proposition}

\begin{proof}
We recall that $x_i + \chi_i(\xi, x) \in \cH_\nabla^\perp$, for each
coordinate $i$. Hence, by Lemma \ref{rino}, there exists a Borel
subset $\cA$ having  $\bbP$--probability $1$ such that  for all $\xi
\in \cA$ and for all $z \in \xi$, 
\[
 \sum _{y \in \xi} r(y-z)|\Phi ( \t_z \xi, y-z)|< \infty\,,
 \qquad \sum _{y \in \xi} r(y-z)\Phi ( \t_z \xi, y-z)=0\,.
\]
This implies that the process $(M^\xi _n)_{n\geq 0}$ defined in terms of
$(X_n)_{n \geq 0}$ as
\[
M_0^\xi=0\,, \qquad  M_n^\xi= \sum_{j=0}^{n-1} \Phi( \t_{X_j} \xi,
X_{j+1}-X_j)\;\; \text{ for } n \geq 1
\]
is a martingale w.r.t.\ $P_{x,\xi}$. By
shift covariance we have $M^\xi _n= \Phi ( \t_{X_0} \xi, X_n- X_0)$
for all $n\geq 0$. In particular, given $m \in \bbN$ and $\xi \in
\cA$, from the Optional Stopping Theorem, for any $m\in\bbN$,
we have that $\Phi( \t_x
\xi,X _{T_1 \wedge m }-x ) \in L^1 (P_{x,\xi})$ and
\begin{equation}\label{vaccinobis}
 E_{x,\xi}
 \Phi(\t_x \xi, X _{T_1 \wedge m }- x)  =
0  \,, \quad x\in\xi\,.
\end{equation}
 Since  $T_1$ is a.s.\ finite (see Lemma \ref{t1lemma}), we
 have
\[
\lim_{m\to\infty}\Phi(\t_x \xi,X_{T_1\wedge m} -x)  =
\Phi(\t_x\xi,X_{T_1}-x)\,,\qquad  P_{x,\xi}\text{--a.s}\,,
\]
In addition, using Lemma \ref{leles} below, we know that a.s.\
$|\Phi(\xi,y)|\leq c(\xi)\,u_{\g,c}\(|y|\)$ for suitable constants
$c,\g>0$ and $c(\xi)<\infty$.
We have
\begin{equation}\label{vitasnella}
\bigl| \Phi(\t_x \xi,X_{T_1\wedge m} -x)\bigr| \leq c_x(\xi)
\,u_{\g,c}\Big(\max_{1\leq j\leq T_1} |X_{j} -x|\Big)\,
, \qquad 
m \geq
1\,.
\end{equation}
Therefore, Corollary \ref{coroscaz} and Lemma \ref{lecda} allow us to use the
Dominated Convergence Theorem to conclude.
\end{proof}
Proposition \ref{pizza} shows that $\Phi(\t_x\xi,Y_n-x)$, with $Y_n=X_{T_n}$, is a martingale for every $x\in\xi\cap\cC_\infty^*$.
Since $Y_t=X_{T_{N_t}}$ for an independent Poisson process with mean $1$, Proposition \ref{pizza}
also implies 
\begin{corollary}\label{mariasole}
For $\bbP$--a.a.\ $\xi$, the process
$\bigl( \Phi( \t_z \xi, Y_t- z)\,:\, t \geq 0\bigr)$ is a
con\-tinuous--time martingale w.r.t.\ to the law of the  restricted random walk $Y_t$  starting at
$z \in \xi \cap \cC^*_\infty$.
\end{corollary}

\section{Heat kernel bound}\label{refrigerio}
In this section we prove the heat kernel bound \eqref{mirtillo2}.
In order to avoid confusion, in this section we restore the
convention to write $Y_n$ for the discrete--time restricted RW and
$\widetilde{Y}_t$ for the continuous--time restricted RW. The proof
of the heat kernel bound \eqref{mirtillo2} is divided in two parts:
in the first one we derive a  similar bound for a cut--off
restricted random walk  (see Proposition \ref{congiuntivite}) by
applying together the isoperimetric estimates of \cite{CF1} and the
method developed in \cite{MoPe}. In the second part
(see the proof of  Proposition
\ref{guerra_banane}), we show that the above cut-off gives an
approximation which is good enough to maintain
 the diffusive heat kernel bound.
In particular, \eqref{mirtillo2} follows immediately from Proposition
 \ref{guerra_banane}.
\subsection{Cut--off of the restricted random walk}
 We fix $L>0$ and  introduce the discrete--time RW $\bigl( X_n^{(L)}\,:\, n\geq 0\bigr)$ on $\xi_L:= \xi \cap [-L,L]^d$
jumping from $x$ to $y $ in $\xi_L$ with probability
\begin{equation} p^{(L)}(x,y)= \frac{r(y-x)}{w^{(L)}(x)} \,, \qquad w^{(L)} (x)= \sum
_{z \in \xi_L} r(z-x)\,.
\end{equation}
 We call $ C_L\subset \bbZ^d $ the largest connected component of  the field
 $\vartheta$ (defined in (\ref{blawhi}))
 inside $[-L, L]^d$. Then,  we set $\cC_L = \cup _{z \in C_L} B(z)$
 and
   $\z_L= \xi \cap \cC_L $.  Let $Y^{(L)}_n$ be the restricted
   random walk associated to $X^{(L)}_n$ when visiting the good points
   $\zeta_L$ (similarly to the definition of $Y_n$ as the  restricted random walk associated to
   $X_n$ when visiting the good points $\cC^*_\infty$).  
We define $\widetilde{Y}^{(L)}_t =
Y^{(L)}_{N_t}$ where $N_t$ is a Poisson process of parameter $1$,
independent of the random walk $Y^{(L)}_n$.
 The following heat kernel bound holds:
\begin{proposition}\label{congiuntivite}
Take $L=L(t)= t^{u}$ with $u>1/2$. Then for $\bbP$--a.a.\ $\xi$
\begin{equation}
\limsup_{t \to \infty}\, \max_{\substack{ x \in \z_L}}\, t^{d/2}    P_{x,\xi}  \bigl( \widetilde{Y}_t ^{(L)}=x\bigr)
<\infty\,.
\end{equation}
\end{proposition}
The fact that $L(t)$ is polynomial in $t$ in the above heat kernel
estimate is essential. Indeed, because of the $\max_{\substack{ x
\in \z_L}}$, one cannot expect the result to be true for functions
$L(t)$ with an exponential growth in $t$, since it would contradict
well known phenomena for the simple random walk on the supercritical
percolation cluster \cite{barlow}.
\begin{proof}
Clearly, $\( w^{(L)} (x)  \,,\, x \in \xi_L \)$ is a
reversible measure for the random walk $  X _n ^{(L)}$. Let
\[
 \o_{x,y}^{(L)}
 = P_{x,\xi}\bigl( Y^{(L)}_1=y \bigr)\,, \qquad x,y \in \z _L \,.
\]
Note that
$ w^{(L)} (x)   \o^{(L)}_{x,y} =  w^{(L)} (y)  \o^{(L)}_{y,x}$ for
all  $x,y \in \z_L $, i.e.\  $\bigl( w^{(L)}(x) \,,\, x \in \z_L )$
 is a reversible measure both for $Y^{(L)}_n$ and
for $\widetilde{Y}^{(L)}_t$ (recall that $\o^{(L)}_{x,y}$ coincides
also with the probability rate of  a jump of $\wt Y ^{(L)}$ from $x$
to $y$).

Let us denote  by $\p_L$ the measure $w^{(L)} (x)$ on $\z_L $ and
call $\varphi _L (t)$, $t>0$, the isoperimetric profile of the RW
$\widetilde{Y}^{(L)}$ w.r.t.\ $\p_L$:
\[
 \varphi _L(t):=
  \inf\Big\{I_U\,:\,U\subset\z_L,\;\p_L(U)\leq\bigl(t\wedge\frac12\bigr)\p_L(\z_L)\Big\}\,,
\]
where $
 I_U:=
 \p_L(U)^{-1}\sumtwo{x\in U,}{y\in\z_L\setminus U}
 \p_L(x)\o^{(L)}_{x,y}$.
 Note that due to the definition of $\z_L $
it holds
\begin{equation}\label{clowngiallo}
 1 \leq \p_L (x) \leq c \,, \qquad x \in
\z_L\end{equation} for some positive constant $c$ independent of
$\xi$ and $L$ (the upper bound follows from Lemma \ref{leT}, the
lower bound is trivial: $w^{(L)}(x)\geq r(0)$).

In order to estimate $P_{x,\xi} \bigl( \widetilde{Y}_t^{(L)}= x\bigr
)$, $x\in \z_L$, we apply Theorem 13 in \cite{MoPe} which states that,
given $\e,t >0$, if
\begin{equation}\label{MP1}
 t\geq \int _{4\p_L(x)/ \p_L(\z_L) }^{4/\e}\frac{8 du }{u \varphi _L^2 (u)}
\end{equation}
then
\begin{equation}\label{MP2}
 P_{x,\xi}\bigl ( \widetilde{Y}^{(L)}_t =x) \leq \frac{\p_L(x)}{\p_L(\z_L)} (
 1+\e)\,.
\end{equation}
To get a bound from below of the isoperimetric profile
$\varphi_L$ we observe that, given $x\not = y $ in $\z_L$,
\begin{equation}\label{neve}
\p_L(x)  \o _{x,y} ^{(L)} \geq \p_L(x) P_{x,\xi} ( X _1^{(L)}  =y ) =
 r(y-x) =  m_L(x) \frac{r(y-x)}{m _L (x) }\,,
\end{equation} where
$
 m _L (x) =\sum _{y\in \z_L} r(y-x) $.
Since $m_L(x) \leq w^{(L)}(x)$,  $m_L(x)$ satisfies a bound of the
same form of \eqref{clowngiallo}. In particular,
\begin{equation}\label{nevebis}
\frac{m_L(U)}{m_L(\z_L)} \leq \k \frac{ \p_L(U) }{\p_L (\z_L)}\,,
\qquad \forall U \subset \z_L \,.\end{equation} We stress that $\k$
is a   positive constant that does  not depend on   $\xi,L$.

Due to \eqref{neve} and \eqref{nevebis} we conclude that
\begin{equation}\label{pecorino}
\varphi _L ( t) \geq   \psi_L(\k t)\,, \qquad \forall t \in (0,1/2)
: \k t \in (0,1/2)\,,
\end{equation}
where $\psi_L$ denotes the isoperimetric profile of the
continuous--time  random walk  on $\z_L$ with generator
\[
 \cL f(x)
 =\sum_{y\in\z_L}\frac{r(y-x)}{m_L(x)} (f(y)-f(x))\,, \qquad x \in \z_L\,,
\]
with reversible measure $m_L$. We
take $\g>0$. The value will be fixed at the end. Due to assumption
(H1) and \cite[Lemma 2.1]{CF1},
there
exists a constant $\d>0$ such that $\bbP$--a.s.\ it holds
\begin{equation}\label{porchetta} \psi_L (u) \geq \d \min \left\{
\frac{1}{L^\g}, \frac{ 1}{ u^{1/d}L } \right\}, \qquad 0 < u \leq
1/2, \, L \geq L_0 (\xi ) \,.
\end{equation}
Since $L$ diverges with $t$,    we have $L\geq L_0(\xi)$ eventually
in $t$. Let us choose $  \e = L ^d/t^{d/2}$. Due to our assumption $
4/\e$ goes to $0$ as $t\to \infty$. In particular, we can take $t$
large enough that $4/\e \in (0,1/2)$ and $4\k/\e\in (0,1/2)$. This
together with \eqref{pecorino} implies that $\varphi _L(u) \geq
\psi_L(\k u)$ for $u$ as in the  r.h.s.\ of \eqref{MP1}. In
addition, note that
\begin{equation}\label{stimauni}
 4\p_L(x)/ \p_L(\z_L)^{d} \geq
\tilde c\, L^{-d}\,,
\end{equation}
for some new constant $\tilde c>0$.
Since the bound \eqref{porchetta} reads
\[
\psi_L (u)\geq
\begin{cases}
\frac{\d}{L^\g}\,, & \text{ if } 0<u\leq L^{-d(1-\g)}\,,\\
\frac{\d}{u^{1/d }L}\,, & \text{ if } u \geq L^{-d(1-\g)}\,,
\end{cases}
\]
taking $\g$ small enough we can assume that  $ 4 t^{d/2} \gg L^{d\g}
$, thus implying that (see \eqref{stimauni})
\begin{multline}
\text{r.h.s.\ of } \eqref{MP1}
\leq \int_{\tilde c L^{-d}}^{4t^{d/2} L^{-d}} \frac{8 du }{u \psi _L^2 (\k u)}
= \int_{\k\tilde c L^{-d}}^{4\k t^{d/2} L^{-d}}\frac{8 ds }{s \psi _L^2(s)}
\leq \\
\frac{8 L^{2\g}}{\d^2}\int_{\k\tilde c L^{-d}}^{L^{-d+d\g}}s^{-1}ds
+ \frac{8L^2}{\d^2}\int_{L^{-d+d\g}} ^{4 \k t^{d/2} L^{-d}}
  s^{2/d-1} ds= c (L^{2\g} + t)\,.
\label{signora1}
\end{multline}
Taking $\g$ small, we get $L^{2\g} <t$.  At the cost of changing the
definition of $\e$ by setting $\e=c'\,L ^d/t^{d/2}$ with $c'$ small
enough, we can assume that the last expression in \eqref{signora1}
is smaller than $t$, thus implying \eqref{MP1} and therefore that
\[
 P_{x,\xi}( \widetilde{Y}^{(L)}_t=x )
 \leq  \frac{c}{|\z_L|}( 1+ L^d t^{-d/2} )\,.
\]
At this point the claim follows from the fact that $\bbP$--a.s.\
$|\z_L|\geq c_1 L^d $ for some positive 
constant $c_1$ and for
all $L$.
%
Indeed, defining $C^*_L$ as the maximal connected component in
$[-L,L]^d\cap\bbZ^d$ for the Bernoulli field $Z(p)$, it is known
e.g.\ that if $p$ is large enough, then a.s.\
$|C^*_L| \geq \frac12\, L^d $ for all $L$ sufficiently large.
 Due to  the stochastic domination assumption (H1), the same holds
for $C_L$ as well. Since $\xi$, and therefore $\z_L$, has at least
one point in each box $B(z)$ with $z \in C_L$, it must be $ |\z_L|
\geq c_1 L^d$ for some $c_1=c_1(K)$. This concludes the proof.
\end{proof}

\subsection{Comparison between $\widetilde{Y}_t$ and
$\widetilde{Y}_t^{(L)} $} We first define a coupling $P_x$ between
the random walks $X_n^{(L)}$ and $X_n$, starting at  the same point
$x$ in $\z_L$, as follows. We realize $(X_n\,:\, n \geq 0)$ starting
at $x$, and call $\t =\inf\{ n\geq 1 \,:\, X_n \not \in \xi_L\}$.
Then we set $X^{(L)}_n= X_n$ for $n <\t$, while on $[\t, \infty)$
the random walk $X^{(L)}_n$ evolves independently from $X_n$ with
jump probabilities $p^{(L)}(\cdot, \cdot)$.
To check the validity of the coupling, let $A$ be the event that  $X_n= X^{(L)}_n$ for $ n\leq N$
(i.e.\ $A=\{ N <\t\} $).
 Note that, given $y,z \in \z_L$, the probability $P_x( X^{(L)}_{N+1}= z| X^{(L)}_N=y ,
 A)$ can be written as
\begin{multline*}
 P_x( X_{N+1}=z|X_N=y, A)
+   P_x ( X_{N+1} \not \in \z_L| X_N=y ,A) p^{(L)} (y,z)=\\  \frac{
r(z-y) }{w(y)}+ \frac{w(y)-w^{(L)}(y) }{w(y)} \frac{ r(y-z)}{
w^{(L)}(y)}= \frac{ r(y-z)}{ w^{(L)}(y)}=p^{(L)}(y,z) \,.
\end{multline*}

Introduce   a Poisson process  $N_t$  of parameter $1$
independent from $X_n^{(L)}$ and $X_n$, and therefore also from
$Y^{(L)}_n$ and $Y_n$.
Recall the  continuous--time random
walks
$ \wt{Y} ^{(L)}_t= Y^{(L)}_{N_t}$, $\wt{Y} _t = Y_{N_t}$.
We denote again by $P_x$ the  probability measure of the space where all the above processes  are defined, and
we write $E_x$ for the associated expectation.
An important consequence of this coupling is the following observation. There exists $\e=\e(d)>0$ such that, for $m\in\bbN$:
 \be\la{awross}
P_x\(\exists \,n\leq m\,: \;Y_n\neq Y^{(L)}_n \) \leq
P_x\Big(\max_{1\leq j\leq T_m} |X_{j}| > \e\, L\Big)\,.
 \end{equation}
Here $T_m$ is, as usual, the time of the $m$-th visit to
$\cC_\infty^*$ for the walk $X_n$. The above claim, in turn, is an
immediate consequence of the following
\begin{lemma}\label{sonno}
There exists $\e=\e(d)>0$ such that $\bbP$--a.s., for all
sufficiently large $L$, it holds $
\z_L\cap [-\e L, \e L ]^d=
\xi\cap\cC^*_\infty\cap[-\e L, \e L]^d$.
\end{lemma}
\begin{proof}
Let us prove the equivalent statement for the corresponding white $K$-boxes. Namely, setting $\ell=L/K$, we want to prove
\be\la{awro}
C_L \cap [-\e \ell,\e \ell]^d = C_\infty^*\cap  [-\e \ell,\e \ell]^d \,.
\end{equation}
By the stochastic domination assumption, and well known facts about Bernoulli percolation with large $p$ (see e.g.\
\cite[Proposition B.2]{CF1}), we may assume that $C_L$ coincides with the largest connected component of $[-\ell,\ell]^d\cap
C_\infty^*$. Thus, the only thing that can go wrong in checking (\ref{awro}) is that
there exist two vertices $x,y\in [-\e \ell,\e \ell]^d\cap
C_\infty^*$ that are not connected within $[-\ell,\ell]^d$.
Call $F_\ell$ this event. Using the stochastic domination assumption, and the fact that
$p$ is large, one can check that this event has exponentially (in $\ell$) small probability for a suitable $\e>0$.
To see this, let $d_C(x,y)$ denote the graph distance of two vertices $x,y\in C_\infty^*$
in the graph $C_\infty^*$ (this is often called the chemical distance). From known estimates \cite{AnPi}, for $\g>0$,
if $p$ is large, there exists $a>0$ such that
\be\la{awros}
\bbP\(d_C(x,y)\geq (1+\g)d_1(x,y)\tc x,y\in C_\infty^*\)\leq a^{-1}\,\nep{-a\,d_1(x,y)}\,.
\end{equation}
Let $x,y$ be two vertices as in the event $F_\ell$. Note the bounds
$d_1(x,y)\leq d\e\ell$ and $ d_C(x,y)\geq 2(1-\e)\ell$.
Moreover, one can find  $y'$ such that $|y'|_\infty \leq 3\e\,\ell$,
$4d\e\ell\geq d_1(x,y')\geq \e\ell$, and
$d_C(x,y')\geq 2(1-2\e)\ell\geq 2(1-2\e)(4d\e)^{-1}d_1(x,y')$.
Therefore, taking $\e$ small enough, a union bound and (\ref{awros})
imply, for some constant $c$, the   exponential bound
$$
\bbP(F_\ell)\leq \sum_{\;|x|\leq \e\ell\;}\sum_{\;|y'|\leq 3\e\ell\;}
 a^{-1}\,\nep{-a\,\e\,\ell}\leq c\,\,\nep{-c^{-1}\,\e\,\ell}\,.
$$
 The identity (\ref{awro}) then follows from
the Borel-Cantelli lemma.
\end{proof}

We can finally prove \eqref{mirtillo2}, an immediate consequence of
\begin{proposition}\label{guerra_banane} For $\bbP$--a.a.\ $\xi $, 
\begin{equation}\label{caciara}
\limsup_{t \to \infty} \max _{\substack{ x \in  \xi\cap\cC_\infty^* \,:\\\, |x|_\infty \leq t  } }
t^{d/2} P_x(\wt{Y} _t=x) < \infty\,.
\end{equation}
\end{proposition}
\begin{proof}
Since $P_x(|N_t-t| \geq t/2) \leq e^{-c t}$ for some positive constant
$c$, we can write
$$
 \bigl|  P_x(\wt{Y} _t=x ) -  \sum _{n = \lfloor t/2\rfloor}
^{\lfloor 3t/2 \rfloor } P_x ( Y_n =x )P( N_t=n) \bigr|\leq
e^{-ct}\,.
$$
A similar expression holds  for $\wt{Y}^{(L)}_t$ and $Y_n^{(L)}$.
 On the other hand, thanks to (\ref{awross})
we can use (for $n\leq \lfloor 3t/2\rfloor$)
\[ P_x (Y_n=x) \leq P_x \( Y^{(L)}_n=x\) + P_x \bigl( \max_{1\leq j \leq  T_{\lfloor 3t/2\rfloor} }  |X_j|>\e \,L \bigr)\,.
\]
Together with the above observations, this gives:
\begin{multline}\label{anagnina}
P_x\(\wt{Y} _t=x \) \leq e^{-ct} +\sum _{n = \lfloor t/2\rfloor}
^{\lfloor 3t/2 \rfloor } P_x ( Y_n =x )P( N_t=n)\\
\leq 2 e^{-ct} + P_x\(\wt{Y}^{(L)} _t=x \)+P_x \bigl( \max_{1\leq j
\leq T_{\lfloor 3t/2\rfloor} }  |X_j|>\e\,L \bigr) \,.\end{multline}
To bound the last expression, we use
Markov inequality and Lemma \ref{moltobuono}: 
\begin{multline}\label{battistini}
P_x \( \max_{1\leq j \leq  T_{\lfloor 3t/2\rfloor} }  |X_j|>\e\,L \)
\leq (\e\,L) ^{-1} E_x
\( \max_{1\leq j \leq  T_{\lfloor 3t/2\rfloor} }  |X_j|  \) \\ \leq
\k\,\e^{-1} L^{-1}  \(|x|+ [1+\log (1+|x|)]  \) \,\lfloor 3t/2 \rfloor\,.
\end{multline}
If $L=t^u$, and $u>1$, then we can assume $t\leq\e\,L$, and
collecting \eqref{anagnina}, \eqref{battistini} and invoking
Proposition \ref{congiuntivite} (using again Lemma \ref{sonno})
 we get for $\bbP$--a.a.\ $\xi$ that
$ t^{d/2} P_x(\wt{Y} _t=x )\leq  c\( 1+ \,t^{d/2 -u +2}\) $ for all
$x \in \xi\cap\cC_\infty^*$ such that $|x|_\infty \leq t $, for some
finite constant $c=c(\xi)$. Taking e.g.\ $u=2+d/2$ concludes the
proof.
\end{proof}

\section{Expected distance bound}\label{distante}
In this section we prove the distance estimate \eqref{mirtillo1}.
Given $x,y \in \xi \cap \cC_\infty^*$ define the heat kernel by
$q_t(x,y)=P_{x,\xi}(Y_t=y)/w(y)$. Given  $\d>0$,  define also 
\begin{align}
&  D  = \sup_{\xi\in\cN}\sup_{x\in\xi \cap \cC^* _\infty}\max_{1\leq
j\leq T_1}
 E_{x, \xi }\bigl[\bar d(x,X_j)^2 \bigr]\\
&  M(x,t) = E_{x,\xi} \bigl[\bar d(x,Y_t)\bigr] = \sum_{y\in \xi \cap \cC^*_\infty} \bar d(x,y)q_t(x,y)w(y) \\
&  Q(x,t)  = -E_{x, \xi} \bigl[ \log q_t(x,Y_t) \bigr]= - \sum_{y\in \xi \cap \cC^*_\infty} q_t(x,y)w(y)\log q_t(x,y) \\
& \Cvol(x, \d )
  = \sup_{0<s<\delta}\biggl\{s^d\sum_{y\in\xi \cap \cC^*_\infty}w(y)\nep{-\bar d(x,y)s}\biggr\}\,.
\end{align}
(By continuity, $Q(x,0)=\log w(x)$.) By Lemma \ref{lecda20}, we know
that $D<\infty$.

\begin{lemma}\label{distanceBounds}
 For all $x$,
\begin{align}
&    M'(x,t)^2\leq D Q'(x,t)\,, \qquad \forall t\geq 0\,,\label{upperbound} \\
&   M(x,t)^d\geq\Exp{-1-\Cvol(x, \d)+Q(x,t)}\,,  \qquad \text{ if
}\; M(x,t)>\delta^{-1}\,. \label{lowerbound}
\end{align}
\end{lemma}
\begin{proof}
 The proof of \eqref{lowerbound}  is  an adaptation of \cite[Lemma 3.3]{barlow}; see also \cite[Lemma
 6.3(a)]{BP} for a similar argument.
To prove \eqref{upperbound},
 recall that $\o_{x,y}$ denotes  the jump rate  of the restricted
random walk (cf \eqref{eq_omegaxy}).
 For \eqref{upperbound},
 following \cite[Proposition 3.4]{barlow}
 almost exactly, we use the triangle inequality for $\bar d$ and then
 Schwarz' inequality  to arrive at:
 \begin{multline}
    M'(x,t)^2
   \leq \frac 12\biggl(\sum_{y,z} q_t(x,y) \o_{y,z}  \bar d(y,z)^2\biggr)\times \\
   \biggl(\sum_{y,z}
  \o_{y,z} \left(q_t(x,y)-q_t(x,z)\right)(\log q_t(x,y)-\log q_t(x,z))\biggr)\,,\label{patroclo}
 \end{multline}
where $\sum_{y,z}$ corresponds to $\sum _{y\in \xi \cap \cC_\infty^*
}\sum _{z \in \xi \cap \cC_\infty^*}$.

 The conclusion is now very different from \cite{barlow}  and the use of the distance $\bar d$ instead of the euclidean distance becomes crucial.
We observe that
\begin{multline*}\sum_{y,z} q_t(x,y) \o_{y,z} \bar d(y,z)^2
 =\sum _{y} q_t(x,y) \sum_z P_{y, \xi}(X_{T_1} =z)\bar d(y,z)^2\\
 =
 \sum_{y} P_{x, \xi}( Y_t=y) E_{y,\xi}  \bigl[\bar d(y,X_{T_1})^2\bigr]\leq D\,,
 \end{multline*}
so that the first factor inside the brackets in the last member of
\eqref{patroclo}  is bounded by $D$.
 Exactly as in \cite[Proposition 3.4]{barlow}, the second factor is equal to
 $2Q'(x,t)$.
 We therefore have $M'(x,t)^2\leq D Q'(x,t)$, as desired.
\end{proof}


\begin{lemma}\label{mezzanotte} Take $K,T_0$ satisfying \eqref{probdbar} for some $a>0$.
 For $\bbP$--a.a. $\xi$, 
 \begin{equation}\label{cts}
\Cvol(x, \d )\leq  C(\xi)  \d^d [1+\log (1+|x| )
]^d+C(\xi)e^{C(\xi) \d}  \,, \qquad x \in \xi \cap \cC^*_\infty
 \end{equation}
 for some positive constant $C(\xi)$. In particular,
for $\bbP$--a.a. $\xi$, 
\begin{equation}\label{ctsbis}
\max _{n \geq 1}   \max _{\substack{  x\in \xi\cap  \cC^*_\infty:\\
|x|_\infty \leq n } }\;\, \sup_{ t \geq  n } \Cvol (x,1/\sqrt{t} ):=
C_1 (\xi) < \infty \,.
\end{equation}
\end{lemma}
\begin{proof}
The last bound \eqref{ctsbis} trivially follows from \eqref{cts}.
Therefore, we concentrate  on \eqref{cts}. Due to Lemma \ref{leT}
and the definition of the random field  $\vartheta$, we know that
$w(y) \leq c=c(K,T_0)$ for all $y \in \xi \cap  \cC_\infty ^*$.
Moreover, we know that all $K$--boxes $B(z)$ with $z \in C^*_\infty$
are not overcrowded, i.e.\ $\xi ( B(z) ) < T_0$. In particular, we
can bound
 $ \sum _{y \in \xi \cap B(z) } e^{-s |x-y|} $ from above by $T_0 e^{-s K (|v-z|-c(d) ) }$ if $x \in B(v)$.
 Let $\k= \k(\xi,K)$ be the positive constant appearing in Corollary \ref{coroscaz}. We define
 $$ W(x) =\bigl \{y \in \xi \cap \cC^*_\infty:  |x-y|/\k  \leq 2[1+\log (1+|x| ) ]\bigr\}\,.$$
Then, applying also  Corollary \ref{coroscaz},
 we conclude that
\begin{multline*}
\Cvol(x,\d ) \leq c \sup_{0 <s<\d} \Big\{
  s^d |W(x)|+ s^d
\sum_{ y \in \cC^*_\infty \setminus W(x) }
 e^{- \frac{s|x-y|}{2\k} }
\Big\}
\\\leq C \d^d [1+\log (1+|x| ) ]^d +
%
Ce^{C\d} \sup_{0<s<\d} \bigl\{  s^d \sum _{z   \in \bbZ^d}    e^{-s
\frac{K|z|}{2\k} }\bigr\}
%
 \,,
\end{multline*}
for a positive constant $C$ independent from $x$. The last term  can
be estimated by $$
c(d) C \int _1 ^\infty s^d e^{-s
\frac{ K r}{ 2 \k}} r^{d-1} dr=c(d) C \int _s^\infty e^{-y} y^{d-1}
dy \leq c'(d) C \,,
$$
thus concluding the proof of \eqref{cts}.
\end{proof}
Let us now come back to the heat kernel estimate of Proposition
\ref{guerra_banane}.
We know that 
$ t^{d/2} \bbP_x^\xi (Y_t =x )$ is bounded from above uniformly as
$t \geq 1$,  $ x \in \xi \cap \cC^*_\infty$ and  $|x|\leq t$,  for
$\bbP$--a.a.\ $\xi$.
Since $\sup_{y \in \xi \cap \cC^*_\infty} w(y) < \infty$ (see Lemma
\ref{leT}) and $q_t(x,y) \leq q_t(x,x)$, this implies that there
exists a  (finite) positive constant $A=A(\xi)$ (that we take larger
than $1$) such that
\begin{equation}\label{sibillini}
 \sup _{t\geq |x| \lor 1  }  \;\sup_{y\in \xi\cap  \cC^*_\infty}\; t^{d/2 } q_t (x,y) \leq A  \, , \qquad \forall x \in \xi \cap \cC^\infty_*\,.
\end{equation}
\begin{proposition}\label{esausta} Let $t(x)= |x| \lor e$ and set
$ T(x)=   t(x)\log t(x)  \vee \frac{c^2}{dD}$,
 where the positive constant $c=c(\xi)$ is the same appearing in \eqref{phiposa} of Lemma \ref{lecda20} with $n=p=1$.
 Then for $\bbP$--a.a. $\xi$  there exists a  constant $C_0(\xi)\geq 1 $ such that
 \begin{equation}\label{sfinita}
  M(x,t)\leq  C_0(\xi) \sqrt{t} \qquad \forall x \in \xi \cap \cC^*_\infty\,, \; \forall t \geq T(x)\,.
 \end{equation}
\end{proposition}
\begin{proof}
 We will follow \cite[Proposition 3.4]{barlow}.
 Since  $C_0(\xi) \geq 1$, we can assume
 that $M(x,t)\geq\sqrt t$ otherwise we have nothing to prove. In this case,
since $t \geq T(x) \geq |x|$, by \eqref{lowerbound} and
\eqref{ctsbis} we can estimate
\begin{equation}\label{lostrozzo}
M(x,t)^d\geq\Exp{-1-C_1(\xi) +Q(x,t)}\,.
\end{equation}
We will use this key lower bound at the end.

Following \cite{barlow}, we  define $L(t)=\frac1d(Q(x,t)+\log
A-\frac d2\log t)$ (recall \eqref{sibillini}).  Note that
 $L(t)\geq 0$ on $t\geq t(x)$ by \eqref{sibillini}.
Then, we define
 \begin{equation}\label{rafaello}
 t_0:=
 \begin{cases}
 1 &  \text{ if } L(t) \geq 0  \text{ on } [1, t(x)]\,,\\
 \sup\{ t \in (0, t(x)]\,:\, L(t)<0 \} & \text{ otherwise} \,.
\end{cases}
\end{equation}
Note that in the second case, it must be $L(t_0)=0$, i.e.\ $
Q(x,t_0)= -\log A+ \frac d2\log t_0$.

We claim that  $M(x,t_0) \leq \sqrt{ d D T(x)}$. To this aim, let us
first assume that $L(t) \geq 0$ for all $t >0$. Then $t_0=1$ and
from the definition of $T(x)$ and Lemma \ref{lecda20} we deduce that
$$M(x,t_0)= E_{x,\xi}  \bigl( \bar d (x, Y^{\text{c.t.}}_1) \bigr)= E_{x,\xi}
\Big[ E_{x,\xi}  \bigl( \bar d (x, Y^{\text{d.t.}}_{N_t} )|N_t)
\Big] \leq c\, E(N_t)= c \leq \sqrt{ dD   T(x) }\,,$$ where $N_t$ is
a Poisson process of parameter $1$ independent from the
discrete--time restricted random walk and were  ``c.t." and ``d.t."
mean continuous--time and discrete--time, respectively (to avoid
ambiguity). Let us now assume that  $L(t) < 0$ for some $t$ (the
second case of \eqref{rafaello}). Since $L(t) \geq 0$ for $t \geq
t(x) $ as already observed, it must be $t_0 < t(x)$.
 By \eqref{upperbound} in Lemma \ref{distanceBounds}  and
 Schwarz' inequality, $M(x,t_0)$ can be bounded from above by
 $ \sqrt{t_0D}\bigl(\int_0^{t_0}Q'(x,s)\dd s\bigr)^{1/2}$.
 By continuity at both endpoints and using that $L(t_0)=0$, this last
 expression
  is bounded from above by
 \begin{equation}
  \sqrt{t_0D}\left(\frac d2\log t_0-\log A-\log w(x)\right)^{1/2}
  \leq\sqrt{t_0Dd\log t_0}
  \leq\sqrt{D d T(x)},
 \end{equation}
 (recall that  
 $A \geq 1$ and  $w(x) \geq r(0)=1$).  This concludes the proof of our claim.

Since $Q'(x,t)= d L'(t)-d/(2t)$ and using \eqref{upperbound}, by the
same computations as in \cite{barlow}
 we get for all $t \geq t_0$
 \begin{multline}
  M(x,t)-M(x,t_0)
   \leq \sqrt{dD}\int_{t_0}^t\Big(\frac1{2s}+L'(s)\Big)^{1/2}\dd s
   \leq \sqrt{2d Dt} + L(t) \sqrt{d D t}\,.
 \end{multline}
 Using now the bound $M(x,t_0) \leq \sqrt{ d D T(x)}$ we
conclude that
 \begin{equation}\label{trenino_pierpi}
  M(x,t)\leq\sqrt{dDT(x)}+\sqrt{2dDt}+L(t)\sqrt{dDt}
  \leq(1+\sqrt2)\sqrt{dDt}+L(t)\sqrt{dDt}.
 \end{equation}
 Conversely, because $ \sqrt t\leq M(x,t)$, we can apply
 \eqref{lostrozzo}  to find that
 \begin{equation}\label{nave_pierpi}
  M(x,t)\geq C_2(\xi)
\nep{L(t)}\sqrt t
 \end{equation}
 for some positive constant $C_2$ depending on $\xi$.
 Combining these last two equations \eqref{trenino_pierpi}, \eqref{nave_pierpi}, and eliminating the common $\sqrt t$,
 we see that
 \begin{equation}
  \nep{L(t)}\leq \bigl[ \sqrt{d D} /C_2(\xi)\bigr ]
  (1+\sqrt2+L(t) )
 \end{equation}
 Since $e ^y \geq 1+y+\frac12\,y^2$ for all $y \geq 0$, the above formula implies that
 $ L(t)^2 \leq a+ b L(t)
$ for suitable constants $a=a(\xi)$, $b=b(\xi)$. This last bound
implies that $L(t) \leq C_3(\xi)$. Coming back to 
\eqref{trenino_pierpi} we get \eqref{sfinita}.
 \end{proof}

\subsection{Proof of \eqref{mirtillo1}}
 We have now all the tools to prove \eqref{mirtillo1}. We take  $K,T_0$ satisfying Corollary \ref{coroscaz}  and we define
 $b_n= n^{\g}$ with $\g \in (1,2)$. If $t \geq b_n$ with $n$ large enough,
 then $t \geq  T(x)$ for all     $x\in \xi \cap
 \cC^*_\infty$ such that   $|x |_\infty \leq n $. In particular, applying Proposition \ref{esausta} we conclude that for $\bbP$--a.a. $\xi$, 
 \begin{equation}\label{catarchico}
 \limsup _{n\uparrow \infty}  \max_{ \substack{ x \in \xi \cap \cC^*_\infty:\\ |x|_\infty \leq n } }\;\;
 \sup_{t \geq b_n} \frac{E_{x,\xi} \bigl[\bar d (Y_t,x ) \bigr]}{ \sqrt{t} }  < \infty\,.
 \end{equation}
We now apply Corollary \ref{coroscaz}, to estimate
$$
\frac{|Y_t-x|}{\sqrt{t} }\leq \k n^{-\g/2} [1 +\log (1+n )]   + \k\frac{\bar d (x,Y_t)}{\sqrt{t} }\,,
$$
The above bound together with \eqref{catarchico} trivially implies
\eqref{mirtillo1}. \qed




\section{Sublinearity of the corrector}\label{sublinear}
This section is devoted to the proof of Theorem \ref{piccolo}.

\subsection{Preliminary bounds}
We start with a polynomial estimate on the size of the corrector for points within
the cluster $\cC_\infty$.
Note that we are now working   with the cluster of occupied boxes $\cC_\infty$ and not
with the (smaller) cluster of white boxes $\cC_\infty^\ast$. We will come back to the latter towards  the end of this section.

\begin{lemma}\label{polgrowthbis} For $\theta > d+1$,
 \begin{equation}\label{brunabis}
  \lim _{n\to\infty} n^{-\theta}
  \max_{\substack{ x,y\in\xi \cap \cC_\infty\\|x|_\infty,|y|_\infty\leq n}} \bigl| \chi(\t_x
  \xi, y-x) \bigr|
  =0\,,\quad \bbP\textrm{--a.s.}
 \end{equation}
\end{lemma}
\begin{proof}
For any  $x,y \in\xi\cap\cC_\infty$ with $|x|_\infty, |y|_\infty \leq n$ there
exists a
path 
$x=x_0,\dotsc,x_m=y\in\xi$, with $x_i$ and $x_{i+1}$ belonging
to adjacent $K$--boxes $B_i,B_{i+1}$ on $\cC_\infty$.
For a fixed $\l<\infty$, let $E_{\l,n}\subset \cN $, denote the event 
that, for any $x,y\in\xi\cap\cC_\infty$ with $|x|_\infty, |y|_\infty\leq n$, there
exists such a path with the additional property that
$\max_i|x_i|_\infty \leq \l\, n$.

Note that there exists $\d=\d(K)>0$ such that $r(x_{i+1}-x_i)\geq
\d$ for every $i$. Therefore, using the shift--covariance property
we get
\begin{equation}\la{chi1bis}
  |\chi(\t_x \xi,y-x )|
 \leq\d^{-1}\sum_{i=0}^{m-1}r(x_{i+1}-x_i)
 |\chi(\t_{x_i}\xi,x_{i+1}-x_i)|\,.
\end{equation}
We shall write $B\subset B(\ell)$ when a $K$--box $B$ is contained
in the $|\cdot|_\infty$--box in $\bbR^d$  centered at the origin, of
size $\ell$. Thus on $ E_{\l,n}$ we can estimate
\begin{multline}\la{chio1bis}
 \cR_n(\xi):=\max_{\substack{ x,y\in\xi \cap \cC_\infty\\|x|_\infty,|y|_\infty \leq n}} \bigl| \chi(\t_x
  \xi, y-x) \bigr|  \leq  \d^{-1}
 \sum_{B\subset B(\l n)}
 \sum_{x\in\xi\cap B}
 \sum_{y\in\xi}
 r(y-x)|\chi(\t_x\xi,y-x)|
  \\=
  \d^{-1} \sum _{\substack{x \in \xi: \\|x|_\infty \leq \l n }}
  \sum _{y \in \xi} r(y-x)|\chi(\t_x\xi,y-x)|= \d^{-1} \sum _{\substack{x \in \xi: \\|x|_\infty \leq \l n }} g(\t_x \xi)
   \,,
\end{multline}
where the function $g: \cN_0 \to [0,\infty)$ is defined as
$g(\xi)= \sum _{z \in \xi} r(z) | \chi (\xi,z)|$. Thus,
\be\label{chio2bis}
 \bbP \Big[\cR_n \,1_{\{E_{\l,n}\}}\geq n^\theta\Big]
  \leq n^{-\theta}\,\bbE\Big[\cR_n \,;\: E_{\l,n}\Big ]
  \leq
  \d^{-1}\,n^{-\theta}\, \bbE\Big[\sumtwo{x\in\xi:}{|x|_\infty\leq\l n}g(\t_x\xi)\Big]\,.
\end{equation}
Applying now the Campbell identity \eqref{campbell}, with the
notation (\ref{scalare}) we can write the last expectation as
\begin{equation}\label{verona} \rho (2 \l n)^d \bbE_0 ( g (\xi) )\leq
\rho (2 \l n) ^d \bbE_0( w(0) )^{1/2} \| \chi \|_{L^2 (\mu)} <
\infty\,.
\end{equation}
Using this bound in \eqref{chio2bis} we obtain
$
 \bbP \bigl[\cR_n\,1_{\{E_{\l,n}\}} \geq n^\theta\bigr]
 \leq C\,n^{d-\theta}$,
for some finite constant $C=C(\l,\r,\d)$. In conclusion,
$$
 \bbP \left[\cR_n\geq n^\theta\right]\leq  \bbP \left[\cR_n\,1_{\{E_{\l,n}\}} \geq n^\theta\right] + \bbP[E_{\l,n}^c]\leq
 C\,n^{d-\theta} + \bbP[E_{\l,n}^c]\,.
$$
 If $\l$ is sufficiently large, the
same argument used in the proof of Lemma \ref{sonno} shows that
$\bbP[E_{\l,n}^c]$ is exponentially decaying in $n$. Taking
$\theta>d+1$, the Borel-Cantelli lemma implies that
$n^{-\theta}\cR_n\,\to 0$, $\bbP$--almost surely.
\end{proof}

The next lemma extends the estimate of Lemma \ref{polgrowthbis} to the case where
$y\notin \cC_\infty$. For $\a\leq 1$ this remains polynomial. When
$\a>1$ the bound is of the form $\exp((\log n)^\a)$. To unify the
notation
we use the function
\[u_{\g,c}(t)=t^\g \exp \bigl[ c (\log (t+1) )^\a\bigr]\]
introduced in Lemma \ref{lecda}.

\begin{lemma}\la{leles} There exist  constants $\g,c $ such that
for $\bbP$--a.a.\ $\xi$ and for all $n \geq n_0(\xi)$,
\begin{equation}\label{dindondan}
\max_{\substack{ x \in \xi\cap \cC_\infty, |x| \leq n\\ y \in \xi\,,
\; |y|\leq n } } | \chi (\t_x \xi , y-x)|\leq u_{\g,c}(n)\,.
\end{equation}
\end{lemma}
\begin{proof}
As in the proof of Proposition \ref{polgrowthbis}, we get
\begin{equation}\label{ob1}
\bbE\Big[\sum_{x\in\xi: |x|\leq n}\sum_{y\in\xi} |\chi(\t_x\xi,y-x)|
r(y-x)
\Big]\leq c\,n^d\,,
\end{equation}
for some constant $c$. From \eqref{ob1} and Markov's inequality, the
Borel-Cantelli lemma shows that $\bbP$--a.s.\ for all $n$ large
enough: \be\la{ob2} \max_{x\in\xi:\;|x|\leq n}
\sum_{y\in\xi}|\chi(\t_x\xi,y-x)|
r(y-x) \leq n^{d+2}\,.
\end{equation}
Now, take $x,y\in\xi$ such that $|x|,|y|\leq n$ with $n$ large. If
$x\in\xi\cap \cC_\infty$, from Lemma \ref{polgrowthbis},
$|\chi(\t_x\xi,z-x)|\leq n^\theta$ for all $z\in\xi\cap \cC_\infty$,
$|z|\leq n$. However, a.s.\ there exists $z\in\xi\cap \cC_\infty$,
$|z|\leq n$, such that $|z-y|\leq C\log n$ (the distance of $y$ from
$\cC_\infty$ cannot be larger than $C\log n$). Thus, the claim
follows by writing $ |\chi(\t_x\xi,y-x)|\leq |\chi(\t_x\xi,z-x)| +
|\chi(\t_z\xi,y-z)|$ and using \eqref{ob2} on the obvious bound
\[
|\chi(\t_z\xi,y-z)|\leq
\sup_{x\in\xi:\;|x|\leq n} \nep{(C\log n)^\a}\,
\sum_{y'\in\xi}|\chi(\t_x\xi,y'-x)|
r(y'-x)\,. \qedhere
\]
\end{proof}

\subsection{Sublinearity
along a given direction in $\cC_\infty$}\label{mediano1}
Let us fix a coordinate vector $e$ (i.e.\ $e \in \bbZ^d $, $|e|_1=1$).
Recall the notation $B_z=B(z)$, for the $K$-box at $z\in\bbZ^d$.
Omitting  the dependence  on  $e$, we set $ n_0(\xi)=0$
and define inductively
\[
 n_{i+1} (\xi)= \min \{ j>n_i (\xi)\,:\, B_{je} \subset
 \cC_\infty \}\,.
\]
By property (A) in  Section \ref{equoesolidale}, for $K$ large, the
above maps are well--defined $\bbP$--a.s.

\smallskip

We introduce the space $\Omega= \{ \xi: B_0\subset \cC_\infty(\xi)
\} \times B_0^{\bbN}$. A generic element of this space is denoted by
$\o= \bigl(\xi, (v_i\,:\, i\in \bbN  ) \bigr)$.  On the space
$\Omega $ we define a probability measure $\bbP_*$ (depending on the
coordinate vector $e$)  as follows: the marginal distribution of
$\xi$ is given by $\bbP( \cdot|B_0 \subset \cC_\infty) $ while,
conditioned to $\xi$, the sequence $(v_i\,:\, i\in \bbN)$ is
determined by choosing independently for each index $i$ a point $w
_i$ with uniform probability in $B_{ n_i(\xi) e} \cap \xi $ and then
setting
\begin{equation}\label{fabiovoloma}
 v_i= w_i -n_i(\xi)K e \,,
 \end{equation}
 so that $v_i\in B_0$.
Trivially, by \eqref{fabiovoloma}, knowing  $\o= \bigl(\xi,
(v_i\,:\, i\in \bbN  ) \bigr)$ the points  $w_i$ are univocally
determined, hence   we  write $w_i=w_i(\o) $.
 Below we write $\bbE_*$ for the expectation w.r.t.\ $\bbP_*$.
We point out that the
  space $\O$ is an example of the bridge spaces mentioned in the
Introduction.
The following key result is a consequence of assumption (H2):
\begin{lemma}\label{babbo}
Consider the map $\cT: \Omega \to \Omega$ defined as
\[
 \cT\bigl( \xi, (v_i\,:\, i \in \bbN) \bigr )
 = \bigl( \t_{n_1(\xi) K e }  \xi,(v_{i+1}\,:\,  i \in \bbN)\bigr)\,.\]
Then $\bbP_*$ is ergodic and
stationary w.r.t.\ the transformation $\cT$.
\end{lemma}
\begin{proof}
Consider the space $\Theta$ with probability measure $P:=P^{(K,e)}$ involved in assumption (H2). Define
the subset $\cW \subset \Theta$ as
\[
 \cW
 :=\bigl\{\bigl(\xi,(a_i:i\in\bbZ)\bigr)\in\Theta\,
 :\,B_0\subset\cC_\infty\bigr\}
\]
Then $P(\cW)= \bbP( B_0 \subset \cC_\infty)>0$.

Recall the transformation $\t: \Theta \to \Theta $ introduced in assumption (H2).
It is invertible and ergodic w.r.t.\ $P$ (by assumption (H2)). It is simple to check that it is measure preserving
(using the stationarity of $\bbP$).

 Let $F: \Theta \to \bbN \cup\{ \infty\}$
 be defined as
$$
 F (\vartheta)= \min \{ k \geq 1\,:\, \t ^k \vartheta \in \cW\}= n_1(\xi)\,,
 \qquad\vartheta= \bigl( \xi , (a_i: i \in \bbZ) \bigr)  \in \Theta
$$
and set $ \cS (\vartheta )= \t^{F(\vartheta) } (\vartheta)$ if
$F(\vartheta)<\infty$ (define $\cS$ arbitrarily on the event $\{
F(\vartheta)=\infty\}$, having zero $P$--probability). By the above
observations all the conditions of Lemma 3.3 in \cite{BB} are
satisfied. In particular, we get that $\cS$ (restricted to $\cW$) is
a measure preserving and ergodic transformation with respect to
 $P( \cdot |\cW)$.

Consider now the map $ \p : \cW \to \O $ mapping
 $ (\xi, ( a_i: i \in \bbZ))$  to
 $ \bigl(\xi, ( a_{ n_i ( \xi ) }: i \in \bbN ) \bigr)$.
Note that
\begin{equation}\label{nonnobile} \cT ( \p( \vartheta) ) = \p ( \cS ( \vartheta ) )\qquad  \forall \vartheta \in \cW \,.
\end{equation}
Take  $A \subset \O$ measurable such that
$ \cT (A)=A$. Due to \eqref{nonnobile},
$\cS(\p^{-1}(A))=\p^{-1}(A)$ and therefore
$ P( \p^{-1} (A)|\cW  ) \in \{0,1\}$ by the ergodicity of
$\cS$ w.r.t.\ $P( \cdot | \cW)$.
Since  $\bbP_*( A)= P( \p^{-1} (A) |\cW) $, we conclude that
$\bbP_*(A)\in \{0,1\}$.
\end{proof}

We define the vector function $\tilde \chi: \Omega \times \bbR^d
\to \bbR^d$ as
\begin{equation}\label{mucca0}
\tilde \chi (\o,x )= \chi ( \t_{v_0} \xi, x-v_0)\,, \qquad \o= \bigl(
\xi, (v_i\,:\, i \in \bbN) \bigr) \,.
\end{equation}
Note that from the shift--covariance property of $\chi$ (cf.\ Lemma \ref{festicciola})
\begin{equation}\label{mucca1}
\tilde \chi (\o,x )-\tilde \chi (\o,y)= \chi ( \t_y \xi, x-y)\,,
\qquad \forall \xi \in \cN\,,\; \forall x,y \in \xi\,.
\end{equation}

\begin{lemma}\label{torlonia}
$\bbE_* (|\tilde  \chi( \o, w_1)|) < \infty$ and $ \bbE_* ( \tilde
\chi (\o, w_1))=0$, where $w_1$ is defined as in
(\ref{fabiovoloma}).
\end{lemma}
\begin{proof}
Let us first show that $\bbE_* (|\tilde  \chi( \o, w_1)|) < \infty$.
For $\xi \in \{ B_0 \subset \cC_\infty\} $, we define $ d ( \xi) $
as the length of the  minimal path in the cluster $C_\infty (\xi)
\subset \bbZ^d$ from $0$ to $n_1 (\xi) e$. Note that $ d (\xi) \geq
n_1(\xi)$.  Then, setting
$g_k(\xi)=\sum_{y\in\xi}r(y)|\chi(\xi,y)|^k$, by the same arguments
leading to \eqref{chio2bis} and applying twice  Schwarz' inequality,
we get
\begin{align*}
\bbE_* &(|\tilde  \chi( \o, w_1)|) = \sum_{j=1}^\infty \bbE_*
(|\tilde \chi( \o, w_1)|\,;\,d(\xi)=j )\nonumber \leq\\
&\frac{1}{\d}  \sum _{j=1}^\infty
 \bbE\Big[
 \sum_{ x\in \xi: |x|_\infty \leq Kj }g_1(\t_x \xi)\,;\, d(\xi)=j \, \Big| \, B_0 \subset \cC_\infty
 \Big]
\leq  \frac{1}{\d   \bbP( B_0 \subset \cC_\infty)^{1/2}}
\times\nonumber
\\ &\sum _{j=1}^\infty \bbP\Big[ d(\xi) =  j\,\Big| \,B_0 \subset
\cC_\infty\Big]^{1/2}\bbE \Big[\sum_{  x\in \xi: |x|_\infty \leq K j
} g_0(\t_x \xi)
 \Big]^{1/4 }
 \bbE  \Big[\sum_{  x\in \xi:|x|_\infty
\leq K j } g_2(\t_x \xi)
 \Big]^{1/4 }\,.
\end{align*}
Using the Campbell identity as in \eqref{verona}, we get that the
last two expectations are bounded by $ C(K,d) j^d \| \chi
\|_{L^2(\mu)} $ and  $C(K,d) j^d$, respectively. Finally, due  to
property (A) in Section \ref{equoesolidale} and standard facts in
percolation theory (see for example Lemma 4.4 in \cite{BB}),
$\bbP\bigl[ d(\xi)
 \geq j| B_0 \subset \cC_\infty] \leq e^{-a j }$ for some positive constant
$a=a(K,d)$. Collecting the above bounds we get that $ \bbE_*
(|\tilde  \chi( \o, w_1)|)<\infty$.


 We know that $\chi $ is the $L^2(\mu)$--limit of a sequence $\chi_n
 $ of functions of the form $\chi_n (\xi, x)= G_n (\t_x \xi)-
 G_n(\xi) $, where $G_n : \cN_0\to \bbR^d$ is  bounded and
 measurable.
 Since $(1,1)_\mu <\infty$, we derive that $ \chi_n \in L^1(\mu)$ and   that $\|\chi_n
  -\chi\|_{L^1(\mu) } $ goes to zero as $n\to \infty$.
  Repeating the above computations with $\tilde \chi$ replaced by
  $\tilde \chi- \tilde \chi_n$ we conclude that
$
 \bbE_*\bigl(|\tilde\chi(\o, w_1)-\tilde\chi_n(\o, w_1)|\bigr)$ goes to zero as $n \to \infty$.
In particular, 
  $ \lim_{n\to \infty} \bbE_*( \tilde \chi_n(\o, w_1))= \bbE_*(\tilde
  \chi(\o, w_1) )$.
  On the other hand, we can write
  \[
   \bbE_*(\tilde\chi_n(\o, w_1))
   = \bbE_*( \chi_n (\t_{v_0} \xi, w_1-v_0))
   = \bbE_*(  G_n (\t_{w_1} \xi) ) - \bbE_*( G_n(\t_{v_0} \xi))\,.
\]
Setting  $\o = \bigl(\xi, (v_i \,:\, i \in \bbN)\bigr) $ and $F_n(\o)=G_n (\t_{v_0}
\xi)$, we can write $G_n (\t_{w_1} \xi)  = F_n (\cT \o) $. Hence,
the conclusion follows from Lemma \ref{babbo}.
\end{proof}

\begin{lemma}\label{mamma} With the notation of Eq.\
  (\ref{fabiovoloma}), one has
\begin{equation}\label{casa}
\lim _{k\to \infty} \frac{ \tilde \chi( \o , w_k)}{k}=0 \,,
\qquad 
\bbP_*\text{--a.s.}
\end{equation}
\end{lemma}
\begin{proof} Let $\o= \bigl( \xi, (v_i: i \in \bbN) \bigr)$.
Since $ v_0 =w_0$, it holds $\tilde \chi ( \o, w_0) = \chi (\t_{v_0}
\xi , 0 ) =0$. Hence, we can write
$\tilde \chi (\o, w_k)
=\sum _{j=0}^{k-1}\bigl( \tilde \chi (\o, w_{j+1})-\tilde \chi (\o,
w_j) \bigr)$.
Applying now    \eqref{mucca1}, we get $\tilde \chi (\o, w_k)=\sum
_{j=0}^{k-1} \chi ( \t_{w_j} \xi , w_{j+1}-w_j) $.
On the other hand, since  $ \cT^j \o =
 \bigl(\t_{n_j(\xi)K e}\xi,(v_{j+i}:i\in\bbN)\bigr)$, we can write
\begin{multline}\label{sole2}
\chi( \t_{w_j} \xi,w_{j+1}-w_j)=\\
\chi\left( \t_{v_j} (\t_{n_j(\xi) K  e } \xi), v_{j+1}+ n_1(
\t_{n_j(\xi)K  e } \xi)Ke -v_j \right)=  \tilde \chi \bigl(\cT^j\o,
w_1 (\cT^j \o) \bigr) \,.
\end{multline}
Hence, $\tilde \chi(\o, w_k)
 = \sum _{j=0}^{k-1}
 \tilde  \chi \bigl(\cT^j \o, w_1 (\cT^j \o) \bigr)$.
The conclusion follows from   the ergodicity stated in Lemma
\ref{babbo} and the results of  Lemma \ref{torlonia}.
\end{proof}

Next, we state a simple  corollary of the above lemma, which is the starting point of our further investigations. In order to stress
 the dependence of the map $n_i (\cdot)$ from  the vector $e$ we write
 $n_i ^{(e)}$. Below, $\bbN_+= \{ 1,2, \dots\}$.
\begin{corollary}\label{plasmon} Given a vector $e \in \bbZ^d$ with $|e|_1=1$, for $\bbP(\cdot | B_0 \subset \cC_\infty)$--a.a.\ $\xi$
 there exists a random sequence of points $(w_k^{(e)} : k \in \bbN_+)$   such that
$ w^{(e) }_k  \in \xi \cap B (  n^{(e) } _k (\xi)   e)$, $k \in \bbN_+$ and
\begin{equation}\label{arrivoeparto} 
\lim _{ k \to \infty }  \max _{x_0\in \xi \cap B_0} \frac{|\chi (
\t_{x_0} \xi , w^{(e)}_k - x_0 )|}{k} = 0 \,.
\end{equation}
\end{corollary}

\subsection{Sublinearity on
average in $\cC_\infty$}\label{mediano} In this section we  derive
from Corollary \ref{plasmon} the sublinearity on average of the
corrector field on $\xi \cap \cC_\infty$. A similar problem is
attacked by Berger and Biskup in Section 5.2 in \cite{BB} for the
random walk on the supercritical percolation cluster. Their method
cannot be applied directly to our context, and the
adaptation of the geometric construction in \cite{BB} would lead to a tremendous technical effort.
We propose here a different construction, based on a two-scale argument, which allows us to give a self--contained
treatment of
the problem.  
The two scales refer to the fact that below the cluster $C_\infty$ is considered
at scale $K$ (i.e.\ $C^{K}_{\infty}$)
and at scale $mK$ (i.e.\ $C^{mK}_{\infty}$), where the key point is that the cluster
$C^{mK}_{\infty}$ can be made arbitrarily dense in $\bbZ^d$ by taking $m$ large.
Our target is to prove the following result:
\begin{proposition}\label{barbapapa}
For each $\e_0>0$, for $\bbP(\cdot|B_0 \subset \cC_\infty)$--a.a.\
$\xi \in \cN$ and for all $x_0 \in \xi \cap B_0$, it holds
\begin{equation}\label{800m} \lim _{n\to \infty} \frac{1}{n^d} \sum
_{x \in \xi \cap \cC_\infty\,: |x|_\infty\leq n}\, \bbI\{ |\chi
(\t_{x_0}\xi, x-x_0 )|\geq \e_0 n\}=0\,.
\end{equation}
\end{proposition}
For the reader's convenience, we isolate from the proof some
technical lemmata.

\medskip

Call $\bbB:=\{ e \in \bbZ^d\,:\, |e|_1 =1 \}$
and $\L_s:=[-s,
s]^d$. Recall the definition of the random field $\s^K$. In order to
stress the dependence on $K$, we write here $B_x^K$, $C_\infty^K$
and $\cC_\infty^K$.

Given positive numbers $C,\e $ and $m \in \bbN_+$, we consider the
Borel subsets $\cA_{C, \e, m }$ and $ \cA_{C,m}$ in $ \cN$ defined
as the family of
$\xi \in \cN$ satisfying properties (P1) and  (P2), respectively:\\

\begin{enumerate}[(P1)]
\item
 for all  $e \in \bbB$ and $N \in \bbN_+$, if $j\in \bbN_+$
satisfies $B^{mK} _{j e  } \subset \cC^{mK} _\infty\cap \L_{ mK N}$
(i.e.\ $j e \in C_\infty ^{mK}\cap \L_N$) then there exists a point
$x \in B^{mK} _{j e} \cap \xi$ such that
\begin{equation*} \max _{x_0 \in \xi\cap B^{ mK} _0}
|\chi ( \t_{x_0} \xi, x-x_0)| \leq C+ \e N\,.
\end{equation*}

\item for any $x,x' \in \xi \cap B^{ m K }_0$ one has
 \[
  |\chi ( \t_{x} \xi, x'-x)| \leq C.\]
\end{enumerate}

Let us fix $\e,\d \in (0,1)$. Thanks to property (A) (see Subsection
\ref{equoesolidale}), we can fix once for all $m$ so large
 that \begin{equation}\label{ostacoli}
 \bbP( 0 \in  C^{m K} _\infty )
\geq 1- \d\,. \end{equation} We have stated Corollary \ref{plasmon}
working with the $K$--partition of $\bbR^d$,  but trivially  the
conclusion remains valid if we work with the $mK$--partition (recall
property (A)). In particular, having fixed $\e$ and $m$, we can
find $C$ large enough that $ \bbP( \cA_{C, \e, m }|0 \in
C^{mK}_\infty  ) \geq 1- \d$.  Taking $C$ large we also have $ \bbP(
\cA_{C,m} | 0 \in C^{mK}_\infty ) \geq 1- \d$.  In particular,
\begin{equation}\label{gomma}
\bbP( \cA_{C,\e, m}\cap \cA_{C,m} | 0 \in C^{mK}_\infty ) \geq 1-
2\d\,.
\end{equation}

Given an integer $\nu$ with  $1\leq \nu \leq d$,  we call
\[
 \L_n^{\nu} :=\L_n\cap \{ x \in\bbR^d\,:\, x_i =0 \;\forall i >\nu\}\,.
\]

\begin{lemma}\label{commissione}
For $\bbP$--a.a.\ $\xi$ there exists  $ n_0= n_0 (\xi)$, depending
also on $C,m, \e, \d$, such that for all $\nu:1\leq \nu \leq d$ and
for all $n \geq n_0$ one has
\begin{equation}\label{asilonido}
\frac{1}{|\L^\nu_n \cap \bbZ^d|} \sum _{x \in  \L^\nu_n \cap
C^{mK}_\infty} \bbI( \t_{m K x} \xi \in \cA_{C,\e,m} \cap \cA_{C,m }
   )\geq 1-3  \d\,.
\end{equation}
\end{lemma}
\begin{proof}
By the ergodicity assumption (H2) and the bounds \eqref{ostacoli}
and \eqref{gomma}, we have $\bbP$--a.s.\ that
\begin{multline*}\lim _{n \to \infty}\frac{1}{ |\L^\nu_n \cap \bbZ^d |} \sum
_{x \in \L^\nu_n \cap C^{m K }_\infty} \bbI( \t_{m K  x} \xi \in
\cA_{C,\e,m}
\cap \cA_{C,m})= \\
\lim _{n \to \infty}\frac{1}{ |\L^\nu_n \cap \bbZ^d |} \sum _{x
\in \L^\nu_n \cap \bbZ^d } \bbI( \t_{m K x} \xi \in \cA_{C,\e,m}
\cap
\cA_{C,m} \cap \{ 0 \in C_\infty ^{mK}\} )=\\
\bbP( \cA_{C,\e,m} \cap \cA_{C,m} |0 \in C^{mK}_\infty) \bbP(
 0 \in C_\infty ^{mK })\geq (1-2 \d )(1-\d) >1- 3\d \,.\qedhere
\end{multline*}
\end{proof}

\medskip

Suppose now that $\xi$ satisfies  \eqref{asilonido} for all $n \geq
n_0(\xi)$ (below we take   $n\geq n_0(\xi)$). Call
\[
 G^\nu_n
 :=\{x\in\L_n^\nu\cap C^{m K}_\infty\,:\,\t_{m K x}\xi\in
 \cA_{C,\e,m} \cap \cA_{C,m}  \}\subset \bbZ^d \,.
\]
By \eqref{asilonido}, one has 
\begin{equation}\label{tuono1}
|G^\nu_n|/ |\L_n ^\nu \cap \bbZ^d|\geq 1- 3\d\,. \end{equation}
Given $a \in \bbZ^d$ and $1\leq \nu \leq d$, we set $a^\nu=(a_1,a_2,
\dots, a_\nu, 0, \dots, 0 )$.  Then we define
\begin{align*}
& \bbG _n ^\nu= \bigl\{x \in \L_n \cap \bbZ^d\,:\, x^\nu \in
G_n^\nu\}\,,\\
& \bbG_n = \bigl\{x \in \L_n \cap \bbZ^d\,:\, x^\nu \in G_n^\nu \;\;
\forall \nu: 1\leq \nu \leq d\}\,.
\end{align*}
Trivially, $\bbG_n = \cap_{\nu=1}^d \bbG_n ^\nu$.  Moreover, by
\eqref{tuono1} it holds $| \bbG^\nu_n|/ |\L_n \cap \bbZ^d|\geq 1-
3\d$ and therefore, applying   De Morgan's law,
\begin{equation}\label{lampo1} | \bbG_n |/ |\L_n \cap \bbZ^d | \geq
1- 3 d \d \,.
\end{equation}

\medskip

\begin{lemma}\label{barbalalla}
 Suppose that $ \xi$ satisfies \eqref{asilonido} and take $n \geq n_0(\xi)$. If  $x \in \xi \cap B^{mK}_a$ with $a \in \bbG_n$, then there exists
 $x^{(1)}\in \xi \cap  B^{mK}_{a^1}$ such that
$| \chi (\t_x \xi, x^{(1)}-x)| \leq Cd + \e d n$.
\end{lemma}
\begin{proof}
  Since $a \in G_n^d$ and
 $a^{d-1} \in G^{d-1}_n \subset  C_\infty ^{mK}$, by property (P1) applied to
 $ \t_{mK a} \xi$   with $N=n$ we  know that
there exists $x^{(d-1)}  \in  \xi \cap B^{mK}_{a ^{d-1} } $ such
that
\[|\chi (\t_x \xi, x^{(d-1)} -x)| \leq C+ \e n ,.\]
Repeating the above argument, we obtain that  there exist points
$x^{(i)}$, $2\leq i \leq d$, such that: $x^{(d)}=x$, $x^{(i)}\in
\cap B^{mK}_{a ^i }$ and  $|\chi(\t_{x^{(i)} } \xi, x^{(i-1)}
-x^{(i)})| \leq C+ \e n$.
In particular, by
the shift covariance property, 
\begin{equation*}
| \chi (\t_x \xi, x^{(1)}-x)| \leq \sum _{i=2}^{d} |\chi
(\t_{x^{(i)} } \xi, x^{(i-1)} -x^{(i)})|\leq d C + d \e n\,.
\qedhere
\end{equation*}
\end{proof}

\begin{lemma}\label{barbaforte}
Suppose that $ \xi$ satisfies \eqref{asilonido} and take $n \geq
n_0(\xi)$. Then, for all  $x,y \in \xi \cap ( \cup_{a\in  \bbG_n}
B^{Km }_a)$, it holds $|\chi ( \t_x \xi, y-x)| \leq 4 d C+ 3 d \e
n$.
\end{lemma}
\begin{proof}
Suppose that $x \in \xi \cap B^{mK}_a$ with  $a  \in \bbG_n $ and $y
\in \xi \cap B^{mK}_b$ with  $b  \in  \bbG _n $. Take $x^{(1)}$,
$y^{(1)}$ as in Lemma \ref{barbalalla}. By  shift covariance,
\begin{multline*}
\chi ( \t_x \xi, y-x)= \chi ( \t_x \xi, x^{(1)}-x)+\chi(
\t_{x^{(1)}}
\xi, y-x^{(1)})\\
=\chi ( \t_x \xi, x^{(1)}-x)+\chi( \t_{x^{(1)}} \xi,
y^{(1)}-x^{(1)})+ \chi ( \t_{y^{(1)}} \xi, y-y^{(1)})\,.
\end{multline*}
Again by the shift covariance (see \eqref{ariaviziata}), it holds 
$\chi ( \t_{y^{(1)}} \xi, y-y^{(1)})=- \chi (\t_y \xi , y^{(1)}
-y)$. Hence, by the bound in  Lemma \ref{barbalalla} and the
analogous estimate for  $y$ and $ y^{(1)}$,
\begin{equation}\label{cavallo}
|\chi ( \t_x \xi, y-x)| \leq 2 d C+ 2 d \e n+ |\chi( \t_{x^{(1)}}
\xi, y^{(1)}-x^{(1)}) |\,.
\end{equation}
On the other hand, $a^1 \in G_n^1$ and $b^1 \in C^{mK}_\infty$. By
property (P1) applied to $\t_{m K a^1 } \xi$, we conclude that there
exists $\bar y \in B^{mK}_{b^1}$ such that
$|\chi( \t_{x^{(1)}} \xi, \bar y -x^{(1)}) |\leq C + \e n$,
while by property (P2) applied to $ \t_{m K b^1} \xi$ we get that
$|\chi( \t_{\bar y } \xi, y^{(1)}- \bar y ) |\leq C$. The thesis
then follows observing that
the shift covariance  implies the identity
\[
\chi( \t_{x^{(1)}} \xi, y^{(1)}-x^{(1)})= \chi( \t_{x^{(1)}} \xi,
\bar y -x^{(1)})+ \chi( \t_{\bar y } \xi, y^{(1)}- \bar y
)\,.\qedhere
\]
\end{proof}

We are finally able to conclude:

\begin{proof}[Proof of Proposition \ref{barbapapa}.] We need to
show that, given $\e_0,\d_0>0$, for $\bbP(\cdot|B_0 \subset
\cC_\infty)$--a.a.\ $\xi \in \cN$ and for all $x_0 \in \xi \cap B_0$,
\begin{equation}\label{atletica} \lim _{s\to \infty} \frac{1}{|\L_s\cap \bbZ^d |} \sum
_{\substack{x \in \xi \cap \cC_\infty\,:\\ |x|\leq s}} \bbI\{ |\chi
(\t_{x_0}\xi, x-x_0 )|\geq \e_0 s \}\leq \d_0\,.
\end{equation}
 Fix $L$ such that \begin{equation}\label{paolo} K^{-d} \bbE\bigl ( \xi(B^K_0);
 \xi (B^K_0)  \geq L\bigr) \leq \d_0 /2\,.
\end{equation}
Take
 $\e:= \e_0 / (8d)$ and take $\d>0$ small enough that
\begin{equation}\label{barbamamma}
 \d <
 \frac{\bbP(0 \in C_\infty^K)}{6d}\wedge\frac{\d_0 K^d}{6d L}\,,
\end{equation}
Set $r:= \lceil s/K\rceil$ and $n:= \lceil r/m \rceil$.
 Then choose first $m$ and after that $C$ as
 in the above construction, i.e.\ choose $m$ large enough
 to assure \eqref{ostacoli}, after that choose  $C$ large
 enough to assure \eqref{gomma}.
Finally, take $s$ large enough that  $ r \geq m  n_0(\xi)$, where
$n_0(\xi)$ is  as in Lemma \ref{commissione} (this is meaningfull
for $\bbP$--a.a.\ $\xi$, and in particular for $\bbP( \cdot |B_0
\subset \cC_\infty)$--a.a.\ $\xi$).

By ergodicity, $|\L^1_r|^{-1}\sum _{j \in \L_r^1} \bbI ( j e_1 \in
C^K_\infty)$ converges to $p:= \bbP( 0 \in C^K_\infty)$ as $r \to
\infty$. In particular, for $r$ large enough (we write  $r \geq
r_1(\xi)$) the above average is larger than $p/2$.  Hence, at the
cost of losing a set of zero $\bbP$--probability and taking  $r\geq
r_1(\xi)\lor m n_0(\xi) $,  we can assume that
\begin{equation}\label{famissima}
 |\{j\in\L^1_{r}\,,\; j e_1 \in C^K_\infty\}|\geq (2r +1)p /2\,.
\end{equation}
Call $\p$ the projection of $\bbZ^d$ on its first coordinate axis
$\bbZ e_1$, namely $\p(x)=x^1=(x_1, 0,\dots, 0)$.
 Note that $\p
(\bbG_n) \subset \bbG_n$ and, due to \eqref{lampo1},
\[ 
|\p (\bbG_n) |\geq |\bbG_n|/ (2n+1)^{d-1}\geq (1- 3 d\d)(2n+1)\,. \]
In particular, \begin{equation}\label{setissima} | \{ j\in
\L^1_r\,:\, \lfloor j/m \rfloor  e_1 \in \bbG_n \}| \geq (1- 3
d\d)(2n+1)m \geq (1-3 d \d) (2r +1) \,.
\end{equation}
Since by \eqref{barbamamma}
\[
(1- 3 d \d)(2r+1)  + (2r+1)p /2 = (2r+1)  (1-3 d \d+ p/2)
>2r+1\,,
\]
 we get that the set in
the l.h.s.\ of \eqref{famissima} and the set in the l.h.s.\ of
\eqref{setissima} must intersect. Hence,
 there exists $j \in \L_{r }^1 $ such that $
B_{j e_1}^K \subset \cC_\infty^K $ and $B_{j e_1}^K \subset B^{
mK}_a$ for some $a \in \bbG_n$.

\smallskip

Thanks to Corollary \ref{plasmon} on scale $K$, $\bbP(\cdot|B_0
\subset \cC_\infty)$-a.s.\
 there exists $x \in\xi$ in the above  box $ B_{j e_1}^K$ such
that
\begin{equation}\label{barbazoo}|\chi( \t_{x_0}\xi ,x-x_0) |\leq C(\xi) + \e
r \,, \qquad \forall x_0 \in \xi \cap B_0\,.
\end{equation}
 Note that $x \in  \xi \cap ( \cup_{a\in \bbG_n} B^{mK }_a)$. For
 all
 other points   $ y \in \xi \cap ( \cup_{a\in \bbG_n} B^{mK }_a)$, by
\eqref{barbazoo}, Lemma \ref{barbaforte} and the choice $\e=
\e_0/8d$  we have
\begin{multline}
 |\chi(\t_{x_0}\xi,y-x_0)|
 \leq |\chi(\t_{x_0}\xi,x-x_0)|+|\chi(\t_{x}\xi,y-x)|
 \leq (C(\xi)+\e r)+(4d C+3drm^{-1}\e) \\
 \leq C'(\xi)+4dr\e
 \leq C'(\xi)+\e_0 r/2
 \leq C'(\xi)+\e_0 s/2\,.
 \la{member}
 \end{multline}
In particular, for $s$ large enough (\ref{member}) is smaller than
$\e_0 s$.

\smallskip

Then we can bound
\begin{multline}\label{raggio}
\sumtwo{y\in\xi\cap\cC^K_\infty\,:}{|y|_\infty\leq s}
\bbI\bigl\{|\chi(\t_{x_0}\xi,y-x_0)|\geq\e_0 s \bigr\} \leq
\sum_{z\in\L_r\cap\bbZ^d}\xi(B^K_z) \bbI(\xi(B^K_z)\geq L) \\+
  \sumtwo{a\in\L_n\cap\bbZ^d\,:}{a\not\in\bbG_n}
  \sumtwo{z\in\bbZ^d:}{B^K_z\subset B^{mK}_a}\xi(B^K_z)\bbI\bigl(\xi(B^K_z)\leq L\bigr)
  =: |\L_s\cap \bbZ^d |\bigl(A_1(s)+A_2(s)\bigr) \,.
\end{multline}
Using \eqref{paolo} and  the ergodicity in assumption (H2),
at the cost of removing a set of zero $\bbP$--probability,
we have
\begin{equation}
\lim_{s\to\infty}A_1(s)
= K^{-d}\bbE\bigl(\xi(B^K_0) ; \xi(B^K_0)\geq L\bigr)\leq \d_0/2\,.
\end{equation}
On the other hand,
\[
 A_2(s) \leq
\frac{1}{|\L_s\cap \bbZ^d | }
\sumtwo{a \in \L_n\cap\bbZ^d\,:}{a\not \in \bbG_n}
\sumtwo{z\in \bbZ^d:}{B^K_z \subset B^{mK}_a}L
\leq\frac{L m^d }{|\L_s\cap \bbZ^d | }(| \L_n\cap \bbZ^d|-|\bbG_n|)
\]
Since $s \sim n Km$, by \eqref{lampo1}  we have that $ \varlimsup_{s
\to \infty} A_2(s) \leq L K^{-d} 3 d\d$, which is smaller than
$\d_0/2$ by our choice of $\d$ (cf.\ \eqref{barbamamma}). Coming
back to \eqref{raggio}, we get the thesis.
\end{proof}

\subsection{Sublinearity on average in $\cC^*_\infty$}\label{ponticello}
We now  need  to come back to the set of good points $\xi \cap \cC^*_\infty$.
\begin{lemma}\label{merendinabuona}
 If  $\cA \subset \cN$  is  a measurable set such that $\bbP(\cA| B_0 \subset \cC_\infty)=1$, then
 $\bbP ( \cA | B_0 \subset \cC^*_\infty)=1$.
\end{lemma}
\begin{proof}
Since  $\cC^* _\infty \subset \cC_\infty$,  the set  
$
 \cB:= \{ B_0 \subset  \cC^*_\infty\} \setminus \cA$ is contained by
 the set $\cD:=
  \{ B_0 \subset \cC_\infty \} \setminus \cA$.
Therefore we have the following sequence of implications
\begin{multline}
\bbP ( \cA | B_0 \subset \cC^*_\infty)<1 \Rightarrow  \bbP( \cB| B_0 \subset \cC^*_\infty) >0 \Rightarrow \bbP(\cB)>0 \Rightarrow
\bbP ( \cD)>0 \\
\Rightarrow  \bbP ( \cD|B_0 \subset \cC_\infty)>0 \Rightarrow \bbP( \cA| B_0 \subset \cC_\infty) <1\,,\hskip2cm
\end{multline}
showing the contrapositive.
\end{proof}

 From Proposition \ref{barbapapa} and Lemma \ref{merendinabuona}
  we easily obtain the following  
 \begin{corollary}\label{winkx} Given $\e>0$,
for $\bbP(\cdot|B_0 \subset \cC^*_\infty)$--a.a.\ $\xi
\in \cN$ and for all $x_0 \in \xi \cap B_0$, 
\begin{equation}\label{8000m} \lim _{n\to \infty} \frac{1}{n^d} \sum
_{x \in \xi \cap \cC^*_\infty: |x|\leq n}\, \bbI\{ |\chi
(\t_{x_0}\xi, x-x_0 )|\geq \e n\}=0\,.
\end{equation}
\end{corollary}

\subsection{Strong sublinearity in $\cC^*_\infty$}\label{puntiglioso}


\begin{lemma}
\label{pisolino}
For each $\e>0$, for $\bbP(\cdot | B_0 \subset \cC^*_\infty) $--a.a.\
$\xi $ and for all $x_0 \in \xi \cap B_0$, 
\begin{equation}\label{simonetta}
 \lim _{n\to \infty} \frac{1}{n} \max_{ \substack{x \in \xi
\cap \cC^*_\infty\,:\\ |x|_\infty \leq n}} | \chi (\t_{x_0}\xi ,x-x_0 )| =0\,.
\end{equation}
\end{lemma}
\begin{proof}
Let us define \begin{equation}\la{errenne} R_n (\xi) =\max _{x_0 \in \xi \cap B_0} \; \max_{ \substack{z \in \xi \cap \cC^*_\infty:\\
|z|_\infty\leq n}}| \chi (\t_{x_0}\xi ,z-x_0 )|\,.\end{equation} Due to
Lemma \ref{polgrowthbis} and Lemma \ref{merendinabuona}, for $\theta
>d+1$
\begin{equation}\label{telaviv}
\lim _{n \to \infty} n^{-\theta} R_n =0 \,, \qquad \bbP(
\cdot |B_0 \subset \cC^*_\infty)\text{--a.s.}
\end{equation}
Following an idea of Y.\ Peres,   we only need to prove  a recursive
bound of the form:
%
for each $\e, \d>0$, there exists an a.s.\ finite
random variable $n_0=n_0(\xi, \e,\d)$ such that
\begin{equation}\label{canegianni}
R_n \leq \e n + \d R_{5n}\,,
\qquad n \geq n_0\,.
\end{equation}
From (\ref{canegianni}), using the input \eqref{telaviv}, it is easy to conclude;
see the explanation after \cite[Lemma 5.1]{BP}.


We turn to the proof of (\ref{canegianni}).
We take   $\xi \in \cN$ such that $\{B_0 \subset \cC^*_\infty\} $ and  satisfying \eqref{telaviv}.
Moreover, assume
$\xi$  and $b_n=o(n^2)$ satisfy
  \eqref{mirtillo1} and  \eqref{mirtillo2} of Proposition \ref{mirta}.
Take $x_0,z $ such that $R_n(\xi)=| \chi
(\t_{x_0}\xi ,z-x_0 )|$, $x_0 \in \xi \cap B_0$, $z \in \xi \cap \cC^*_\infty$, $|z|_\infty \leq n$. Similarly to \cite{BP}, take
 $t= t(n) \geq
b_{4n} \lor n $ (we will specify the function $t(n)$ at the end).
  Fix positive constants $C_1, C_2$ such that
   the expressions $\max_x \sup_t $ in \eqref{mirtillo1} and \eqref{mirtillo2} are bounded by $C_1$ and $C_2$, respectively, if
$n $ is large enough (that is,   $n\geq  n_*(\xi)$). Take $n \geq K
\lor n_* (\xi)$. Finally, define the stopping time
\[ S_n :=\inf\{t \geq 0:|Y_t-z|_\infty\geq 2n\}\,. \]
 Due to  Corollary \ref{mariasole} and the Optional
Stopping Theorem we can write
\begin{equation}\label{eq1}
 E_{z,\xi} \Big[ \chi( \t_z \xi, Y_{t\wedge S_n} -
z)+ Y_{t \wedge S_n}-z\Big] =  \chi( \t_z \xi, 0)=0\,.
\end{equation}
By the shift--covariance property we can write
\begin{equation}\label{eq2}
 \chi (\t_{x_0}\xi ,z-x_0
)=  \chi (\t_{x_0}\xi ,Y_{t\wedge S_n} -x_0 )-\chi( \t_z \xi,
Y_{t\wedge S_n} -z)\,.
\end{equation}
Combining \eqref{eq1} and \eqref{eq2} we get
\begin{equation}\label{eq3}
\chi (\t_{x_0}\xi ,z-x_0 )=E_{z,\xi} \Big[ \chi (\t_{x_0}\xi
,Y_{t\wedge S_n} -x_0 )+Y_{t \wedge S_n}-z\Big]\,,
\end{equation}
thus implying that $R_n (\xi) \leq E_{z,\xi} \Big[ \bigl |F(\xi,
Y_{t\wedge S_n} )\bigr| \Big]$ where
  \begin{equation}
F(\xi, Y_{t\wedge S_n} ):= \chi (\t_{x_0}\xi ,Y_{t\wedge S_n} -x_0
)+Y_{t \wedge S_n}-z\,.
\end{equation}

\medskip

Let us  introduce the event $A=\{S_n <t\,,\, |Y_{S_n}-z|_\infty > 4
n\}$.  Using \eqref{telaviv} we can bound
\begin{multline}\label{kafka}
E_{z,\xi} \Big[ \bigl |F(\xi, Y_{t\wedge S_n} )\bigr|\,;\, A
\Big]\leq \sum_{k=4}^\infty E_{z,\xi} \Big[ \bigl |F(\xi, Y_{t\wedge
S_n} )\bigr|\,;\, A \,;\, |Y_{S_n}-z|_\infty \in [kn,(k+1)n)
\Big]\\
\leq  \sum_{k=4}^\infty\bigl\{C(\xi) k^\theta n ^\theta+ (k+1)n
\bigr\}
P_{z,\xi} \Big[ | Y_{S_n}-z|_\infty  \in[kn,(k+1)n)\,, \; S_n <t \Big] \\
\leq 
C' 
n^{\theta}  \sum_{k=4}^\infty k^{\theta } P_{z,\xi}
\Big[ | Y_{S_n}-z|_\infty  \in[kn,(k+1)n)\,, \; S_n <t \Big] \,.
\end{multline}
Recall that $Y_t$ is the poissonization of the discrete time random
walk $n \mapsto X_{T_n}$.
Taking $t=\vartheta n^2$, with $\vartheta>0$, and using 
Lemma \ref{moltobuono}, we can estimate
\begin{equation*}
\begin{split}
&P_{z,\xi}  \biggl[ | Y_{S_n}-z|_\infty    \in[kn,(k+1)n)\,, \; S_n
<t \biggr] \leq
(kn)^{-\g} E_{z,\xi}\biggl [ |Y_{S_n }-z|^\g ; S_n <t\biggr]\\
& \leq (kn)^{-\g} \sum _{j=1}^\infty P(N_t=j) E_{z,\xi}
\biggl[\max_{ i: 1\leq i \leq T_j } |X_i-z|^\g  \biggr]
 \\ &
\leq  (kn)^{-\g}[1+\log (1+ c(d) n)]^\g  \sum _{j=1}^\infty P(N_t=j)
j
 =
 (kn)^{-\g}[1+\log (1+ c(d) n)]^\g \vartheta n^2 \,.
 \end{split}
\end{equation*}
Coming back to \eqref{kafka} and taking $\g>2+\theta$,  we get
\begin{equation}\label{fuocone}
E_{z,\xi} \Big[ \bigl |F(\xi, Y_{t\wedge S_n} )\bigr|\,;\, A
\Big]\leq  c\vartheta  n^{\theta +2- \g}[1+\log (1+ c(d) n)]^\g \sum
_{k=4}^\infty k^{\theta-\g}  \leq (\e/2)n
\end{equation}
for $n$ large enough.

The above expectation, coming from the presence of long jumps,  does
not appear in \cite{BP}. On the other hand, using Proposition
\ref{mirta} and taking  $t=\vartheta n^2$  with $\vartheta>0$, the
control of $E_{z,\xi} \bigl[ \bigl |F(\xi, Y_{t\wedge S_n}
)\bigr|\,;\, A^c \bigr]$  can be obtained by the same computations
in  the proof of Lemma 5.1 in \cite{BP}. As a final result, one gets
\eqref{canegianni}.
\end{proof}

\subsection{Proof of Theorem \ref{piccolo}}\label{mandarino}
We now have most of  the tools needed to prove the sublinearity of the corrector field stated
in Theorem \ref{piccolo}. First, we need to link the Palm distribution to the probability measures used above. To this end
we introduce a bridge
probability space. We
call  $\bbQ_0$   the distribution on $\cN_0$  of the point process
$\xi$ defined in this way: pick a configuration $\tilde \xi  \in
\cN$ with law $\bbP( \cdot | \tilde \xi (B_0)\geq 1) $, pick a point
 $v_0 \in \tilde \xi \cap B_0$ with uniform probability, and set $\xi=
\t_{v_0} \tilde \xi$.

\begin{lemma}
Let $\cA\subset \cN_0$ be a measurable subset. Then
 $\bbQ_0(\cA)=0,1$ if and only if $\bbP_0 (\cA)=0,1$, respectively.
\end{lemma}
\begin{proof}
By taking the complements, it is enough to prove that $\bbQ_0(\cA)=1$ if and only if $\bbP_0(A)=1$.
Consider the measurable set
$
 \cB
 :=\{ \xi \in \cN\,:\, \t_x \xi \in \cA\,\; \forall x \in \xi\cap
 B_0\}$.
 By the Campbell identity \eqref{campbell} with $f(x,\xi)= \bbI (x \in
B_0\,; \, \xi\in \cA\}$ we have
\begin{equation*}
 \bbP_0(\cA)= \frac{1 }{\rho K^d } \int _{\cN}  \bbP( d \xi) \int
_{B_0} \xi (dx )  \bbI( \t_x \xi \in \cA)\leq \frac{1 }{\rho K^d }
\int _{\cN}  \bbP( d \xi) \int _{B_0} \xi (dx )  1 =1\,,
\end{equation*}
which implies that $\bbP_0(\cA)=1$ if and only if $\bbP(\cB)=1$.

Similarly, by definition of $\bbQ_0$  we have
\begin{multline*}
\bbQ_0( \cA)= \frac{1}{\bbP( \xi (B_0) \geq 1)} \int
_{ \{ \xi: \xi(B_0) \geq 1\}}\bbP( d \xi) \frac{ \int _{B_0 } \xi
(dx ) \bbI ( \t_x \xi \in \cA) }{\xi (B_0) }\\
\leq \frac{1}{\bbP(
\xi (B_0) \geq 1)} \int _{ \{ \xi: \xi(B_0) \geq 1\}}\bbP( d \xi)
\frac{ \int _{B_0 } \xi (dx ) 1 }{\xi (B_0) }=1 \,,
\end{multline*}
which implies that $\bbQ_0(\cA)=1$ if and only if $\bbP(\cB)=1$.
\end{proof}

Thanks to the above lemma, to prove Theorem \ref{piccolo}
it suffices to show that
 for $\bbP( \cdot | \xi (B_0)\geq 1)$--a.a.\ $\xi$,
\begin{equation}\label{lafine?}
\lim _{n\to \infty} \frac{1}{n} \max _{x_0 \in \xi \cap B_0}\;\; \max _{\substack{ x \in \xi:\\ |x|_\infty \leq n}}
\;| \chi ( \t_{x_0} \xi, x- x_0)| =0\,.
\end{equation}
The plan of the proof is the following: we first   improve Lemma \ref{pisolino} by passing from $\bbP( \cdot |B_0 \subset \cC^*_\infty)$
to $\bbP( \cdot | \xi(B_0) \geq 1)$ (see Lemma \ref{saratoga});
after that we remove the restriction $x \in\cC^* _\infty$ which appears in Lemma \ref{saratoga} by applying
the Optional Stopping Theorem.

\medskip

 We fix a coordinate vector $e$ and define the map $n_*:  \cN \to \bbN_+:=\{1,2, \dots\}$
 as follows:
 \begin{equation}\label{naso} n_*(\xi)= \min\{ n\in \bbN_+\,:\, B_{n e}\subset
\cC^*_\infty\} \,.
\end{equation}
By  assumption (H1), the map 
is well defined
$\bbP$--a.s.
\begin{lemma}\label{mimuovo}
Call $P$ the law on $\cN $ of the
point process $ \t_{n_*(\xi) K e } \xi$, where   $\xi\in \cN $ is
sampled with probability $\bbP ( \cdot | \xi (B_0)\geq 1)$. Then,
for any measurable subset $\cA \subset \cN$, 
\[ \bbP( \cA \,|\, B_0 \subset \cC^*_\infty )=1  \Rightarrow P(\cA)=1\,. \]
\end{lemma}
\begin{proof}
 Given a bounded
 measurable  function $f :\cN
\to \bbR$ we can write
\[
 E_{P}(f)
 = \int \bbP\bigl( d \xi\,|\, \xi (B_0) \geq 1\bigr) f( \t_{n_*(\xi)K e}\xi)
 = \frac{\bbE\left[
  f(\t_{n_*(\xi)K e} \xi);\xi (B_0) \geq 1  \right]}{
    \bbP( \xi (B_0) \geq 1)}\,.
\]
 Moreover,
\begin{multline}
 \bbE\left[
f (\t_{n_*(\xi)K e} \xi);\xi (B_0) \geq 1 \right]= \sum _{m=1}
^\infty  \bbE\left[ f( \t_{m K e}  \xi); \xi
(B_0) \geq 1; n_*(\xi) =m  \right]\\
=\sum_{m=1}^\infty \bbE\bigl[ f( \t_{mK e} \xi)\,;\, \xi (B_0)\geq
1\,; B_{ke} \not \subset \cC^*_\infty \; \forall  k:1\leq k \leq
m-1\,,\; B_{m e} \subset \cC^*_\infty\bigr]\,.
\end{multline}
Due to the stationarity of $\bbP$ the last expression can be written
as
\begin{multline}
\sum_{m=1}^\infty \bbE\bigl[  f(  \xi)\,;\, \xi (B_{-m e })\geq 1\,;
B_{ke} \not \subset \cC^*_\infty \; \forall  k:1-m\leq k \leq -1\,;
B_0 \subset
\cC^*_\infty \bigr]\\
=
\bbE\bigl[ f(\xi) \hat n (\xi) \,; B_0 \subset \cC^*_\infty
\bigr]\,,
\end{multline}
where we define
\begin{align*}
& n_-(\xi)= \max \{j\leq -1\,:\, B_{j e} \subset \cC^*_\infty\}\,,\\
& \hat n (\xi) = \sharp \{j\,:\, n_-(\xi) \leq j \leq -1\,, \;
\xi(B_{je}) \geq 1\}\,.
\end{align*}
Collecting the above observations we get
\begin{equation}
E_P(f)= \frac{
\bbE\left[ f(\xi) \hat n (\xi) \,; B_0 \subset \cC^*_\infty  \right]
}{ \bbP( \xi (B_0) \geq 1) } = \bbE \left[ f(\xi) \hat n(\xi)\,;\,
B_0\subset \cC^*_\infty  \,|\, \xi (B_0) \geq 1\right]\,.
\end{equation}
This proves that
\begin{equation}\label{importante}
 P( d \xi)=\hat n(\xi)
\bbI( B_0\subset \cC^*_\infty )  \bbP( d\xi \,|\, \xi (B_0) \geq 1)=
\hat n(\xi)
\frac{ \bbI( B_0\subset \cC^*_\infty ) }{ \bbP ( \xi(B_0) \geq 1) } \bbP( d\xi)
\,.
\end{equation}
Since $P$ is a probability measure it must hold $ 1=\bbE\left[\hat
n(\xi); B_0\subset \cC^*_\infty \right]/\bbP[ \xi (B_0) \geq 1]$.
Take $\cA\subset \cN$ measurable and satisfying $\bbP( \cA\,|\,B_0
\subset \cC^*_\infty )=1$. This implies that
$
 \bbI( \xi \in \cA) \bbI( B_0 \subset \cC^*_\infty)
 =\bbI ( B_0 \subset \cC^*_\infty)$ $\bbP$--a.s.
In particular,
 we
have
\[
 P( \cA)
 =\frac{
  \bbE\bigl[\hat n(\xi)\bbI(B_0\subset\cC^*_\infty)\bbI(\xi\in\cA)\bigr]
  }{P(\xi(B_0)\geq 1)}
 =\frac{\bbE\bigl[\hat n(\xi)\bbI(B_0\subset\cC^*_\infty)\bigr]}{
  P(\xi(B_0)\geq 1)}\,,
\]
while we have already shown that the last member equals $1$.
\qedhere
\end{proof}

Recall the definition (\ref{errenne}) of $R_n$.
\begin{lemma}\label{saratoga}
For $\bbP( \cdot | \xi (B_0) \geq 1 )$--a.a.\ $\xi$,  it holds
$
 \lim _{n\to \infty} R_n/n
 =0$.
\end{lemma}
\begin{proof}

Recall the definition of the  function $n_*(\xi)$ given
in \eqref{naso}.
As a byproduct of
 Lemma \ref{pisolino}  and  Lemma \ref{mimuovo}, we get  for
 $\bbP\bigl( \cdot | \xi(B_0) \geq 1\bigr)$--a.a.\ $\xi$ that, for any $x' \in
 \xi \cap B_{n_*(\xi) e}$, it holds
 \begin{equation}\label{girolamo}
 \lim _{n\rightarrow \infty}\frac{1}{n}
  \max _{
 \substack{y \in\xi \cap \cC_\infty ^*\\
 |y-n_*(\xi) Ke |_\infty \leq 2n }}
 \bigl| \chi \bigl(\t_{x'} \xi, y-x'\bigr)
 \bigr|=0
 \end{equation}
(the above  choice of $2n$ instead of  $n$ is due to later
applications). On the other hand,
  by the shift covariance, given $x_0 \in \xi \cap B_0$, $x \in \xi \cap \cC^*_\infty $ and $x' \in \xi \cap B_{n_*(\xi) e}$, we can write
\begin{equation}\label{lezione}
 \chi( \t_{x_0} \xi, x-x_0)
 =\chi( \t_{x_0} \xi, x' -x_0)+ \chi (\t_{x'}  \xi , x-x' )\,.
\end{equation}
Since, given $\xi$ and $x \in \xi \cap \cC^*_\infty $  with
$|x|_\infty \leq n$, it holds   $|n_*(\xi) K e -x| _\infty \leq 2 n$
for $n$ large enough, one can apply \eqref{girolamo} with $y=x$.
From \eqref{girolamo} and \eqref{lezione} we then obtain
\begin{equation}
\lim _{n\rightarrow \infty} \frac{1}{n} \max _{x _0 \in \xi \cap
B_0} \max _{\substack{x \in \xi \cap \cC_\infty^*\\ |x|_\infty \leq
n } } \bigl| \chi( \t_{x_0} \xi, x-x_0) \bigr|=0\,,
\end{equation}
which corresponds to the thesis.
\end{proof}

Let us finally conclude. Take $x_0 \in B_0 \cap \xi$ and $x \in \xi $ with $|x|_\infty \leq n$. From the Optional Stopping Theorem (cf.\ Proposition \ref{pizza})
we know that
\be
x+\chi ( \t_{x_0 } \xi, x - x_0) = E_{x, \xi}\left[X_{T_1}+\chi ( \t_{x_0 } \xi, X_{T_1} - x_0)
\right]\,.
\end{equation}
Take $\e<1$ and let $\cS_k=\{(k-1)n^\e\leq |X_{T_1}-x| < kn^\e \}$, $k=1,2,\dots$.
Recalling \eqref{errenne}, we can estimate
\be 
|\chi ( \t_{x_0 } \xi, x - x_0)| \leq n^\e + R_{2n} + \sum_{k=2}^\infty (kn^\e + R_{n+kn^\e})P_{x,\xi}(\cS_k)   \,.
\end{equation}
Lemma \ref{saratoga} gives us the estimate $R_n=o(n)$, and therefore we see that the desired conclusion follows if we can show that a.s.
$\sum_{k=2}^\infty (n+kn^\e)P_{x, \xi}(\cS_k) = o(n)$.
This, in turn, follows from Lemma \ref{moltobuono} and the fact that
$|x|_\infty\leq n$. Indeed, for every $k\geq 1$, and $p\geq 1$:
\[
P_{x, \xi}(\cS_{k+1})\leq P_{x, \xi}(|X_{T_1}-x| \geq kn^\e)\leq
\k\,(kn^\e)^{-p}\,(\log n)^p\,.
\]
Taking e.g.\ $p=3$ we get the desired bound. 

\section{Proof of Theorem \ref{atreiu}  in the presence of  energy marks}\label{estensione}
Let us suppose now that the function $u(E_x,E_y)$ is non trivial. In
this case, the environment of the random walk is given by
$\o = \bigl\{ (x, E_x) : x \in \xi \bigr\} $ and  corresponds to a marked simple point process. We refer to
\cite[Section 2]{FSS} for detailed definitions and
references.
Here we simply recall that  stationarity and ergodicity of the point process $\xi$ automatically extend to $\bbP$.
Moreover, the Campbell identity remains valid in the marked case
(with suitable changes).

We fix some notation. We write $\tilde \cN$ for the state space of the  marked point process $\o$ and, given $v\in \bbR^d$, we define the
 translation $\t_v \o$ as
\[
 \t_v \o:=\bigl\{ (x -v ,E_x) : x \in \xi \bigr\}\,, \qquad
 \o=\bigl\{ (x, E_x) : x \in \xi \bigr\} \,.
\]
 Let us suppose that
$\bbP$ is an ergodic stationary marked  simple point process with
finite second  moment. As already mentioned assumption (H1) is the
same, while assumption (H2) has to be slightly modified as follows:

\begin{itemize}
\item[(H2)]  for each $K>0$ and for each
  vector $e \in \bbZ^d$ with $|e |_1 =1$, consider the product  probability space $\Theta:=\wt\cN \times
  [( [0,K)^d \times \bbR)  \cup \{ \partial \} ] ^{\bbZ} $ whose elements $\bigl(\o, (a_i : i \in \bbZ) \bigr)$ are sampled as follows:
  choose $\o=\bigl\{ (x, E_x) : x \in \xi \bigr\} $ with law  $ \bbP$, afterwards  choose independently  for  each index $i$  a point $b_i \in \xi \cap B( i e)$ with uniform probability
  and set $a_i:= (b_i - i K e , E_{b_i})$ (if $\xi \cap B(ie)=\emptyset$, set $a_i = \partial$). We assume that  the resulting law $P^{(K, e)} $ on
 $\tilde \cN \times
  [( [0,K)^d\times \bbR)  \cup \{ \partial \} )
 ]^{\bbZ} $
  is ergodic w.r.t.\ the transformation
 \begin{equation}\t :\bigl(\o, (a_i : i \in \bbZ) \bigr) \to \bigl( \t_{Ke } \o , ( a_{i+1} : i \in \bbZ ) \bigr) \,.
 \end{equation}
\end{itemize}

\medskip

The proof that the marked PPP satisfies the new assumption (H2) is
similar to the non--marked case. The
 proof of Theorem
\ref{atreiu} in the presence of the energy marks can be obtained by
a straightforward extension of the proof presented in the
non--marked case. Indeed,  the presence of the energy marks is
rather painless since the weights $e^{-u(E_x,E_y) }$ are uniformly
bounded from above and from below by some positive constants.
Furthermore, the following covariant property, implicitly used in
the non--marked case,  holds: writing $p_{\o} (x,y)$ for the jump
probability of $X_n$ in the environment $\o$, then $ p_{\o}(x,y)=
p_{\t_v \o} (x-v,y-v)$ for all $v \in \xi$.

\appendix

\section{Strong invariance principle for Mott random walk on diluted lattices}
In this appendix we discuss the quenched invariance principle for diluted lattices.
The proof differs from the one of Theorem \ref{atreiu} in few points (mainly related to ergodicity) that we comment below.
In order to simplify the notation, we disregard the energy marks (all the arguments can be easily adapted  to the marked case).

 We start with a lattice $\G$ (or crystal, cf.\
\cite{AM}): $\G$ is  a locally finite  set $\Gamma\subset \bbR^d$
such that for a suitable basis $v_1, v_2, \dots, v_d$ of $\bbR^d$,
it holds
\begin{equation}\label{paranza}
\G -x =\G \, \qquad \forall x\in G:=\bigl \{ z_1v_1+z_2v_2 +\cdots
+z_d v_d \,:\, z_i \in \ZZ \;\; \forall i  \bigr\}\,.
\end{equation}
Let $\D$ be the elementary cell defined as $\bigl \{t_1 v_1+t_2 v_2
+\cdots+t_d v_d \,:\, 0\leq t_i <1\;\; \forall  i \bigr\}$. (Note
that both the group  $G$ and the cell $\D$ depend on the basis $v_1,
v_2, \dots, v_d$.)

Let $\omega=\bigl(\omega_x\,:\, x\in \G\bigr)$ be a site Bernoulli
percolation on $\G$ with parameter $p\in (0,1]$. For each $u\in \G \cap \D$
we call $P_u$ the law on $\cN_0$ of the random point process  given by
$\{0\}\cup \{x-u: x\in \G\,,\; \o_x=1\}$  and we consider $\bbP_0 =
\frac{1}{|\D\cap \G|} \sum _{u\in \D\cap \G} P_u $. As proved in \cite{FM}, $\bbP_0$ does not depend on the basis $v_1, \dots, v_d$ and on the fundamental cell $\D$, moreover $\bbP_0$ is indeed the Palm distribution of the
 stationary point process with law $\bbP$ realized as $ \{ x-V: x \in \G, \; \o_x =1 \}$ where $V$ is a random vector uniformly distributed in the fundamental cell $\D$, independent from the field $\o$. Finally, we call $P$ the law of the point process $\{ x \in \G:\o_x=1\}$.

It is simple to check that both the discrete--time and the continuous--time Mott random walks are well defined  $\bbP_0$--a.s., $\bbP$--a.s.\ and $P$--a.s.
Moreover, as in Theorem \ref{atreiu} and Corollary \ref{strong}, proving the invariance principle for
$\bbP_0$--a.a.\ $\xi$ with starting point $x_0=0$, one  automatically gets the  strong invariance principle.

 The corrector field is defined as in Section
\ref{ascolipiceno}.  By applying a linear isomorphism, we can assume
that the basis $v_1, \dotsc, v_d$ coincides with the coordinate
basis of $\bbZ^d$ and therefore that $\D=[0,1)^d$.  We restrict to
$K \in \bbN_+$. Then  under $P$, the point processes $B_K(z) \cap
\t_{Kz}\xi$ with $z \in \bbZ^d$ are i.i.d.  In particular, $P$ is
stationary and ergodic w.r.t.\ the translation   $\t_{K v_i}$.
Moreover,
 sampling $\xi$ with law $P$, the random field $\s^K(\xi)$ is a Bernoulli random field, supercritical if $K$ is taken large enough. Define $C_\infty$ as its unique infinite cluster and define $\cC_\infty$ as before. In  the definition of the  law $\bbP_*$ on the space $\O$ given in Section \ref{mediano1},  replace $\bbP$ with $P$. With this trick, $\bbP_*$ remains ergodic and stationary w.r.t.\ to the map $\cT$ defined in Section \ref{mediano1} and one can prove the sublinearity of the corrector field along a given direction. At this point,
substituting $\bbP$ with $P$, the proof of the quenched invariance
principle follows the same main steps of the proof of Theorem
\ref{atreiu}, even with huge simplications. Indeed, working with
diluted lattices overcrowded regions are absent. In particular,
taking $T_0$ large enough, the field $\vartheta ^{K,T_0}$ coincides
with $\s^K$.

\section{Miscellanea}\label{natalina}

      We start with a key technical lemma: 
 \begin{lemma}\label{techn_lemma}
Let $\bbP _0$ be the Palm distribution
  associated to a stationary  simple   point process $\bbP$ with finite density.  
\begin{enumerate}[(i)]
\item Suppose that $\rho_2 =\bbE\(\xi([0,1)^d)^k\)< \infty$, and let  $f:\bbR^d \times \cN_0 \to \bbR$ be a
measurable function satisfying $\bbE_0\Big[ \sum_{x \in \xi} |f (x,
\xi  )|\Big]<\infty$. Then $\bbE \Big[ \sum _{x \in \xi} | f( -x,
\t_x \xi) |\Big]<\infty$ and
\begin{equation}\label{kiwi1}
 \bbE_0\Big[ \sum_{x \in \xi} f (x, \xi  )\Big]= \bbE \Big[ \sum _{x \in \xi}  f( -x, \t_x \xi) \Big] \,.
\end{equation}
\item 
Let $n$ be a nonnegative integer such that
$\rho_{n+1}<\infty$. Then
\begin{equation}
\bbE_0\Big[ \bigl( \sum_{x \in \xi}      e^{-\g |x|^\a }
\bigr)^n\Big] <\infty\,, \qquad \forall \a, \g>0\,.
\end{equation}
\end{enumerate}
\end{lemma}

\begin{proof}
Part (ii) can be derived generalizing  the proof   Lemma 2 of
\cite{FSS}. The proof of part (i) uses some arguments taken from the
proof of Lemma 1 (i) in \cite{FSS}. We give some more details.
Without loss of generality we can assume that $f\geq 0$. We define
$F(\xi)=\sum _{x \in \xi} f(x,\xi)$ and $G(\xi)= \sum _{x \in \xi}
f(-x, \t_x \xi)$. Note that, given $u \in\xi$, it holds $ F(\t_u
\xi)= \sum _{ y \in \xi} f(y-u, \t_u \xi) $ and $  G( \t_u \xi)=
\sum_{y \in \xi} f (u-y, \t_y \xi)$. In particular,  taking $L>0$
and setting $\L_L=[-L/2,L/2]^d$,  by the  Campbell identity we can
write $ \bbE_0[F] = A(L)+B(L)$ and $\bbE_0[G]=A(L)+ C(L)$ where
\begin{align*}
& A(L)= \frac{1}{\rho L^d} \bbE\Big[ \sum _{u \in \xi \cap \L_L}
\sum _{y \in \xi\cap \L_L } f(y-u, \t_u \xi ) \Big]= \frac{1}{\rho
L^d} \bbE\Big[ \sum _{y \in \xi \cap \L_L} \sum _{u \in \xi\cap \L_L
} f(u-y, \t_y \xi ) \Big]
\,,\\
& B(L)= \frac{1}{\rho L^d} \bbE\Big[ \sum _{u \in \xi \cap \L_L}
\sum _{y \in \xi \setminus \L_L } f(y-u, \t_u \xi ) \Big]\,,\\
& C(L)=\frac{1}{\rho L^d} \bbE\Big[ \sum _{u \in \xi \cap \L_L} \sum
_{y \in \xi \setminus \L_L } f(u-y, \t_y \xi ) \Big]\,.
\end{align*}
Since
\begin{align}
& \sum _{u \in \xi \cap \L_L} \sum _{y \in \xi\setminus \L_L }
f(y-u, \t_u \xi ) \leq \sum _{u \in \xi \cap \L_L} \sum _{ x \in
\t_u \xi \setminus \L_{2L}} f(x, \t_u \xi) \,,\\
& \sum _{u \in \xi \cap \L_L} \sum _{y \in \xi \setminus \L_L }
f(u-y, \t_y \xi ) \leq  \sum_{u \in \xi\cap \L_L} \sum _{x \in \t_u
\xi \setminus \L_{2L} } f(-x, \t_x \t_u \xi)\,,
\end{align}
by  Campbell's identity we can bound $B(L)$ and $ C(L)$
from above by 
$\bbE_0\big[ \sum _{x \in \xi \setminus \L_{2L}} f
(x,\xi) \big]$ and  $\bbE_0 \big[ \sum _{x \in \xi\setminus \L_{2L}
} f(-x, \t_x \xi )\big]$, respectively. Suppose for the moment that
$f(x, \xi)$ is bounded and $f(x,\xi)=0$ if $ |x|_\infty \geq \ell$ for some
positive $\ell$. This assures that all the above expectations are
bounded (we invoke part (ii) and the assumption $\rho_2<\infty$). In
particular, by the Dominated Convergence Theorem, we conclude that
$B(L)$ and $C(L)$ go to zero as $L\to  \infty$. As a consequence, it
holds $\bbE_0[F]= \bbE_0[G]$, which is simply the thesis in point
(i). On the other hand, given a general nonnegative function $f$ and
a constant $\ell>0$, we can define the cutoff
$$ f_\ell ( x,\xi  ) = \begin{cases} f(x,\xi ) & \text{ if } |x|_\infty \leq
\ell \,,\; f ( x,\xi  ) \leq \ell\,,\\
0 &\text{ otherwise}\,.
\end{cases}
$$
Then the thesis holds for $f_\ell$ (by what was proved above) and
extends to $f$ by the Monotone Convergence Theorem.
\end{proof}
\medskip

\begin{lemma}\label{francoforte}
Given a measurable subset $\cA _0\subset \cN_0$, define $\cA \subset \cN$ as
\[
 \cA = \{ \xi \in \cN\,:\, \t_x \xi \in \cA_0 \; \; \forall x \in \xi \} \,.
\]
Then $\bbP_0 (\cA_0)=  1$ if and only if $\bbP (\cA)= 1$.
\end{lemma}
\begin{proof}
Given $L>0$ we set $\L_L=[-L,L]^d$ and  we apply the Campbell identity \eqref{campbell} to the function $f(x,\xi):=
\bbI( x \in \L_L \,;\, \xi \in \cA_0)$:
\begin{equation*}(2L)^d= (2L)^d \bbP_0(\cA_0)= \rho^{-1}  \bbE\Big[ \sum_{x \in \xi\cap\L_L} \bbI (\t_x \xi \in \cA_0 ) \Big]\leq
\rho^{-1} \bbE\Big[\xi ( \L_L) \Big]= (2L)^d\,.
\end{equation*}
Hence,
 all members in the above expression must be equal. In particular, $\bbP$--a.s.\ it holds $ \t_x \xi \in \cA_0$
for all $x \in \xi \cap \L_L$.
Using the arbitrariness of $L$, we conclude.
\end{proof}

\begin{lemma}\label{arancia}
Suppose that  $\bbP$ is the law of a stationary ergodic marked
simple point process with finite second moment, or a marked diluted
lattice.
Then, both for
$\bbP$ and for $\bbP_0$--a.a. $\o$, the DTRW and the CTRW are well
defined for any starting point $x_0 \in \xi$.
\end{lemma}
\begin{proof}
First, we point out that Lemma \ref{francoforte} holds also in the
marked case and the proof is very similar. By the assumption of
finite second moment (recall that  diluted lattices have finite
moments of all orders) it holds $\bbE_0[w(0)] < \infty$. This
implies that for $\bbP_0$--a.a.\ $\o$ and for all $x \in \xi $ it
holds $w(x)< \infty$. By Lemma \ref{francoforte}, the same property
is fulfilled $\bbP$--a.e. As a  consequence, the DTRW is well
defined. The claim for the CTRW follows from \cite[Prop. 10]{FSS}
and Lemma \ref{francoforte} (diluted lattices can be treated apart
since due to the uniform density bounds the proof becomes trivial).
\end{proof}

\bigskip

\subsection*{Acknowledgements} We kindly  thank S.\ Popov for very
useful discussions. P.C.\ and A.F.\ acknowledge the financial support
of the European Research Council through the ``Advanced Grant'' 
PTRELSS 228032. T.\ Prescott thanks the Department of Mathematics of
Universit\`a  Roma Tre for their kind hospitality.

\end{document}